\documentclass[aap]{imsart}
\pdfoutput=1
\RequirePackage{amsmath,amsthm,amsfonts,amssymb}
\usepackage{cite}
\RequirePackage[colorlinks,citecolor=blue,urlcolor=blue]{hyperref}
\RequirePackage{graphicx}
\graphicspath{{../pics/}{./}{../}{pics/}}

\usepackage{amsopn} 
\usepackage{mathtools}
\usepackage{microtype} 
\usepackage{makecell} 
\usepackage{enumerate} %(i) enumeration
\usepackage{booktabs}
\usepackage{etoolbox}

\usepackage{bbding} % for \SixFlowerAltPetal
\usepackage{comment}
\usepackage[leqno]{cases}
\usepackage{xcolor}
\usepackage{anyfontsize}

\startlocaldefs
\theoremstyle{plain}
\newtheorem{theorem}{Theorem}[section]
\newtheorem{lem}[theorem]{Lemma}
\newtheorem{prop}[theorem]{Proposition}
\newtheorem{coroll}[theorem]{Corollary}
\theoremstyle{remark}
\newtheorem{defi}[theorem]{Definition}
\theoremstyle{definition}
\newtheorem{remark}[theorem]{Remark}
\def\myqedhere{\leavevmode\unskip\penalty9999 \hbox{}\nobreak\hfill\quad\hbox{\ensuremath{\blacksquare}}\gdef\qeddone{1}}
\AtBeginEnvironment{remark}{\def\qeddone{0}}
\AtEndEnvironment{remark}{\ifnum \qeddone=0 \hfill \ensuremath{\blacksquare} \fi}

\newtheorem{example}[theorem]{Example}
\def\myqedhere{\leavevmode\unskip\penalty9999 \hbox{}\nobreak\hfill\quad\hbox{\ensuremath{\blacksquare}}\gdef\qeddone{1}}
\AtBeginEnvironment{example}{\def\qeddone{0}}
\AtEndEnvironment{example}{\ifnum \qeddone=0 \hfill \ensuremath{\blacksquare} \fi}
\newcommand{\myt}[1]{\textit{#1}}

\newcommand{\OEISs}[1]{\href{https://oeis.org/#1}{#1}}
\newcommand{\fallfak}[2]{\ensuremath{#1^{\underline{#2}}}}
\newcommand{\Stir}[2]{\genfrac{ \{ }{ \} }{0pt}{}{#1}{#2}}
\DeclareMathOperator{\Res}{\textrm{Res}}
\newcommand{\ith}[1]{\ensuremath{ {#1}\textsuperscript{th}}}

\newcommand{\C}{\ensuremath{\mathbb{C}}}
\newcommand{\N}{\ensuremath{\mathbb{N}}}
\newcommand{\R}{\ensuremath{\mathbb{R}}}
\newcommand{\ZZ}{\ensuremath{\mathbb{Z}}}
\newcommand{\Sc}{\ensuremath{\mathcal{S}}}
\newcommand{\F}{\ensuremath{\mathcal{F}}}
\newcommand{\HH}{\ensuremath{\mathcal{H}}} 
\newcommand{\G}{\ensuremath{\mathcal{G}}}
\newcommand{\Z}{\ensuremath{\mathcal{Z}}}
\newcommand{\Ce}{\ensuremath{\mathcal{C}}}

\renewcommand{\S}{\ensuremath{{S}}}
\def\P{{\mathbb {P}}}
\def\E{{\mathbb {E}}}
\def\V{{\mathbb {V}}}

\def\mybeta{r} %for the stable laws S_alpha,beta -> S_alpha^{-r}
\def\myalpha{a} %for example 7.2 (urns)
\def\myybeta{b} %for example 7.2 (urns)

\newcommand{\ML}{\operatorname{ML}}
\newcommand{\GML}{\operatorname{ML}}
\newcommand{\BML}{\operatorname{ML}}
\DeclareMathOperator{\cov}{\operatorname{Cov}}
\newcommand{\Seq}{\text{\textsc{Seq}}}
\newcommand{\Cyc}{\text{\textsc{Cyc}}}

\DeclareMathOperator{\MPo}{\text{MPo}} 
\def\Dir{\operatorname{Dir}}
\def\tilt{\operatorname{tilt}}
\DeclareMathOperator{\sgn}{\text{sgn}}

\newcommand{\law}{\ensuremath{\stackrel{d}=}}
\newcommand{\claw}{\ensuremath{\xrightarrow{d}}}

\DeclareMathOperator*{\myarrow}{\xrightarrow{~~~~}}
\newcommand{\cmom}{\displaystyle \ensuremath{\myarrow}^{\smash{\raisebox{-3pt}{\footnotesize \textit{d\,\,}}}}_{\smash{\raisebox{2.2pt}{\textit{\footnotesize{m\,}}}}}}
\newcommand{\eq}{\vcentcolon=} % for :=

\newcommand{\lambdaM}{\lambda_M^{\mkern-1.mu\text{\bf{--}}}} 
\newcommand{\mppar}{\xi}

\renewcommand{\O}{\ensuremath{\mathcal O}}

\endlocaldefs

\begin{document}
\begin{frontmatter}
\title{Phase transitions of composition schemes:\texorpdfstring{\\}{} Mittag-Leffler and mixed Poisson distributions} 
\runtitle{Phase transitions of composition schemes}
\runauthor{Cyril Banderier, Markus Kuba, Michael Wallner}
\begin{aug}
\author[A]{\fnms{Cyril} \snm{Banderier}
%\ead[label=e1]{https://lipn.fr/\textasciitilde{}banderier/}
\orcid{0000-0003-0755-3022}},
\author[B]{\fnms{Markus} \snm{Kuba}
%\ead[label=e2]{https://dmg.tuwien.ac.at/kuba/}
\orcid{0000-0001-7188-6601}}
\and
\author[C]{\fnms{Michael} \snm{Wallner}
%\ead[label=e3]{https://dmg.tuwien.ac.at/mwallner/}
\orcid{0000-0001-8581-449X}}

\address[A]{CNRS \& Universit\'e Sorbonne Paris Nord, %\quad \printead{e1}
\url{https://lipn.fr/~banderier/}
}

\address[B]{University of Applied Sciences Technikum Wien, %\quad \printead{e2}
\url{https://dmg.tuwien.ac.at/kuba/}
}

\address[C]{Technische Universit\"at Wien,  %\quad \printead{e3}
\url{https://dmg.tuwien.ac.at/mwallner/}
}
\end{aug}

\begin{abstract}
% \footnotesize
Multitudinous probabilistic and combinatorial objects are associated with generating functions satisfying a composition scheme $F(z)=G(H(z))$.
The analysis becomes challenging when this scheme is critical (i.e., $G$ and~$H$ are simultaneously singular). 
Motivated by many examples (random mappings, planar maps, directed lattice paths), we consider a natural extension of this scheme, namely $F(z,u)=G(u H(z))M(z)$.
We also consider a variant of this scheme, which allows us to analyse 
the number of $H$-components of a given size in~$F$.

We prove that these two models lead to a rich world of limit laws, where~we identify the key rôle played by a new universal law introduced in this article:
 the three-parameter Mittag-Leffler distribution, which is essentially the product of a beta and a Mittag-Leffler distribution.
We also prove (double) phase transitions, additionally 
involving Boltzmann and mixed Poisson distributions, bringing a unified explanation of the associated thresholds. 
In all cases we obtain moment convergence and local limit theorems.
We end with extensions of the critical composition scheme to a cycle scheme and to the multivariate case, 
leading to product distributions. 
Applications are presented for random walks, trees (supertrees of trees, increasingly labelled trees, preferential attachment trees),
triangular P\'olya urns, and the Chinese restaurant process.
\end{abstract}

%60C05: Combinatorial probability
%60E07: Infinitely divisible distributions; stable distributions
%60G50: Sums of independent random variables; random walks
%05A15: Exact enumeration problems, generating functions

\end{frontmatter}
\dedicated{
 \begin{center}
 %\vspace{-8mm}
 \noindent \emph{{\tiny \SixFlowerAltPetal\SixFlowerAltPetal}\ \small{This article is kindly devoted to Alois Panholzer, on the occasion of his 50th birthday.}\ \tiny{\SixFlowerAltPetal\SixFlowerAltPetal}}%8th March 1971 :-) 
 \end{center}
}

\thispagestyle{empty}     
\bgroup  \footnotesize
\tableofcontents    %The article is quite long, so we want the table of contents.
\egroup

\section{Introduction\label{SectionIntro}}
Many combinatorial structures are an assemblage of more basic building blocks, 
and this situation is ubiquitous in many different fields, such as combinatorics, probability theory, and statistical mechanics.
It appears for example in permutations, random walks, random mappings, random forests, parking functions, P\'olya--Eggenberger urn models, 
(Bienaymé)--Galton--Watson processes, destruction procedures in simply generated trees, 
inversions in labelled tree families, generalized plane-oriented recursive trees (scale-free trees), 
set or integer partitions with some constraints, sequences of words, tilings, 
different families of graphs or maps, etc.; see, e.g.,~\cite{BFSS2001,BaFla2002,BanderierHitcenko2012,BaWa2014,BaWa2017,BaWa2018,DS95,DS97,FlaDumPuy2006,FlaSe2009,FS90,FS93,GoldschmidtHaasSenizergues2020,KuPa2014,Pan2003,Pan2006,PanholzerSeitz2011,Puy2005,Wallner2020}.
 In the language of generating functions, one then has a functional composition scheme such as 
\begin{equation*}
F(z)=G\big(H(z)\big).
\end{equation*}

Let us illustrate this composition scheme with some examples 
(each of them being in fact the starting point of many theorems in the literature): a random forest is a set of random trees,
a permutation is a set of cycles,
a bridge (a random walk on~$\mathbb Z$) is a sequence of arches,
functional mappings are sets of cycles of Cayley trees,
supertrees are trees in which each leaf is replaced by another family of trees,
an integer partition is a sequence of parts,
the factorization of a polynomial in a finite field is a multiset of irreducible factors,
and, following the work of Tutte, 
several important families of planar maps can be seen as a ``simple core'' in which each node is replaced by some ``simple map'', etc.
The reader can find many other examples 
illustrating the universality of the scheme $F(z)=G\big(H(z)\big)$
in the wonderful book by Flajolet and Sedgewick on analytic combinatorics~\cite[Sec.~VI.9]{FlaSe2009}. 
Structurally, this composition scheme is at the heart of many fascinating phase transition phenomena (analytically corresponding, e.g., to coalescing saddle points or to confluence of singularities). 
More precisely, let $G(z)=\sum_{n \geq 0} g_n z^n$ and $H(z)=\sum_{n \geq 0} h_n z^n$ be analytic functions at the origin with nonnegative coefficients and $H(0)=0$.
Let $\rho_G$ and $\rho_H$ be the radii of convergence of $G(z)$ and $H(z)$, respectively.
Then, following~\cite{BFSS2001,FlaSe2009},
we focus on \myt{critical composition schemes}.
\begin{defi}[Critical composition scheme]
	The composition scheme 
	$F(z)=G\big(H(z)\big)$ is \emph{critical} if it satisfies $H(\rho_H)=\rho_G$.
	\label{def:critical}	 
\end{defi}
In other words, $G$ and $H$ are  here concomitantly singular.	
We will assume throughout this work that we are always in the critical case (the asymptotic analysis is straightforward otherwise).
Note that this terminology is a generalization of the notion of critical/supercritical/subcritical Galton--Watson processes,
initially popularized by Harris for neutron branching processes~\cite{Harris1963}. %[Chapter IV.5]

 Often, $G(z)$ and $H(z)$ are the counting series of certain combinatorial families $\mathcal{G}$ and $\mathcal{H}$ such that 
$\mathcal{F}=\mathcal{G}(\mathcal{H})$.
We refer to the first part of~\cite{FlaSe2009}
for a more detailed presentation of this combinatorial approach, starting from the so-called atoms, and then assembling them into more elaborate structured blocks via combinatorial constructors.
Some important subclasses of such structures 
were also subject of more probabilistic approaches; see, e.g.,~\cite{ArratiaBarbourTavare2003,Kolchin1999,GoldschmidtHaasSenizergues2020,AddarioBerry2019}. 

Now, our goal is to analyse 
probabilistic properties of critical compositions like
\begin{equation*}
F(z,u)=G\big(u H(z)\big),
\end{equation*}
where $u$ marks each occurrence of objects of $\mathcal{H}$.
From a combinatorial perspective,\linebreak $F(z,u)$ enumerates $\mathcal{F}$-structures of size $n$ made of $k$ ``building blocks from~$\HH$''
(also simply called $\HH$-components), i.e., $\G$-structures made of~$k$ atoms, where each atom is then replaced by an $\HH$-block
(which is itself a structured set of atoms).
For any combinatorial structure in~$\mathcal{F}$, its corresponding $\G$-structure is sometimes 
called its \myt{core}\footnote{The word ``core'' comes from the theory of graphs and maps, 
 where this composition scheme is natural and was, e.g., analysed 
by Janson, Knuth, {\L}uczak, and Pittel~\cite{JansonKnuthLuczakPittel1993} for graphs
and by Banderier, Flajolet, Schaeffer, and Soria~\cite{BFSS2001} for maps.}
 (or ``skeleton'', or ``backbone'').

A first natural question is what is the typical size of this core, i.e., what is the \textit{typical number of $\HH$-components}?
Such insight helps, for example, to make many algorithms on combinatorial structures more efficient, as Knuth shows in~\cite{Knuth1973} (see also~\cite{BFSS2001}, where this insight is used to design faster random generation algorithms).
To answer this question, 
one considers the discrete random variable $X_n$ associated with this core size in a uniformly chosen object of size $n$. 
Its probability mass function is obtained by extraction of coefficients:\nolinebreak
\begin{equation*}
\P\{X_n=k\}= \frac{[z^n u^k]F(z,u)}{[z^n]F(z,1)}=g_k\frac{[z^n]H(z)^k}{f_n},
\end{equation*}
where $[z^n]$ denotes the extraction of coefficient operator $[z^n]\sum_{n} f_n z^n = f_n$. 
As $H(z)$ has typically a singular expansion of the type 
\begin{align*}
H(z) &= \tau_H + c_{H}\Big(1-\frac{z}{\rho_H}\Big)^{\lambda_H} + \dots,
\end{align*}
this implies that the asymptotic behaviour of $\P\{X_n=k\}$ depends 
on the exponent~$\lambda_H$ (which is called the \textit{singular exponent} of $H$).
Actually, this exponent even plays a key r{\^o}le, as it entails four types of asymptotic behaviour for $X_n$:
\begin{itemize}
\item For $\lambda_H>2$ the limit law is related to a Gaussian law.
\item For $1<\lambda_H<2$ the limit law is related to a stable law of parameter $\lambda_H$
(this distribution is supported on~$\R$ and possesses moments up to order $\lambda_H$; e.g., for $\lambda_H=3/2$ this gives the map-Airy distribution).
\item For $0<\lambda_H<1$ we will show that the limit law is related to 
a stable law of parameter $\lambda_H$, or more precisely to a generalized Mittag-Leffler distribution (this distribution is supported on $\R^+$ and has moments of any order; e.g., for $\lambda=1/2$, this gives the Rayleigh distribution). 
\item For $\lambda_H<0$ the scheme is not critical because the function $H(z)$ diverges at $z=\rho_H$, 
and thus leads to a singularity of $G(H(z))$ at some $z<\rho_H$. Such a scheme is called supercritical and typically 
leads to a Gaussian limit law.
\end{itemize}

The case $\lambda_H<0$ is analysed by Flajolet and Sedgewick~\cite{FlaSe2009}\footnote{We refer to \textit{The Hitchhiker's Guide to the Galaxy}, by Douglas Adams, for an interesting epistemology of~\cite{FlaSe2009}.}, building on the seminal work of Bender~\cite{Bender1973}.
The three cases $0<\lambda_H<1$, $1<\lambda_H<2$, and $\lambda_H>2$ were partially analysed by Banderier, Flajolet, Soria, and Schaeffer~\cite[Theorem 12]{BFSS2001}, but without a precise statement for the right renormalizations in the limit laws. 
It is partly due to the fact that the initial motivation of the authors of~\cite{BFSS2001} was to analyse the core of planar maps, 
so they focused on the subcase $\lambda_H=\frac32$, which corresponds to the map-Airy distribution.
Thus, a more complete analysis of the composition scheme in these three regions of~$\lambda_H$ 
remained to be done.\pagebreak

For sure, we expected that the different possible analytic behaviours of $G$ introduce further subcases, but 
 we were surprised that the detailed analyses were much more challenging than expected: As we shall see, they require several new ingredients. 
Our identification of the limit laws involves \textit{moment-tilted distributions},
\textit{product distributions}, and \textit{Boltzmann distributions} 
(see Section~\ref{ssMP} for a formal presentation of these three types of distributions and their key properties).
Therefore, the first main objective of our work is to give a complete landscape of the limit laws associated with critical composition schemes.
We analyse the case $0<\lambda_H<1$ in this article, and the cases $1<\lambda_H<2$ and $\lambda_H>2$ in our companion article~\cite{BanderierKubaWallner2021b}.

Our second main objective is to explain the phase transitions 
observed for the \textit{number of ${\mathcal H}$-components of a given size}.
This builds on the work of Panholzer and the second author~\cite{KuPa2014}, 
in which they started to unify the diversity of limit laws encountered in these phase transitions under the umbrella of \textit{mixed Poisson distributions},
and relies on the study of a size-refined composition scheme that we detail in the next section.

\section{New main results}\label{sec2}
\subsection{Composition schemes analysed in this article}

In this work we complete the analysis of critical compositions with exponent $0<\lambda_H<1$.
Motivated by models associated with directed lattice paths~\cite{BaFla2002,BaWa2014,BaWa2017,BaWa2018,Wallner2020} and triangular P\'olya urns~\cite{FlaDumPuy2006,Jan2006,Jan2010},
we consider two natural schemes (Equations \eqref{Eq4} and \eqref{eq:scheme_size-size-refined} hereafter).
They contain a multiplication by a factor~$M(z)$ which plays a non-trivial r\^ole in our identification of the corresponding limit laws as a three-parameter generalization of Mittag-Leffler distributions.
In Section~\ref{sec:examples}, we illustrate via various examples 
how these two schemes unify and refine many previous studies. 

Firstly, we consider the following \myt{extended composition scheme}
\begin{equation} 
F(z)=G\big(H(z)\big)\cdot M(z),\label{eq:scheme_extended}
\end{equation}
for some functions $F$/$G$/$H$/$M$ analytic at the origin, with nonnegative coefficients.
Such schemes are \emph{critical} if $\rho_G=H(\rho_H)$ (like in Definition~\ref{def:critical}) 
and additionally satisfy $\rho_{M} \geq \rho_H$, where $\rho_M$ is the radius of convergence of $M(z)$ (the analysis is straightforward if the extended composition scheme is not critical). 
In order to enumerate the family $\mathcal{F}$ according to the occurrences of $\mathcal{H}$-components, we consider
\begin{align}
F(z,u)=G\big(u H(z)\big)\cdot M(z),\label{Eq4}
\end{align}
which from now on we will refer to as the extended composition scheme.
Equivalently, $[z^n u^k]F(z,u)$ is the number of $\mathcal{F}$-structures of size $n$ having $k$ $\mathcal H$-components. 
This generalizes the schemes $G(uH(z))$ and $A(uz) B(z)$ analysed in~\cite{BFSS2001} and~\cite{FlaSe2009,Gourdon1998}, respectively.
The corresponding random variable $X_n$ has a probability mass function given by
\begin{equation}
\P\{X_n=k\}= \frac{[z^n u^k]F(z,u)}{[z^n]F(z,1)}=g_k \frac{[z^n]H(z)^k\cdot M(z)}{f_n}.
\label{PXnk}
\end{equation}
Our first main result is Theorem~\ref{TheExtended}, in which we give explicit expressions for the asymptotics of the factorial moments of $X_n$, 
the limit distribution of $X_n$ (suitably normalized), and its density function. 
It appears that this limit distribution differs depending on some relationship between the singular exponents of $G(z)$, $H(z)$, and $M(z)$, 
as summarized in Table~\ref{tab:extended} (where we write $X_n\sim c_n\cdot \mathcal{D}$ when $c_n^{-1}\cdot X_n\to \mathcal{D}$ in distribution for $n\to \infty$). 
In addition to these convergences in distribution, we prove moment convergence for the continuous limit laws.

	{\newcommand{\myw}{3cm}
	\newcommand{\mywb}{3.1cm}
	\begin{table}[ht!] \begin{minipage}[c]{\linewidth}
	\centering%
	\begin{tabular}{@{}lccc@{}}
		\toprule
		\emph{Singular} & 
		$ \lambda_M > \lambda_G \lambda_H$ & 
		$ \lambda_M < \lambda_G \lambda_H $ &
		$ \lambda_M = \lambda_G \lambda_H $ \\
	\emph{exponent} & 
		(pure scheme) &
		(degenerate scheme) &
		(confluent scheme) \\[.5mm]
		\midrule 
		\emph{Limit law \hspace{-1cm}} 
 & 
		continuous &
		discrete &
		linear combination \\
 & 
		(gen.\ Mittag-Leffler $\ML$) &
		(Boltzmann $\mathcal B$) &
		 (${\ML}+{\mathcal B}$) \\[2mm]		
		\emph{Example} & 
		\begin{tabular}{c}\includegraphics[width=\myw]{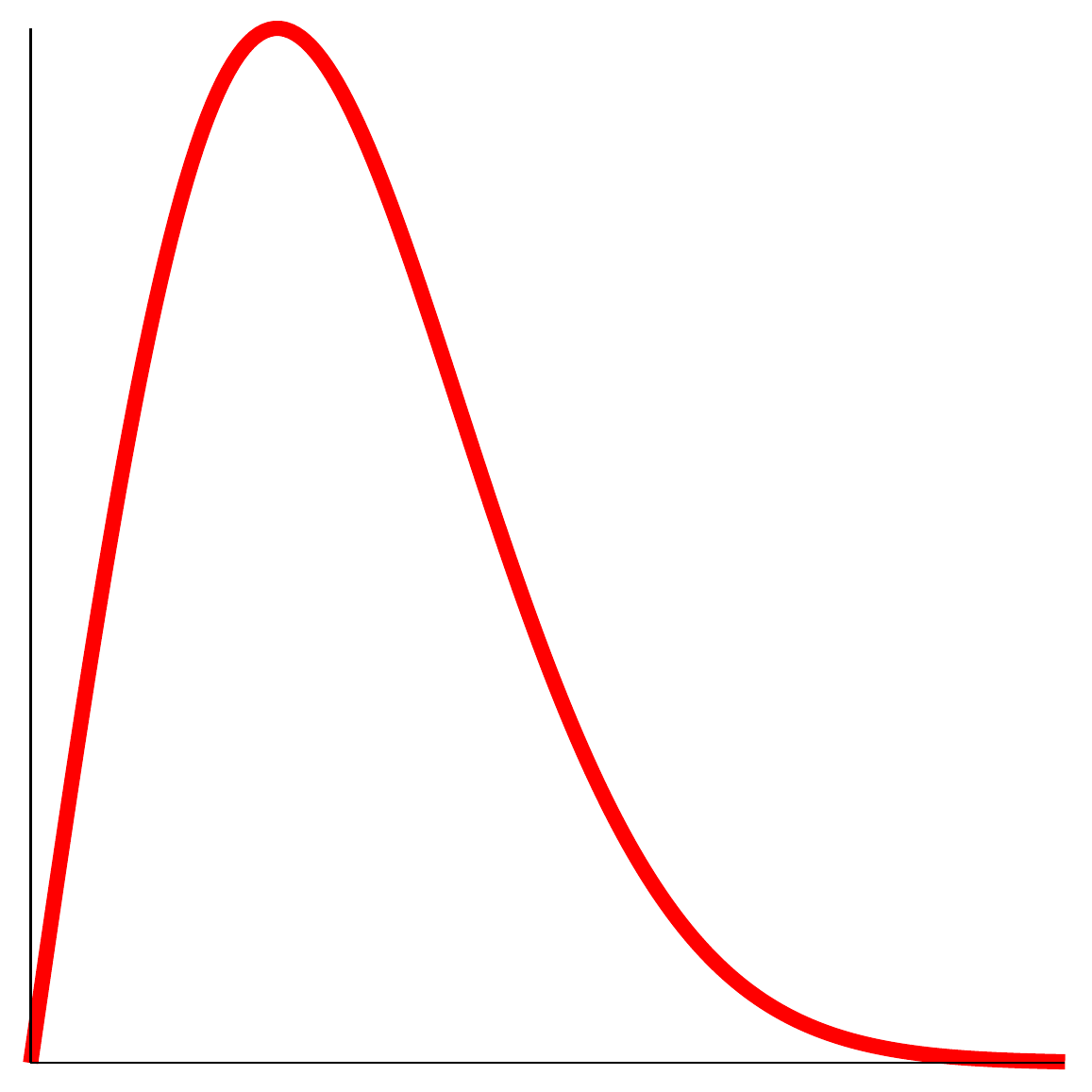}\end{tabular} & 
		\begin{tabular}{c}\includegraphics[width=\myw]{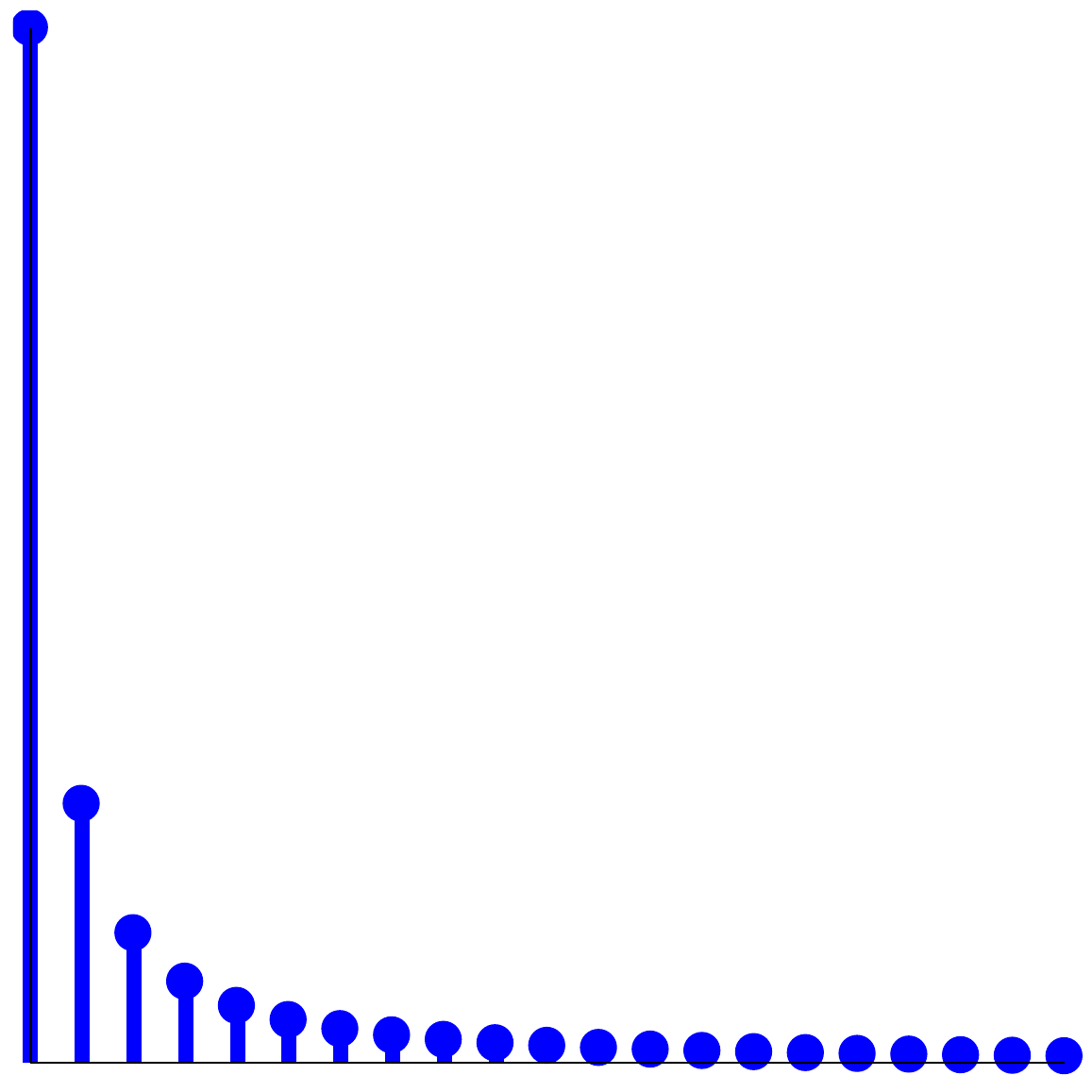}\end{tabular} &
		\begin{tabular}{c}\includegraphics[width=\mywb,trim={8mm 22mm 0 0},clip]{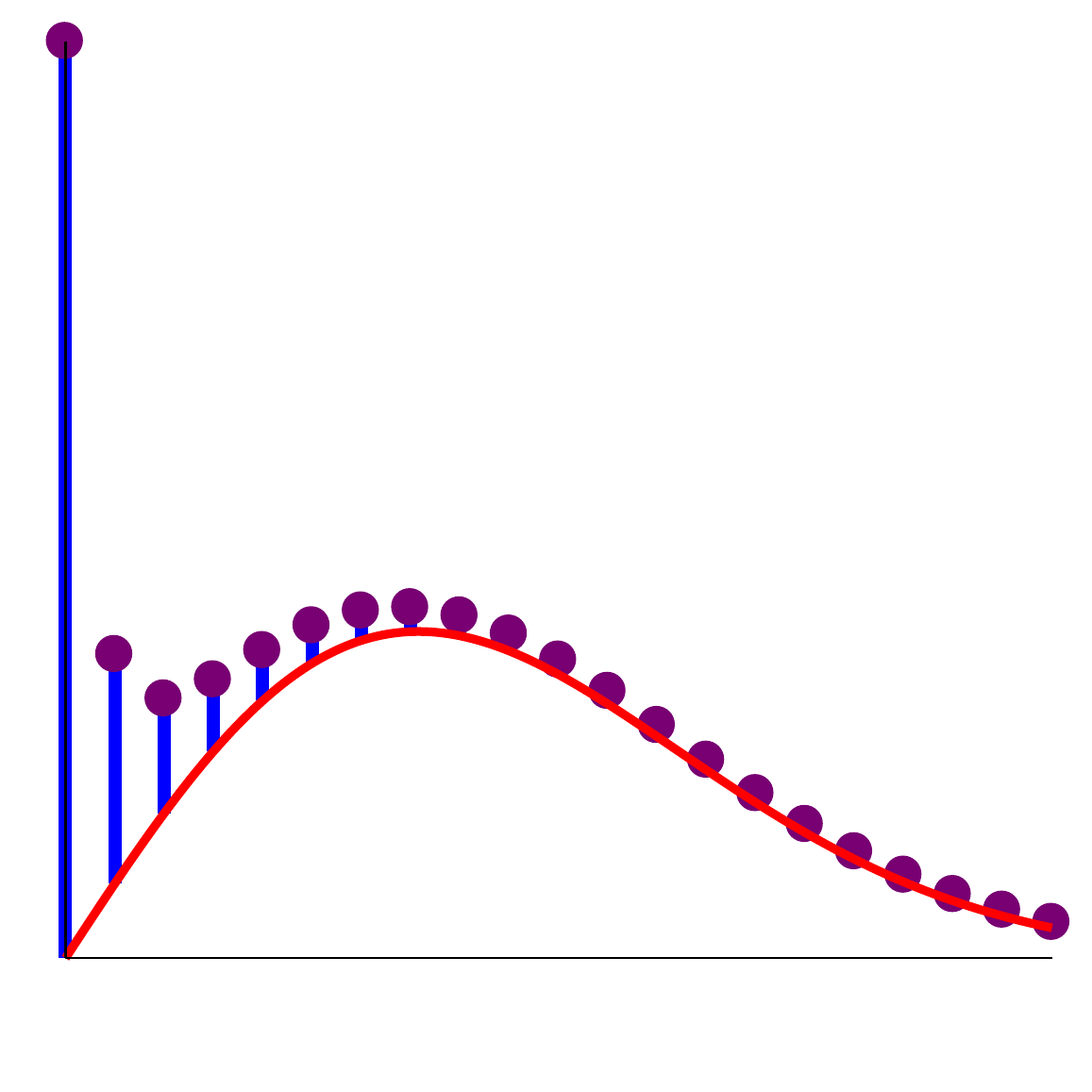}\end{tabular} \\ 
		 & 
		$X_n \sim C n^{\lambda_H} {\ML} $ %We don't comment on C, to keep the flux simple. The readers will infer that it is just a constant.
 &
		$\P\{X_n=k\} \sim \frac{g_k \rho_G^k}{G(\rho_G)} $ &
		$X_n \sim \text{LinComb}(n^{\lambda_H} \ML, \mathcal{B} )$ \\
		\bottomrule
	\end{tabular}
	\caption{Three asymptotic behaviours according to the singular exponents $\lambda_G/\lambda_H/\lambda_M$ of $G/H/M$: the number $X_n$ of $\HH$-components
	in $\F$-structures of size $n$ in the extended scheme~\eqref{eq:scheme_extended} 
is given by three completely different types of limit laws, depending on whether the scheme is analytically pure/confluent/degenerate (Theorems \ref{TheExtended}, \ref{TheExtendedDegenerate}, and \ref{TheExtendedConfluent}).}
	\label{tab:extended}
 \end{minipage}
	\end{table}
	}
\pagebreak

Secondly, we consider a \myt{size-refined composition scheme} which allows us to capture some threshold phenomenon 
via a bivariate generating function $F(z,v)$\footnote{Note that we still use the letter $F$ to denote the main function. 
The auxiliary variable is now~$v$ and no more~$u$ as we are marking a different parameter. 
This choice avoids more cumbersome notation and is also motivated by the fact that we always have $F(z,1)=F(z)$.}:
\begin{equation}
\label{eq:scheme_size-size-refined}
F(z,v)=G\big(H(z) - (1-v)h_j z^j\big)\cdot M(z),\quad j\in\N.
\end{equation}
In this scheme, $[z^n v^k]F(z,v)$ is therefore the number of $\mathcal{F}$-structures of size $n$ 
having $k$ $\mathcal{H}$-components each of size $j$; this is combinatorially summarized by:
\begin{equation*}
\mathcal{F}=\mathcal{G}\big(\mathcal{H}_{\neq j} + v\mathcal{H}_{= j}\big)\times \mathcal{M}.
\end{equation*}
Assuming the uniform distribution amongst structures of size $n$, let $X_{n,j}$ be the number 
of $\HH$-components of size $j$ inside $\mathcal{F}$-structures\footnote{Flajolet and Sedgewick call the corresponding distribution the \textit{profile} of the combinatorial object, by analogy with the profile of integer compositions; see~\cite[p.~169, p.~451, p.~632]{FlaSe2009}.} of size $n$. 
Thus,~the\- random variables $X_{n,j}$ naturally refine the distribution 
of the core size~$X_n$ given in~\eqref{PXnk}, since
\begin{equation*}
\sum_{j\in\N} X_{n,j}=X_n.\label{eq:rvssize-size-refined}
\end{equation*}
Formula~\eqref{eq:scheme_size-size-refined} implies that one has
\begin{equation}
\P\{X_{n,j}=k\}= \frac{[z^n v^k]F(z,v)}{[z^n]F(z,1)}
=\frac{h_j^k}{k!}\frac{[z^{n-j k}]G^{(k)}\big(H(z)-h_j z^j\big)M(z)}{f_n}.
\label{Xnj_proba}
\end{equation}
Our second main result is Theorem~\ref{TheRefined}, 
in which we prove that the factorial moments of $X_{n,j}$ are asymptotically of mixed Poisson type, and establish a convergence in distribution,
with convergence of all moments, towards explicit limit distributions.

We extend these two main results to the composition schemes involving a logarithmic singularity 
in Theorem~\ref{COMPSCHEMECycleThe1} and~\ref{COMPSCHEMECycleThe2}, leading to Mittag-Leffler distributions. 

Then, our third main result is Theorem~\ref{TheMV} in which we give a multivariate generalization 
with arbitrary many variables, leading to Dirichlet product distributions. 

\subsection{Phase transitions}

Analogously to what can happen in physics or chemistry for some small change of temperature, pressure, or concentration,
a phase transition in mathematics corresponds to a sudden non-smooth change of properties under smooth variation of the parameters.
Such non-smooth changes are thus analytically reflected by a singularity of some function associated with these properties.
For combinatorial structures, generating function methods were successfully used to analyse such phase transitions,
 e.g., for random graphs or planar maps~\cite{JansonKnuthLuczakPittel1993, GimenezNoy2009, GimenezNoyRue2013,BFSS2001}, 
for satisfiability problems~\cite{Monasson2005,Puy2005}, 
and for many other problems~\cite{Fuchs2008,Fuchs2012,Hwang1999,Hwang2004,HwangZacharovas2011,FlaSe2009}. 

Usually, the main problems are to locate the phase transition, 
to properly describe the phase transition via special functions in terms of the involved parameters, 
and to give intuitive explanations of the observed phenomena. 
Some cases even exhibit two successive phase transitions. In probability theory, 
such a double phase transition occurs for example with the binomial distribution $\P\{X_n=k\} =\binom{n}{k} p^k (1-p)^{n-k}$ when $p$ depends on $n$, with $p$ bounded away from one. It indeed leads to the following trichotomy involving a continuous to discrete to degenerate phase transition: 
Firstly, when both $\E(X_n)=p(n) \cdot n$ and the variance $\V(X_n)$ tend to infinity, then a central limit theorem for $X_n$ follows. Secondly, if $p(n) \cdot n\to \lambda>0$, then $X_n$ is asymptotically Poisson distributed. Thirdly, if $p(n) \cdot n\to 0$, then $X_n$ degenerates to a Dirac distribution with all mass at $0$. 

The analysis of the size-refined composition scheme unravels, though more subtly, such a double phase transition:
the random variable $X_{n,j}$ given by Equation~\eqref{Xnj_proba} has three phases, each leading to its own limit law, visualized in Table~\ref{tab:size-refinedphases}.

	{\newcommand{\myw}{3.5cm}
	\begin{table}[hbt!]
	\begin{tabular}{@{}lccc@{}}
		\toprule
		\emph{Scale} & 
		$j \ll n^{\frac{\lambda_H}{1+\lambda_H}}$ &
		$j = \Theta\Big(n^{\frac{\lambda_H}{1+\lambda_H}}\Big)$ &
		$j \gg n^{\frac{\lambda_H}{1+\lambda_H}}$ \\
		\midrule 
		\emph{Limit law} 
	& 	continuous &
		discrete &
		discrete \\
 & (gen.\ Mittag-Leffler $\ML$) & (mixed Poisson $\MPo(\mppar \ML)$) & (Dirac) \\[2mm]
\emph{Example} & 
		\begin{tabular}{c}\includegraphics[width=\myw]{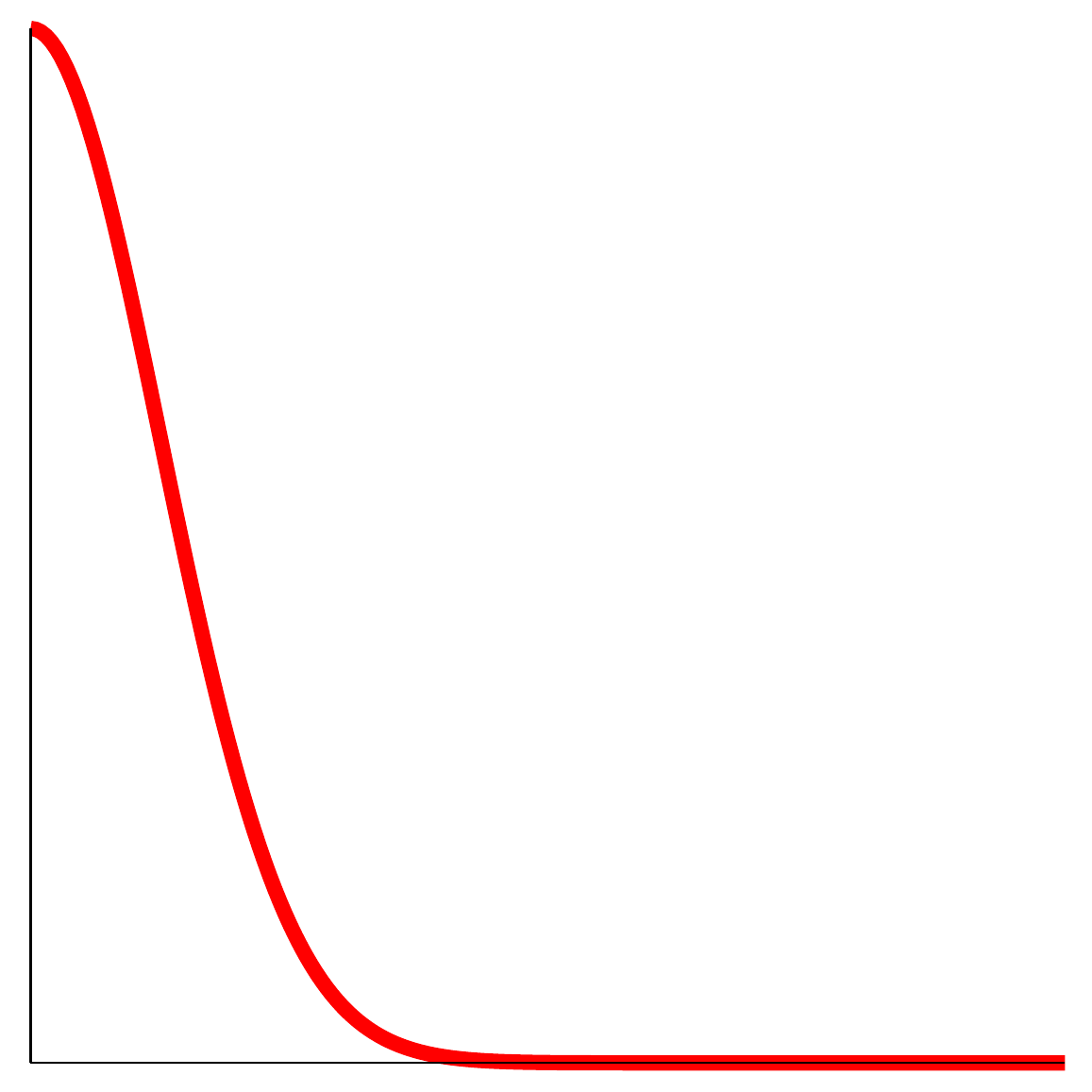}\end{tabular} & 
		\begin{tabular}{c}\includegraphics[width=\myw]{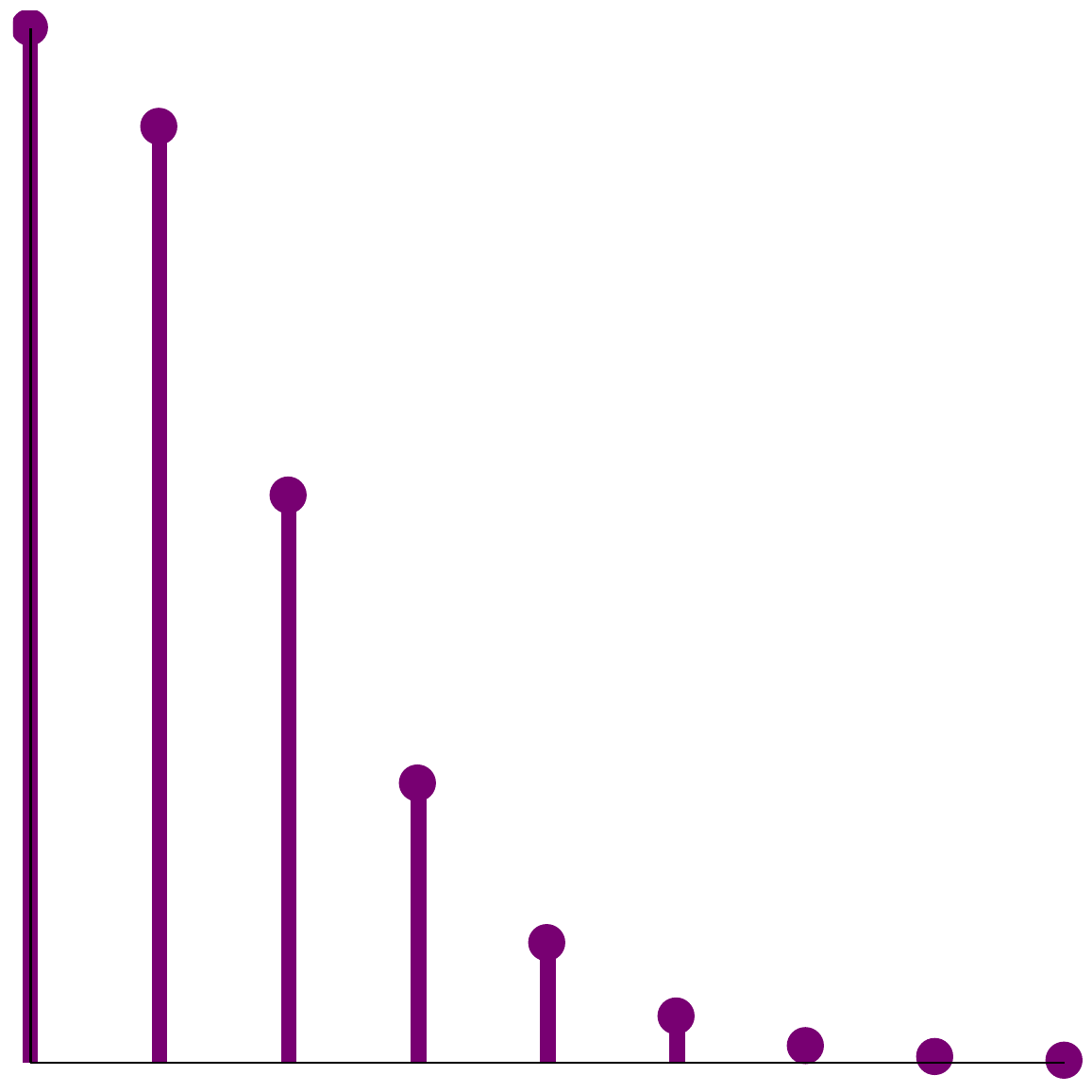}\end{tabular} & 
		\begin{tabular}{c}\includegraphics[width=\myw]{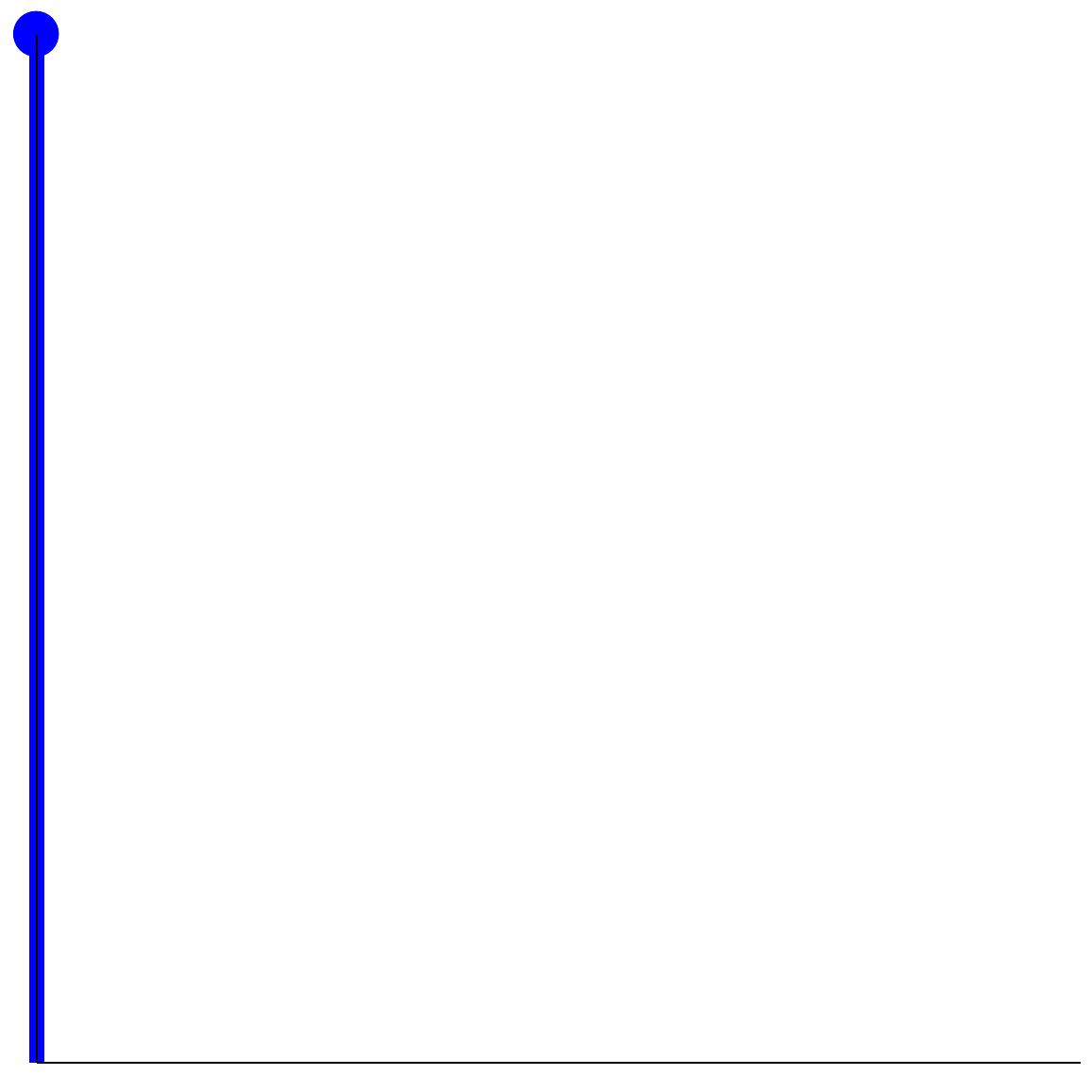}\end{tabular} \\
& $X_{n,j} \sim C \, h_j \rho_H^j n^{\lambda_H} \ML$ 
& 
$X_{n,j} \sim \MPo(\mppar \ML)$ 
& $\P\{X_{n,j} \geq 1\} \sim 0$ \\ 
		\bottomrule
	\end{tabular}
	\caption{Three successive régimes for the number $X_{n,j}$ of $\HH$-components of size $j$
	in $\F$-structures\\ of size $n$ in the critical size-refined scheme~\eqref{eq:scheme_size-size-refined}, depending on the relation between $j$ and~$n$; \\
	this double phase transition is proven in Theorem~\ref{TheRefined}.}
	\label{tab:size-refinedphases}
	\end{table}
	}

Table~\ref{tab:size-refinedphases} also motivates the following important remark.

\begin{remark}[Ubiquity of the exponent 1/3]
Generating functions often have a dominant singularity of the square-root type (i.e., $\lambda_H=1/2$).
This phenomenon is explained by the Drmota--Lalley--Woods theorem: Whenever $H(z)$ can be defined by a strongly connected set of polynomial equations with nonnegative coefficients,
it has a singular exponent $1/2$ (see, e.g., \cite{FlaSe2009,Drmota2009,BanderierDrmota2015}). 
Accordingly, in conjunction with Table~\ref{tab:size-refinedphases}, this explains why one observes a threshold at $j=n^{1/3}$ 
in many phase transitions; see Section~\ref{sec:examples}.
\end{remark}
\pagebreak

One pleasant consequence of our work is that 
it gives a unified explanation of phase transitions from continuous to discrete observed in many examples:
descendants in increasing trees~\cite{Desc-KubPan2005},
node degrees in increasing trees~\cite{KuPa2007}, 
block sizes in $k$-Stirling permutations~\cite{PanKuCPC}, 
stopping times in urn models~\cite{PanKuAdvances}, 
death processes~\cite{PanKu2012}, 
inversions in labelled tree families~\cite{PanholzerSeitz2011},
ancestors and descendants in evolving $k$-tree models~\cite{PanholzerSeitz2012}.

These case by case studies lacked a proper comprehensive and uniform description of the arising phase transitions.
So, instead of treating these combinatorial structures individually,
we directly study the size-refined composition scheme~\eqref{eq:scheme_size-size-refined}. 
As summarized in Table~\ref{tab:size-refinedphases}, we show 
how the phase transitions for the random variable~$X_{n,j}$ depend on the growth of $j=j(n)$ with respect to the size~$n$.
We prove that the distribution of $X_{n,j(n)}$ is\linebreak \textit{continuous} for small values of $j$ (a three-parameter generalization of the Mittag-Leffler distribution),
or \textit{discrete} for some threshold values of $j$ (a Poisson distribution mixed with the previous Mittag-Leffler distribution), 
or a \textit{Dirac distribution} for large values of $j$.
We further exemplify these results on different processes, like the Chinese restaurant process, 
sign changes and returns to zero in random walks, and the branching structure of random trees. 

\subsection{Plan of the paper}

In Section~\ref{SectionPrelim} we collect results from analytic combinatorics. 
We also present our basic assumptions on the generalized composition scheme and collect properties of various distributions that appear later in our main results.
Section~\ref{SecExtended} is devoted to our results on the 
random variable $X_n$ corresponding to the extended composition scheme~\eqref{Eq4} involving $F(z,u)$.
Section~\ref{SecRefined} contains the results for the random variable $X_{n,j}$ 
corresponding to the size-refined composition scheme~\eqref{eq:scheme_size-size-refined} involving $F(z,v)$ and exhibiting phase transitions. 
We also give the covariance and the correlation coefficient of $X_{n,j_1}$, $X_{n,j_2}$, observing again some phase transitions.
In Section~\ref{sec:examples} we discuss various examples to which we apply our results. 
Finally, in Section~\ref{sec:outlook}, we analyse further extensions for a cycle scheme and for a multivariate critical composition scheme. 
We also present new examples for these two extensions.

\section{Singularity analysis, stable laws, and mixed Poisson distributions\label{SectionPrelim}}

In this section, we first present a few important notions from analytic combinatorics~\cite{FlaSe2009}
which we use to identify the radius of convergence and the singular exponents in our composition schemes. 
Then, we present a few results on the family of moment-tilted stable laws, 
based on James~\cite{James2010,James2013,James2015} and Janson~\cite{Jan2010}. 
We also collect properties of mixed Poisson distributions and their factorial moments. 
All of this allows us to identify in Sections~\ref{SecExtended} and~\ref{SecRefined}
the distribution of the $\HH$-components in our composition schemes.

\begin{figure}[h]			
\begin{center}
\includegraphics[width=.62\textwidth]{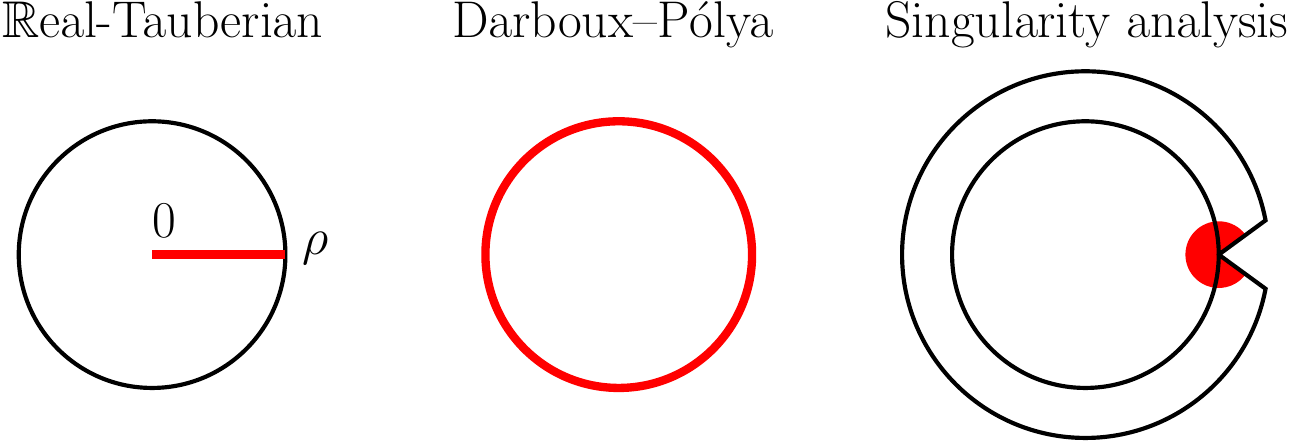} 
\end{center}
\caption{As visually summarized by Flajolet in~\cite{Flajolet2006},
three fundamental methods of asymptotic analysis require information on the function in different parts (shown here in red) of the complex plane.
Flajolet and Odlyzko's singularity analysis~\cite{FlaOd1990} offers more powerful results, but requires analyticity 
in a $\Delta$-domain (tastefully also sometimes called ``Camembert domain'' or ``Pac-Man'' by Flajolet himself!). 
This is the domain inside the blue curve, defined by $\Delta= \{z\in \C \text{ such that } |z|<\rho+\varepsilon \text{ and } \operatorname{arg}(z-\rho)>\theta\}$, 
for some $\varepsilon>0$ and $0<\theta<\pi/2$. 
This analyticity is agreeably typically granted for most combinatorial constructions 
(e.g., for the ones leading to meromorphic, algebraic-logarithmic, hypergeometric, or D-finite functions).}
\label{fig:deltadomain}
\end{figure}

\subsection{Singularity analysis and asymptotic expansions}
\label{sec:prelimsingular}
Let 
\begin{equation*}
F(z) = \sum_{n \geq 0} f_n z^n
\end{equation*}
 be a
function with nonnegative coefficients $f_n$ that is
analytic in a $\Delta$-domain (see Figure~\ref{fig:deltadomain} for this notion)
with a finite radius of convergence~$\rho$ and singular expansion
\begin{equation}
F(z)= P\left(1-\frac{z}{\rho}\right) + c_F \cdot \left(1-\frac{z}{\rho}\right)^{\lambda_F} \left(1+o(1)\right), 
\label{EqSA}
\end{equation}
where $\lambda_{F}\in \R\backslash \{0,1,2,\dots\}$ is called the \emph{singular exponent} (of $F(z)$ at $z=\rho$),
and where $P(x) \in \C[x]$ is a polynomial (of degree $\geq 1$ for $\lambda_F>1$, 
of degree $0$ for $0<\lambda_F<1$, 
and $P=0$ for $\lambda_F<0$).
Then, by standard singularity analysis using the transfer theorem~\cite[Sec.~2]{FlaOd1990}, 
if $\rho$ is the unique singularity of $F(z)$ in $|z|\leq\rho$, 
the Taylor series coefficients of $F(z)$ satisfy 
\begin{equation}
[z^n]F(z) = f_n =\frac{c_F}{\rho^n}\cdot \frac{n^{-\lambda_{F}-1}}{\Gamma(-\lambda_{F})}\cdot \left(1+o(1)\right).
\label{EqSA1}
\end{equation}

As the $f_n$'s are nonnegative, this implies the sign property 
\begin{equation}\label{sign_property}
\sgn(c_F)=\sgn(\Gamma(-\lambda_{F})),
\end{equation}
i.e., due to the sign change of the gamma function at each negative integer we have
$c_F>0$ for $\lambda_F<0$,  $c_F<0$ for $0<\lambda_F<1$, and 
$\sgn(c_F)= (-1)^{\lceil \lambda_F \rceil}$ for $\lambda_F>1$.

\smallskip

Note that it is in general easy to get more asymptotic terms in~\eqref{EqSA},
and that singularity analysis directly translates them into more asymptotic terms in~\eqref{EqSA1}. 
Moreover, if one has several dominant singularities,
one has to sum the contributions of the local expansions at each singularity to get the asymptotics of the coefficients (for more details see~\cite[Chapter~IV.6]{FlaSe2009}; see also the rotation law in~\cite{BaWa2019}
for walks or trees with periodic offspring%
%, which involve multiple dominant singularities as soon as their offspring distribution has a periodic support
). 

Note that algebraic functions constitute one of the main 
sources of functions satisfying the conditions of the expansion~\eqref{EqSA}; indeed, they admit a Puiseux expansion 
\[ \textstyle
	F(z) = \sum_{k \geq k_0} c_k \cdot \left(1-z/\rho\right)^{k/r},
\]
for $k_0, r \in \ZZ$ with $r\geq1$. 
For example, in Section~\ref{sec:Supertrees}~we~will encounter the Catalan generating function $1/2 - \sqrt{1-4z}/2$, for which one has $\rho=1/4$, $P(x)=1/2$, 
$k_0=0$, and $r=2$.

This Catalan example is pleasantly simple, and thus obviously not generic, as its Puiseux expansion contains only two terms.
In full generality such an expansion can involve an infinite number of terms whose sum is converging. % for $|z|<\rho$.
Let us now analyse these singular expansions.

\begin{lem}[Singular expansion]
	\label{lem:Fgenericasy}
	Let $F(z)$ be a power series with nonnegative coefficients satisfying~\eqref{EqSA}. 
	Then $F(z)$ has the following singular expansion
	\begin{align*}
	\begin{aligned}
		F(z) &= 
			\begin{cases}
				c_F \left( 1-\frac{z}{\rho} \right)^{\lambda_F}(1+o(1)) & \text{ if\ } \lambda_F < 0, \\
				\tau_F + c_F \left( 1-\frac{z}{\rho} \right)^{\lambda_F}(1+o(1)) & \text{ if\ } 0 < \lambda_F < 1, \\
				\tau_F + \sum_{i=1}^{\lfloor \lambda_F \rfloor} p_i \left( 1-\frac{z}{\rho} \right)^i + c_F \left( 1-\frac{z}{\rho} \right)^{\lambda_F}(1+o(1)) & \text{ if\ } \lambda_F>1,
			\end{cases}
	\end{aligned}
	\end{align*}
	where $\tau_F=F(\rho)>0$ for $\lambda_F>0$ and $p_1 = - \rho F'(\rho) < 0$ for $\lambda_F>1$.
\end{lem}
\pagebreak

\begin{proof}
	Firstly, if $\lambda_F<0$, then the lowest order of the Puiseux expansion is $\lambda_F$.
	Secondly, for $\lambda_F > 0$, we rewrite~\eqref{EqSA} into
	\begin{align*}
		F(z) &= \sum_{i=0}^k p_i (1-z/\rho)^i + c_{F} (1-z/\rho)^{\lambda_F} + \dots.
	\end{align*}
	Then, we have $p_0=P(0)=F(\rho)$ which we define to be $\tau_F$. 
	We get $\tau_F>0$ as it is an infinite convergent sum of nonnegative not-all-zero terms.
	Next, observe that the nonnegative coefficients of $F(z)$ imply that $F'(\rho)>0$.
	Thus, taking the derivative in the expansion of $F(z)$ we get $\lim_{z \to \rho} F'(z) = +\infty$ for $0 < \lambda_F < 1$ and 
	$p_1 = - \rho F'(\rho) \neq 0$ for $\lambda_F > 1$.
\end{proof}

We now consider the critical scheme $F(z)=G(H(z)) M(z)$, assuming that 
$G(z)$, $H(z)$, and $M(z)$
have a finite radius of convergence and a unique dominant singularity (i.e., the one of smallest modulus).
By Pringsheim's theorem~\cite[p.~240]{FlaSe2009} applied to each of these functions, the non\-negativity of its coefficients implies 
that this dominant singularity lies on the positive real axis and corresponds therefore to its radius of convergence denoted by $\rho_G$, $\rho_H$, and $\rho_M$. 

Note that if $M(z)$ has an infinite radius of convergence (denoted by $\rho_M=+\infty$) or if $\rho_M \neq \rho_H$,
then the asymptotics are easily obtained via
\begin{subnumcases}{[z^n] F(z) \sim}
 G(H(\rho_M)) [z^n] M(z) &\text{if $\rho_M<\rho_H$,} \label{eq1}\\ 
 M(\rho_H) [z^n] G(H(z)) &\text{if $\rho_M>\rho_H$.} \label{eq2}
\end{subnumcases}

Now, as in Lemma~\ref{lem:Fgenericasy}, we define for each function:
\begin{itemize} 
\item the singular exponents $\lambda_G$, $\lambda_H$, and $\lambda_M$,
\item the constant terms $\tau_G$, $\tau_H$, and $\tau_M$,
\item and the singular coefficients $c_G$, $c_H$, and~$c_M$.
\end{itemize}

Thus, thanks to Equations~\eqref{eq1} and \eqref{eq2}, 
we can now focus (without loss of generality, or rather ``without loss of difficulty!'') 
on the case $\rho_M=\rho_H$ which is more involved as here $G(z)$, $H(z)$, and $M(z)$ are \textit{all contributing} to the asymptotics in a nontrivial way.
Then, one gets different r\'egimes (depending on $\lambda_H$) for the asymptotics of the coefficient of $F(z)$.
In this article we focus on the range $0 < \lambda_H < 1$, while we treat the other range $\lambda_H>1$ in a companion article~\cite{BanderierKubaWallner2021b}.

Note that, as our work extends to some interesting combinatorial cases where $M(z)$ has a radius of convergence $\rho_M>\rho_H$, we will also encompass this case, for which it is then convenient to set \mbox{$\lambda_M=+\infty$} (the singular exponent of $M(z)$ at $z=\rho_H$ is infinite whenever $M$ is analytical there). 
The case $\lambda_M=+\infty$ is thus archetypal of cases 
where $M(z)$ only affects the asymptotics of~$f_n$ by a multiplicative constant like in Equation~\eqref{eq2}.

We can now express the singular exponent of $F$ in terms of those of $G/H/M$.

\begin{lem}
	Let $F(z)=G(H(z))M(z)$ be a critical composition scheme that is singular at~$\rho_H$. 
	Then, the singular exponent $\lambda_H$ of $H(z)$ satisfies $\lambda_H>0$. 
	
\noindent	Moreover, for the range $0 < \lambda_H < 1$, the singular exponent $\lambda_F$ of $F(z)$ satisfies
	\begin{align*}
			\lambda_F &= \min(
				\lambda_G \lambda_H, 
				\lambda_H,
				\lambda_M,
				\lambda_G \lambda_H + \lambda_M
				).
	\end{align*}
	For $\lambda_H > 1$, the singular exponent $\lambda_F$ of $F(z)$ satisfies
	\begin{align*}\lambda_F &= \min(
				\lambda_G, 
				\lambda_H,
				\lambda_M,
				\lambda_G + \lambda_M
				).
	\end{align*}
\end{lem}
\pagebreak
\begin{proof}
The claim $\lambda_H>0$ follows from $H(\rho_H)=\rho_G \in (0,\infty)$ as one would have $H(\rho_H)=+\infty$ if $H(z)$ had 
a singularity of negative singular exponent at $z=\rho_H$.
 
Now, we plug the singular expansions from Lemma~\ref{lem:Fgenericasy} at $z=\rho_H$ for $G(z)$, $H(z)$, and $M(z)$
into $F(z)=G(H(z))M(z)$.
When $0<\lambda_H<1$ we get the following expansions 
(in which we omit the terms not contributing to the first-order asymptotics):
\newcommand{\shrink}[1]{\!#1\!}
 {\small\begin{subnumcases}{\ \ F(z) \shrink{=}}
		c_M c_G \left(\frac{-c_H}{\rho_G}\right)^{\lambda_G}\!\!\left(1-\frac{z}{\rho_H}\right)^{\lambda_{G}\lambda_H + \lambda_M} \!+ \dots\! &
	 if $\lambda_G \shrink{<} 0$, $\lambda_M \shrink{<} 0$, 
		 \label{ExpansionF-a} \\
		 \tau_M c_G \left(\frac{-c_H}{\rho_G}\right)^{\lambda_G}\!\!\left(1-\frac{z}{\rho_H}\right)^{\lambda_{G}\lambda_H}
		\!+ \dots &
		if $\lambda_G \shrink{<} 0$, $\lambda_M \shrink{>} 0$, 
		 \label{ExpansionF-b} \\
		\tau_M c_G \left(\frac{-c_H}{\rho_G}\right)^{\lambda_G}\!\!\left(1-\frac{z}{\rho_H}\right)^{\lambda_{G}\lambda_H}
\!\!+\! c_M \tau_G \left(1-\frac{z}{\rho_H}\right)^{\lambda_{M}} \!\!+\! \dots\! &
		if $0\shrink{<}\lambda_G \shrink{<} 1$, 
		 \label{ExpansionF-c} \\
		 G'(\rho_G) \tau_M c_H \!\!\left(1-\frac{z}{\rho_H}\right)^{\lambda_H}
+c_M \tau_G \left(1-\frac{z}{\rho_H}\right)^{\lambda_{M}} \!+ \dots\! &
		if $\lambda_G \shrink{>} 1$, 
		 \label{ExpansionF-d} 
\end{subnumcases}}
\!where, if $\lambda_M=+\infty$, $\tau_M=c_M=M(\rho_H)$ and $\left(1-\frac{z}{\rho_H}\right)^{\lambda_M}=0$. 

\noindent The case $\lambda_H>1$ is obtained analogously.
\begin{comment}
When $\lambda_H>1$ the linear term of $H(z)$ is nonzero and therefore $\lambda_G$ plays the r\^ole of $\lambda_G\lambda_H$ in the above expansions.
So, for $\lambda_H>1$, one gets (omitting the constants for simplicity):
{\small\newcommand{\shrink}[1]{\!#1\!}
\begin{subnumcases}{F(z) \shrink{=}}
		 C_1 \!\!\left(1-\frac{z}{\rho_H}\right)^{\lambda_{G} + \lambda_M} \!+ \dots\! &
	 if $\lambda_G \shrink{<} 0$, $\lambda_M \shrink{<} 0$, 
		 \nonumber \\
		 C_2 \!\!\left(1-\frac{z}{\rho_H}\right)^{\lambda_{G}}
		\!+ \dots &
		if $\lambda_G \shrink{<} 0$, $\lambda_M \shrink{>} 0$, 
		 \nonumber \\
		C_{3a}\!\!\left(1-\frac{z}{\rho_H}\right)^{\lambda_{G}}
\!+C_{3b} \left(1-\frac{z}{\rho_H}\right)^{\lambda_{M}} \!\!+ \dots\! &
		if $0\shrink{<}\lambda_G \shrink{<} 1$, 
		 \nonumber \\
		C_{4a}\left(1-\frac{z}{\rho_H}\right)^{\lambda_G}
\!\!\!+C_{4b} \left(1-\frac{z}{\rho_H}\right)^{\lambda_{M}} 
\!\!+C_{4c} \left(1-\frac{z}{\rho_H}\right)^{\lambda_{H}}\!\!+ \dots\! &
		if $\lambda_G \shrink{>} 1$.
		 \nonumber 
\end{subnumcases}
}
Finally, the singular exponent $\lambda_F$ is equal to the minimal exponent in each of these Puiseux expansions.
\end{comment}
\end{proof}
This lemma motivates the following definition.
\begin{defi}[Pure/confluent/degenerate composition schemes]
	\label{def:pure}
	Consider an extended or size-refined composition scheme~\eqref{eq:scheme_extended} or~\eqref{eq:scheme_size-size-refined} 
	with a unique dominant singularity $\rho_F=\rho_H$, and with $0<\lambda_H<1$.
	It is either analytically
\begin{center}
\renewcommand{\arraystretch}{1.3}
\begin{tabular}{cll}
	$\bullet$ & 
	\emph{pure} &
	if 
 $\begin{cases}
	 \lambda_G<0 \text{ or } \\ 
	 0 < \lambda_G < 1 \text{ and } \lambda_M > \min(\lambda_G \lambda_H,\lambda_H);
	 \end{cases}$\\
 $\bullet$ &
 \emph{confluent} & 
 if \, $0 < \lambda_G < 1$ and $\lambda_M = \min(\lambda_G \lambda_H,\lambda_H)$;\\
	$\bullet$ &
	\emph{degenerate} & 
	if 
		 $\begin{cases}
		 \lambda_G>1 \text{ or } \\ 
		 0 < \lambda_G < 1 \text{ and } \lambda_M < \min(\lambda_G \lambda_H,\lambda_H).
		 \end{cases}$	
\end{tabular}
\end{center}
\end{defi}
This is pictorially summarized by Figure~\ref{fig:pure-confluent-degenerate}.

\begin{figure}[!h]
\begin{center}
\includegraphics[width=.8\textwidth]{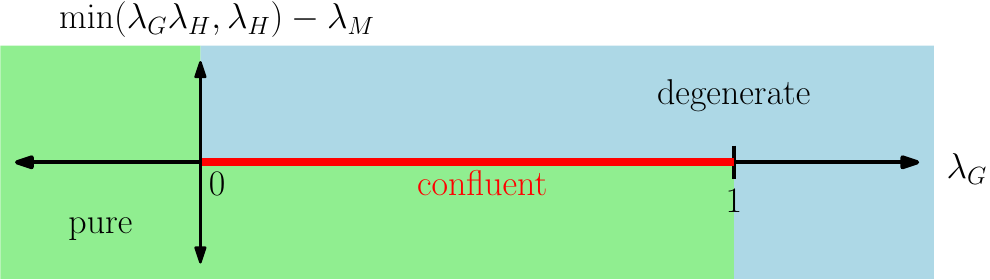} 
\end{center}
\caption{The three different r\'egimes (pure, confluent, degenerate) for extended or size-refined composition schemes: The Puiseux expansions of $G/H/M$ go into resonance (or not), thus leading to these three cases.}
\label{fig:pure-confluent-degenerate}
\end{figure}

We characterize the distributions associated with critical composition schemes
in the analytically pure/confluent/degenerate cases in Theorems~\ref{TheExtended},~\ref{TheExtendedDegenerate}, and~\ref{TheExtendedConfluent}. 
We now present some probabilistic results on the distributions which will appear in these theorems.

\pagebreak
\subsection{Probability distribution melting pot}\label{ssMP}
First, we discuss properties of tilted distributions, their link with positive stable distributions,
and we introduce the three-parameter Mittag-Leffler\linebreak distribution.
Then, we collect properties of a family of discrete distributions called mixed Poisson distributions,
and we end this section with a brief introduction to Boltzmann distributions.

For a random variable $X$ of density $f(x)$,
the tilt of $f(x)$ by a nonnegative integrable function $g(x)$ is the density $\frac{g(x)}{\E(g(X))} \cdot f(x)$.
An important class of tilted densities are the \textit{polynomially tilted densities}, where one tilts by a polynomial $g(x)=x^c$
(with $c$ being any real value such that $\E(X^c)$ is well defined).
We then use the notation
\begin{equation*}
\tilt_c(f(x))=\frac{x^c}{\E(X^c)} \cdot f(x). 
\end{equation*}

\indent Such tilted densities occur in many places: in the degree distribution in preferential attachment trees~\cite{James2013,James2015}, in Lamberti-type laws~\cite{James2010}, 
in triangular urn schemes~\cite{Jan2006,Jan2010}, in node-degrees in plane-oriented recursive trees~\cite{KuPa2007},
and in table sizes in the Chinese restaurant process~\cite{Aldous1983,Pitman1995,Pitman2006,KuPa2014}.
Note that many classes of distributions like the beta distribution, generalized gamma distribution~\cite{BanderierMarchalWallner2020}, the $F$-distribution, 
the beta-prime distribution, and distributions with gamma-type moments~\cite{Jan2010} 
are closed under the tilting operation. 

The following lemma shows that 
the operator $\tilt_c$ admits in fact several equivalent definitions
using the density, the moments, or the Laplace transform (see also~\cite[Remark~2.11]{Jan2010}).

\begin{lem}[Polynomially tilted density functions and moment shifts]
\label{LemMain}
Consider a random variable $X$ with moment sequence $(\mu_s)_{s\ge 0}$ and density $f(x)$ with support $[0,\infty)$. 
Now consider a random variable $X_c$ with $c\in\N$, having a distribution uniquely determined 
by its moments\footnote{See e.g.~Feller~\cite[Chapter VII.3]{Feller1949} for conditions implying that the moments define uniquely the distribution.}. Then the following properties are equivalent:
\begin{enumerate}
	\item\label{ita} Tilted density: $X_c$ is a random variable with density $f_c(x) = \frac{x^c}{\mu_c}\cdot f(x)$.
	\item\label{itb} Shifted moments: $X_c$ is a random variable with moments $\E(X_c^s)=\frac{\mu_{s+c}}{\mu_c}$. 
	\item\label{itc} Differentiated moment generating function: $X_c$ is such that
	$\E(e^{tX_c})=\frac{1}{\mu_c}\frac{d^c}{dt^c}\E(e^{tX})$.
\end{enumerate}
\end{lem}

\begin{remark}[Tilt with $c\in \R$]
For the properties~\eqref{ita} and \eqref{itb} of~Lemma~\ref{LemMain}, 
it is possible to extend their equivalence to $c\in \R$, assuming that the corresponding moments exist. 
More generally, the equivalence between properties~\eqref{ita} and \eqref{itb} stays valid for any random variable with density $f(x)$ 
such that only moments $\mu_1,\dots,\mu_n$ exist up to a certain value $n\ge 1$. 
\end{remark}

\begin{comment}
\begin{remark}[Densities with support $\R$]
In Lemma~\ref{LemMain}, if $c$ is even, then one can drop the restriction that the support of $f(x)$ is in $[0, \infty)$.
\end{remark}
\end{comment}

\begin{remark}[The tilt operator for densities/moments/random variables]
This lemma justifies a slight abuse of notation: 
Starting with the densities of $X$ and $X_c$ linked by $\tilt_c(f(x))=f_c(x)$,
the operator $\tilt_c$ is also used to denote
the corresponding tilted random variable $\tilt_c(X) \eq X_c$
and the corresponding tilted moments $\tilt_c ( \mu_s) \eq \frac{\mu_{s+c}}{\mu_c}$.
\end{remark}

\begin{proof}[Proof of Lemma~\ref{LemMain}]
For \eqref{ita} $\Rightarrow$ \eqref{itb},
first observe that $f_c(x)$ is indeed a density: One has $f_c(x)\geq 0$ on $[0,\infty)$ and $\int_0^{\infty}f_c(x) \, dx=\frac{\mu_c}{\mu_c}=1$. 
 Then, one checks that
\begin{equation*}
\E(X_c^s) 
= \int_0^{\infty}x^s f_c(x) \, dx 
=\int_0^{\infty}x^{s+c} \frac{f(x)}{\mu_c} \, dx
=\frac{\mu_{s+c}}{\mu_c}.
\end{equation*}
The fact that $X_c$ is uniquely determined by its moments then implies \eqref{itb} $\Rightarrow$ \eqref{ita}.

For \eqref{itb} $\Leftrightarrow$ \eqref{itc}, observe that
%\begin{equation*}
$\frac{d^c}{dt^c}\E(e^{tX})=\frac{d^c}{dt^c}\sum_{s\ge 0}\frac{\mu_s t^s}{s!}=\sum_{s\ge c}\frac{\mu_s t^{s-c}}{(s-c)!}
=\sum_{s\ge 0}\frac{\mu_{s+c} t^{s}}{s!},
$
%\end{equation*}
and, on the other hand, \eqref{ita} and \eqref{itb} give 
\begin{equation*}
\E(e^{tX _c})=\int_0^{\infty}f_c(x)e^{tx} \, dx=\frac{1}{\mu_c} \sum_{s\ge 0}\frac{t^s}{s!}\int_0^{\infty}x^{s+c}f(x) \, dx
=\frac{1}{\mu_c} \sum_{s\ge 0}\frac{t^s}{s!}\mu_{s+c}. \qedhere
\end{equation*}
\end{proof}
\pagebreak

Let us now introduce 
\textit{positive stable laws} (also called {\textit{one-sided stable laws}, as their density has support $(0,+\infty)$).
\begin{defi}[Positive stable laws and their negative powers] %Definition 3.7
\label{ex:positivestable}
We say that a positive random variable $S_{\alpha}$ follows a stable law of parameter $\alpha\in(0, 1)$ 
if its Laplace transform is $\E(e^{-t S_{\alpha}}) = e^{-t^{\alpha}}$ 
(see~\cite{UchaikinZolotarev1999} for a general presentation involving skewness, scale, and location parameters; they are respectively always $0$, $1$, and $0$ in our work).
%skewness beta=0, scale c= 1, location mu=0
The density of $S_\alpha$ is\footnote{Throughout this article, we use that, by analytical continuation,
$1/\Gamma(m) = 0$ whenever $m$ is an integer $\leq 0$.}
\begin{align}
f_{S_{\alpha}}(x)
&=\sum_{n=1}^{\infty}\frac{(-1)^n}{n! \Gamma(-n\alpha)}x^{-n\alpha-1}
\label{StableDensity}\\
&=\frac{1}{\pi}\sum_{n=1}^{\infty}(-1)^{n+1}\frac{\Gamma(n\alpha+1)\sin(\pi n\alpha)}{n!}x^{-n\alpha-1} \notag
\\
&=\frac{1}{\pi} \frac{\alpha}{1-\alpha} \int_0^\pi \frac{K(\phi)}{x^{1/(1-\alpha)}} \exp\left(-\frac{K(\phi)}{x^{\alpha/(1-\alpha)}}\right) d\phi,
\label{StableDensityInt}
\end{align}\noindent
where
\begin{equation*}
K(\phi)=\left(\frac{\sin(\alpha \phi)}{\sin(\phi)}\right)^{1/(1-\alpha)}
\frac{\sin((1-\alpha) \phi)}{\sin(\alpha \phi)}. %typo = denominator of K for Ibragimov (and propagated in James, Pitman, etc.)
\end{equation*}
%\begin{equation} equivalent formula
%K(\phi)=\left(\frac{\sin(\alpha \phi)}{\sin(\phi)}\right)^{\alpha/(1-\alpha)}
%\frac{\sin((1-\alpha) \phi)}{\sin(\phi)}.
%\end{equation}
Formula~\eqref{StableDensity} was first obtained by Humbert~\cite{Humbert1945}, and then rigorously proven by Pollard~\cite{Pollard1946}
(see also Feller~\cite[Chapter XVII.6, Lemma 1]{Feller}, with the parameter $\gamma=-\alpha$ therein).
Formula~\eqref{StableDensityInt} is due, up to a typo that we corrected here, to Ibragimov and Chernin~\cite{IbragimovChernin1959}.

Now, let $\mybeta>0$ be some real number. 
Since $\P\{\S_\alpha^{-\mybeta}\le x\}=1-\P\{S_{\alpha}< x^{-1/\mybeta}\}$,
we directly obtain from~\eqref{StableDensity} the density of $\S_\alpha^{-\mybeta}$ on its support $(0,+\infty)$:\begin{equation}
f_{\S_\alpha^{-\mybeta}}(x)=\frac{1}{\mybeta}\sum_{n=1}^{\infty}\frac{(-1)^n}{n!\Gamma(-n\alpha)}x^{n\alpha/\mybeta-1}= \frac{x^{-1/\mybeta-1}}{\mybeta} f_{\S_\alpha}(x^{-1/\mybeta}).
\label{StableDensity2}
\end{equation}
Its moments are given by (see, e.g., Janson's survey on moments of Gamma type~\cite{Jan2010}): 
\begin{equation*}
\E(\S_{\alpha}^{-\mybeta s})=\frac{\Gamma(\frac{s\mybeta}{\alpha}+1)}{\Gamma(s\mybeta+1)},\quad s> -\frac{\alpha}{\mybeta}.
\end{equation*}
\end{defi} 

We will encounter composition schemes leading to powers of stable laws, an important subcase of it being the Mittag-Leffler distribution.\footnote{In the literature, there are unfortunately two distinct distributions
which are called \textit{Mittag-Leffler distribution}.
Both of them are defined in terms of the function $E_\alpha(x)$ introduced in 1903 by Mittag-Leffler~\cite{MittagLeffler1903a,MittagLeffler1903b}. %For the editor: take care to keep the chronological order in the biblio (which is, here, not the alphabetical order of titles). 
The first distribution (which we use in this article)
was popularized by Feller~\cite{Feller1949} %~\cite[Ch.~XIII.8]{Feller1949}
 (with a slight change of variable) and by Darling and Kac~\cite{DarlingKac1957} for the study of the local time of Markov processes.
It has an exponentially bounded tail. The second one, which has a heavy tail, was introduced by Pillai~\cite{Pillai1990}, and should rather be called the Pillai--Mittag-Leffler distribution.}

\begin{defi}[Mittag-Leffler distribution]
\label{ex:mittagleffler}
\def\Ma{M_\alpha}
We say that a random variable $\Ma$ follows a \textit{Mittag-Leffler distribution} $\ML(\alpha)$ if $\smash{\Ma \law \S_\alpha^{-\alpha}}$.
Its moment generating function $\E(e^{t \Ma})$ is the Mittag-Leffler function $E_{\alpha}(t)=\sum_{k\geq 0} \frac{t^k}{\Gamma(1+\alpha k)}$. %(see~\cite{MittagLeffler1903a,MittagLeffler1903b}).
An important special case is
$M_{\frac12}$, the \emph{half-normal distribution} $|\mathcal{N}(0,\sigma^2)|$ with $\sigma=\sqrt{2}$; see Example~\ref{ExHaNo} hereafter.
\pagebreak

More generally, we will encounter computations leading to the following moments of random variables of shape $X_c=\tilt_c(\S_\alpha^{-\mybeta})$
(which are  well defined for $c> -\frac{\alpha}{\mybeta}$):
\begin{equation}
\E(X_c^s)=\frac{\Gamma(\frac{(s+c)\mybeta}{\alpha}+1)}{\Gamma((s+c)\mybeta+1)}\frac{\Gamma(\mybeta c+1)}{\Gamma(\frac{c\mybeta}{\alpha}+1)}
=\frac{\Gamma(\frac{(s+c)\mybeta}{\alpha})}{\Gamma((s+c)\mybeta)}\frac{\Gamma(\mybeta c)}{\Gamma(\frac{c\mybeta}{\alpha})}.
\label{MomentStable}
\end{equation}
Note that, by~\eqref{StableDensity2} and Lemma~\ref{LemMain}, the density of $X_c$ is given by
\begin{equation*}
f_{X_c}(x)=\frac{\Gamma(\mybeta c+1)}{\mybeta \Gamma(\frac{c\mybeta}{\alpha}+1)}\sum_{n=1}^{\infty} \frac{(-1)^n}{n!\Gamma(-n\alpha)}x^{n\alpha/\mybeta+c-1}.
\end{equation*}
E.g., $\smash{\tilt_1(M_{\frac12})}$ is the \emph{Rayleigh distribution} of parameter $\sqrt{2}$; see Example~\ref{ExRay} hereafter.
\end{defi}

Another fundamental instance of such a tilted power of stable distribution is a two-parameter generalization of the Mittag-Leffler distribution,
which was considered in the literature in link with different probabilistic processes, 
for example by Pitman~\cite{Pitman2006}, James~\cite{James2015}, Goldschmidt, Haas, and Sénizergues~\cite{GoldschmidtHaas2015, GoldschmidtHaasSenizergues2020}. 
We will establish in the next section that the analytic phenomenon hiding behind this distribution is a critical composition scheme 
(and we will explain in subsequent sections how it is related to these probabilistic processes). For now,
let us  give a formal definition of this distribution.

\begin{defi}[Two-parameter Mittag-Leffler distribution]
\label{ex:genmittagleffler}

For $\alpha \in (0,1)$ and $\beta > -\alpha$, \linebreak
we say that a random variable $X$ follows a \emph{two-parameter Mittag-Leffler distribution} $\ML ( \alpha, \beta)$ if $X \law (\tilt_{-\beta} (S_{\alpha}))^{-\alpha}$.
%; see Goldschmidt and Haas~\cite{GoldschmidtHaas2015} and James~\cite{James2015}. 
Note that one has $\ML(\alpha,0)=\ML(\alpha)$
and that $\ML(\alpha,\beta)$ is uniquely defined by its moments 
\begin{align}
 \label{momgenML}
 \E(X^s) 
 &= \frac{\Gamma\left(s + \frac{\beta}{\alpha} + 1\right) \Gamma(\beta+1)}{\Gamma(\alpha s + \beta + 1) \Gamma\left(\frac{\beta}{\alpha} + 1\right)}
 = \frac{\Gamma\left(s + \frac{\beta}{\alpha}\right) \Gamma(\beta)}{\Gamma(\alpha s + \beta) \Gamma\left(\frac{\beta}{\alpha}\right)}.
\end{align}
Comparing these moments with~\eqref{MomentStable} we directly get that for $\beta=\alpha$ and $c=\beta/\alpha$ we have 
\begin{equation}\label{tiltedtilt}
\tilt_{\beta/\alpha} (\ML(\alpha)) \law \ML(\alpha,\beta), \text{ \quad i.e., \quad }
 \tilt_{\beta/\alpha} (S_{\alpha}^{-\alpha}) = (\tilt_{-\beta} (S_{\alpha}))^{-\alpha}. 
\end{equation}
In other words, the permutation of the tilt and the power 
creates a change of the tilt parameter.
\end{defi}

Next, we discuss product distributions.
First, we recall properties of the beta distribution.

\begin{defi}[Beta distribution] \label{ex:beta}
A \textit{beta-distributed} random variable $B \law \operatorname{Beta}(a,b)$ with parameters $a,b>0$ has 
a probability density function defined on $(0,1)$ by 
\begin{equation*}f(x)=\frac{\Gamma(a+b)}{\Gamma(a) \Gamma(b)}x^{a-1}(1-x)^{b-1}.\end{equation*}
The moments of $B$ are given by 
\begin{equation} \label{eq:MomentBetaFrac}
\E(B^s) =\frac{\Gamma(s+a)\Gamma(a+b)}{\Gamma(s+a+b)\Gamma(a)},\quad s>0,
\end{equation} 
and the beta distribution is uniquely determined by the sequence of its moments. 
Furthermore, let the reader be convinced of the convenient convention $\operatorname{Beta}(a,0)\law 1$. \end{defi}

We now have all the ingredients to present the main properties of the distribution which will play a key rôle in the next sections, namely, the three-parameter Mittag-Leffler distribution.

\begin{defi}[Three-parameter Mittag-Leffler distribution]\label{def:BetaMittagLeffler}
We define the \textit{three-parameter Mittag-Leffler distribution} $\BML(\alpha,\beta,\gamma)$ as the distribution of the product of independent random variables
	\begin{align*}
	Z \law Y \cdot B^{\alpha}
	\end{align*}
where $Y \law \tilt_{\beta/\alpha}(\S_\alpha^{-\alpha}) \law \ML(\alpha,\beta)$ 
	and $B \law \operatorname{Beta}(\beta,\gamma)$ are respectively distributed like a Mittag-Leffler and a beta distribution, 
	with $0<\alpha<1$, $\beta>0$, and $\gamma\geq 0$.
\end{defi}

\begin{lem}
\label{lem:MomentBetaStable}
 The three-parameter Mittag-Leffler distribution $\BML(\alpha,\beta,\gamma)$ is uniquely determined by 
 its moments 
	\begin{equation}
	\E(Z^s)=
	\frac{\Gamma\left(s+\frac{\beta}{\alpha}\right)\Gamma\left(\beta+\gamma\right)}{\Gamma\left(\alpha s+\beta+\gamma\right)\Gamma\left(\frac{\beta}{\alpha}\right)} ,  \quad s>0.
	\label{MomentBetaStable}
	\end{equation}
What is more, one has the following identity
	\begin{equation}\label{GenMittagLefflerWillRuleTheWorld}
	 \ML(\alpha, \beta, \gamma) \law \ML(\alpha, \beta) \operatorname{Beta}(\beta, \gamma)^{\alpha} \law \ML(\alpha,\beta + \gamma) \operatorname{Beta}(\frac{\beta}{\alpha}, \frac{\gamma}{\alpha}).
 \end{equation}
 In particular, one has  $\BML(\alpha,\beta,0)=\ML(\alpha,\beta)$.
%$\BML(\alpha,0,0)=\ML(\alpha)$,  %could exist via first lim gamma to 0, and then via 16 beta = 0. but not a critical scheme 
Furthermore, one also has the simplification 
\begin{equation*}
\BML(\alpha,\alpha,1-\alpha)=\ML(\alpha), \end{equation*}
and, more generally, for any real $\kappa\geq 1$, $\BML(\alpha,\alpha \kappa,1-\alpha)=\ML(\alpha,\alpha (\kappa-1))$.
\end{lem}
\begin{proof}
The moments characterize a unique distribution via Carleman's criterion~\cite[pp.~189--220]{Carleman23}.
Now, due to the independence of the random variables we have $\E(Z^s)=\E(Y^s) \E(B^{\alpha s})$.
Then, using~\eqref{MomentStable} and~\eqref{eq:MomentBetaFrac}, one gets~\eqref{MomentBetaStable}.
With the relation~\eqref{tiltedtilt}, this gives~\eqref{GenMittagLefflerWillRuleTheWorld}.
The last simplifications of the lemma follow from the identity~\eqref{momgenML}.
\end{proof}

Note that $\BML(\alpha,\beta,\gamma)$
is one important instance of a
``distribution with moments of Gamma type'',
a class of distributions popularized by Janson in his nice thorough survey~\cite{Jan2010},
where he obtained the following asymptotics of the  tail of their density $f(x)$.
\begin{align}\label{tail}
\qquad &f(x) \sim C x^{d-1} \exp(-c x^\frac{1}{1-\alpha}) 
 \text{ \qquad for \ } x\sim +\infty, \\
\text{with \ } &
c=(1-\alpha) \alpha^{\frac{\alpha}{1-\alpha}},
\quad
d=\frac{\beta/\alpha-\beta-\gamma+1/2 }{1-\alpha}, 
\quad %\text{ and  \ } 
C= \frac{\Gamma(\beta+\gamma)}{\Gamma(\beta/\alpha)} \frac{\alpha^{\frac{1-2\gamma}{2(1-\alpha)}}}{\sqrt{2\pi(1-\alpha)}}. \notag
\end{align}
%Maple code to check the last simplification step: d:=(t/a-t-b+1/2)/(1-a); simplify(a*d - t - b + 1/2);
The behaviour at 0 of the density of  $\BML(\alpha,\beta,\gamma)$ follows from Eq.~\eqref{density} in the next section
and is summarized in Table~\ref{ML0}.

\begin{figure}[h]
\begin{tabular}{ccccc} 
\includegraphics[width=.17\textwidth]{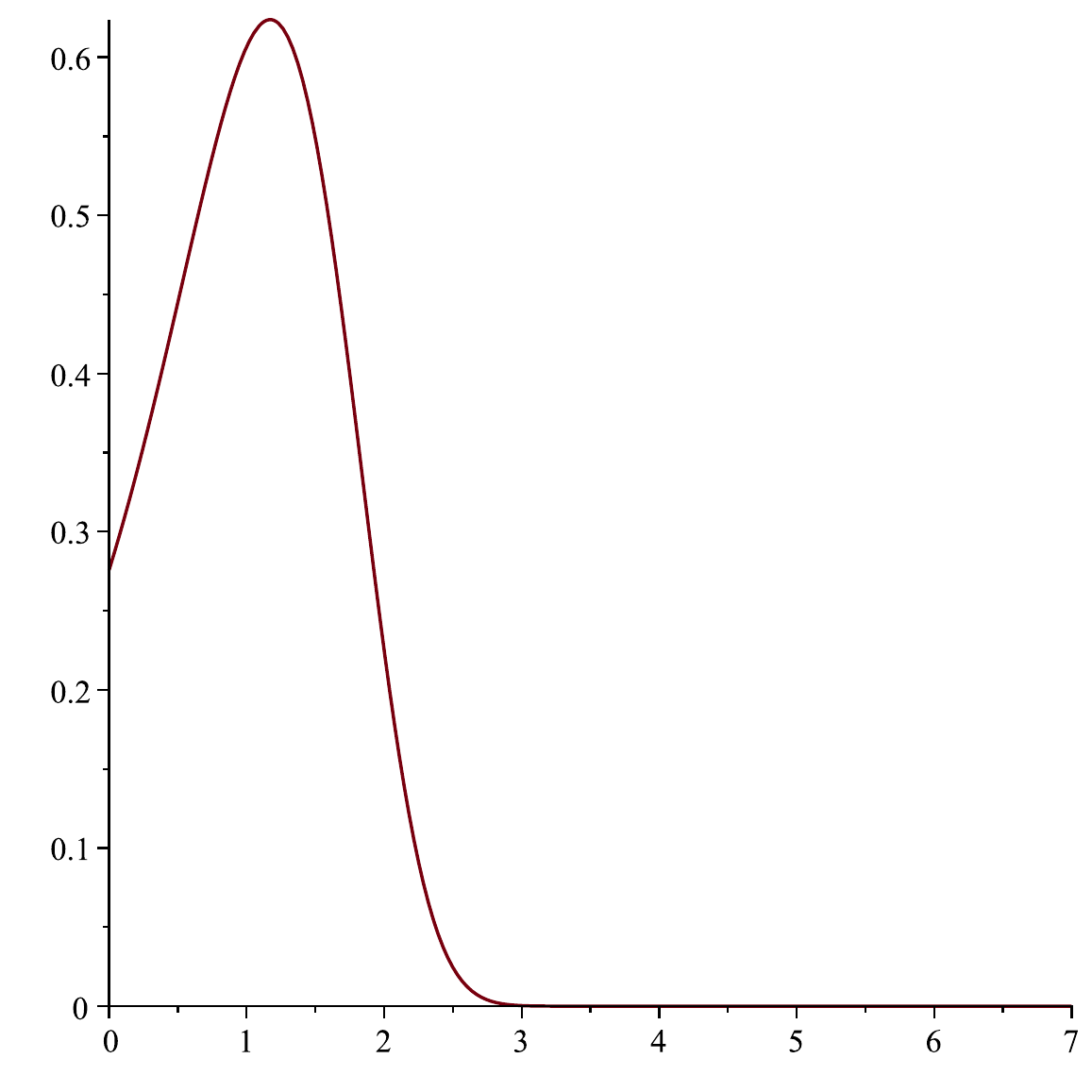} 
&\!\!\!\!\includegraphics[width=.17\textwidth]{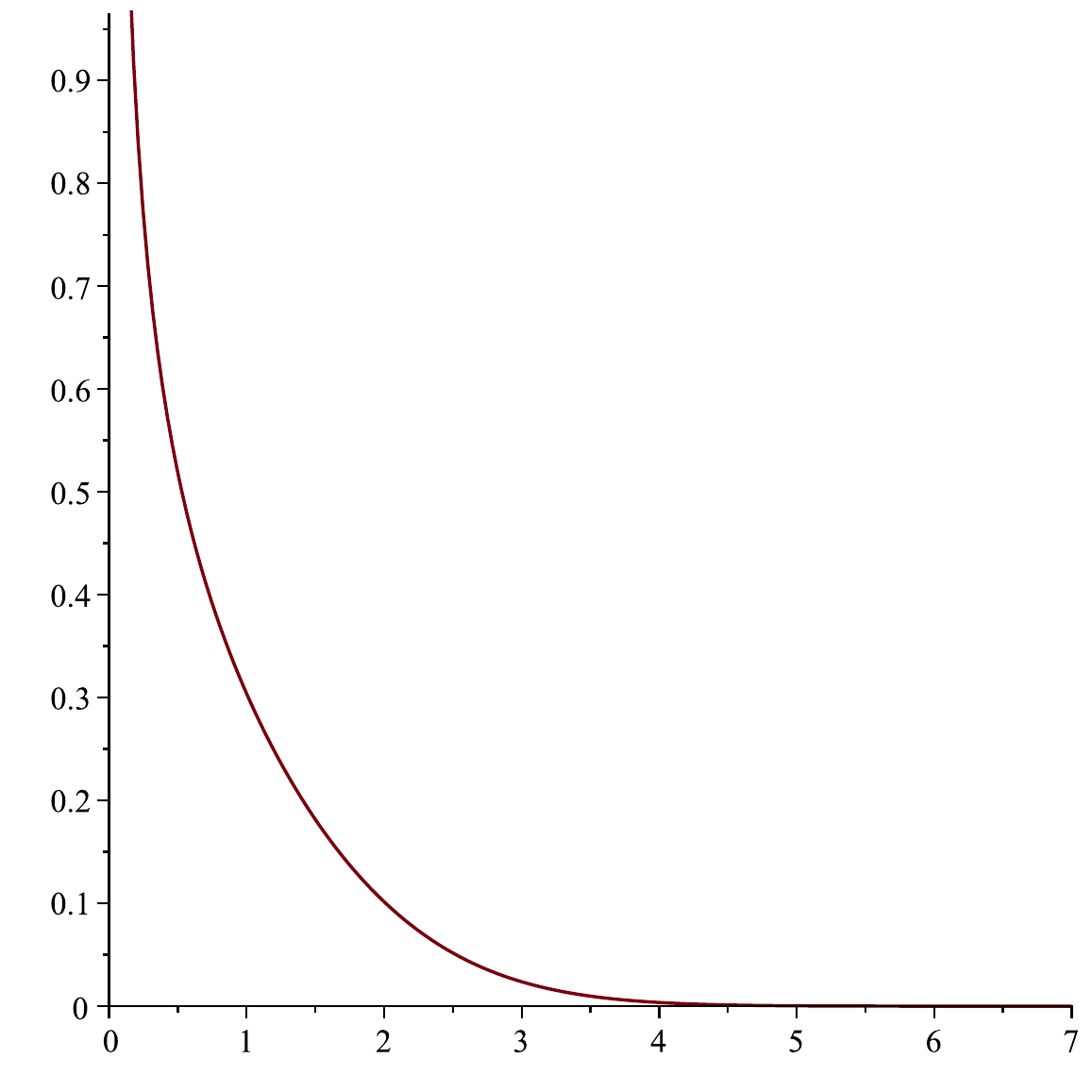} 
&\!\!\!\!\includegraphics[width=.17\textwidth]{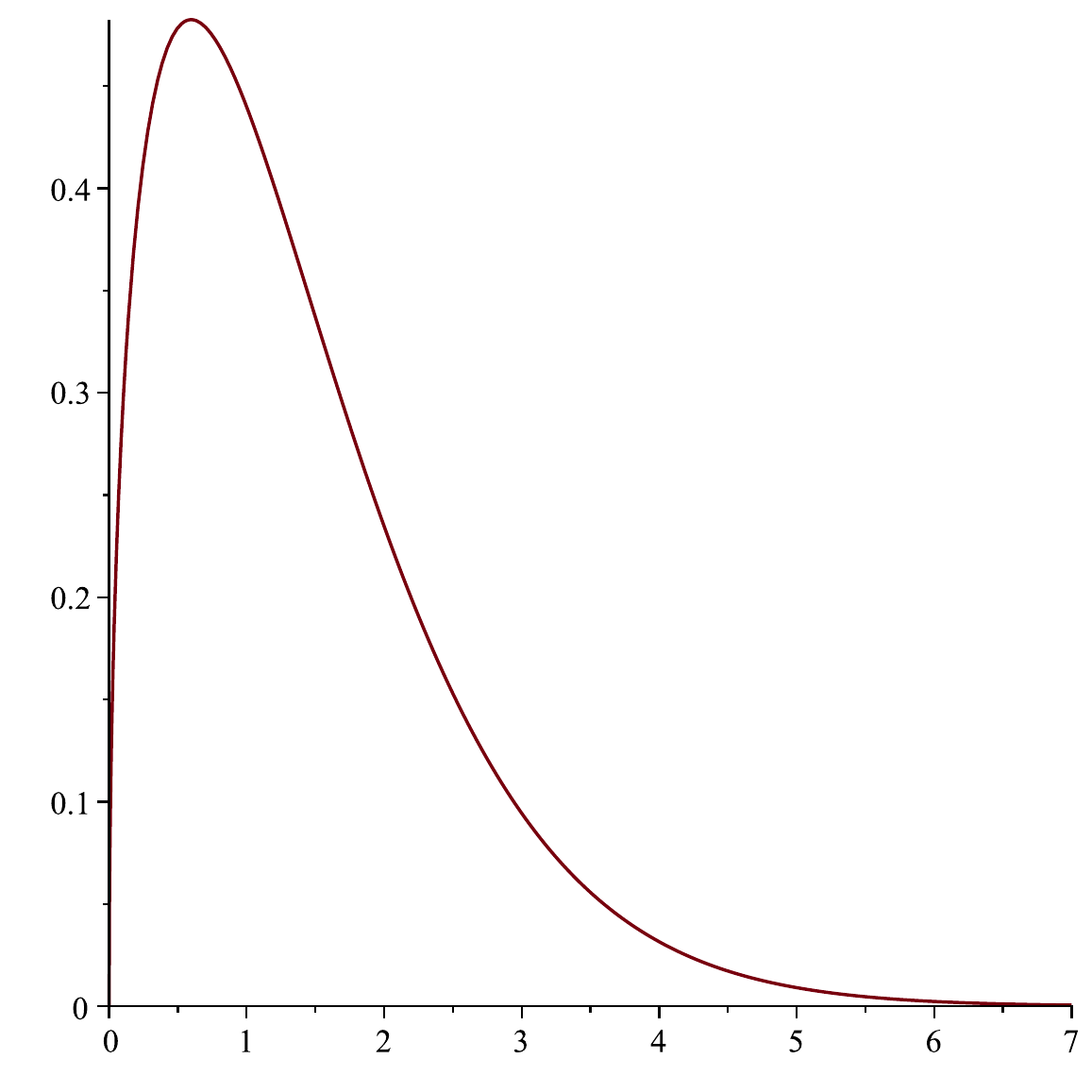} 
&\!\!\!\!\includegraphics[width=.17\textwidth]{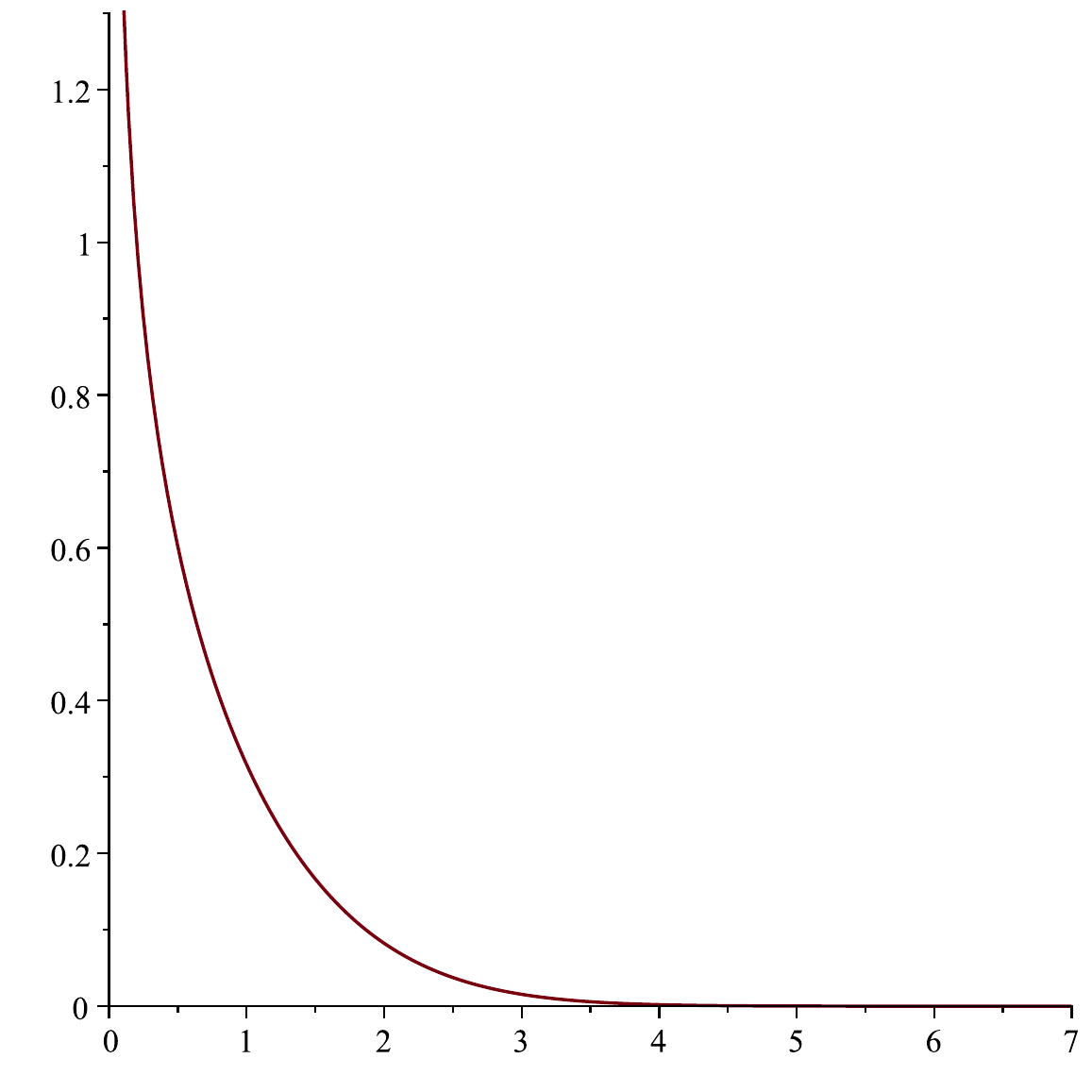} 
&\!\!\!\!\includegraphics[width=.17\textwidth]{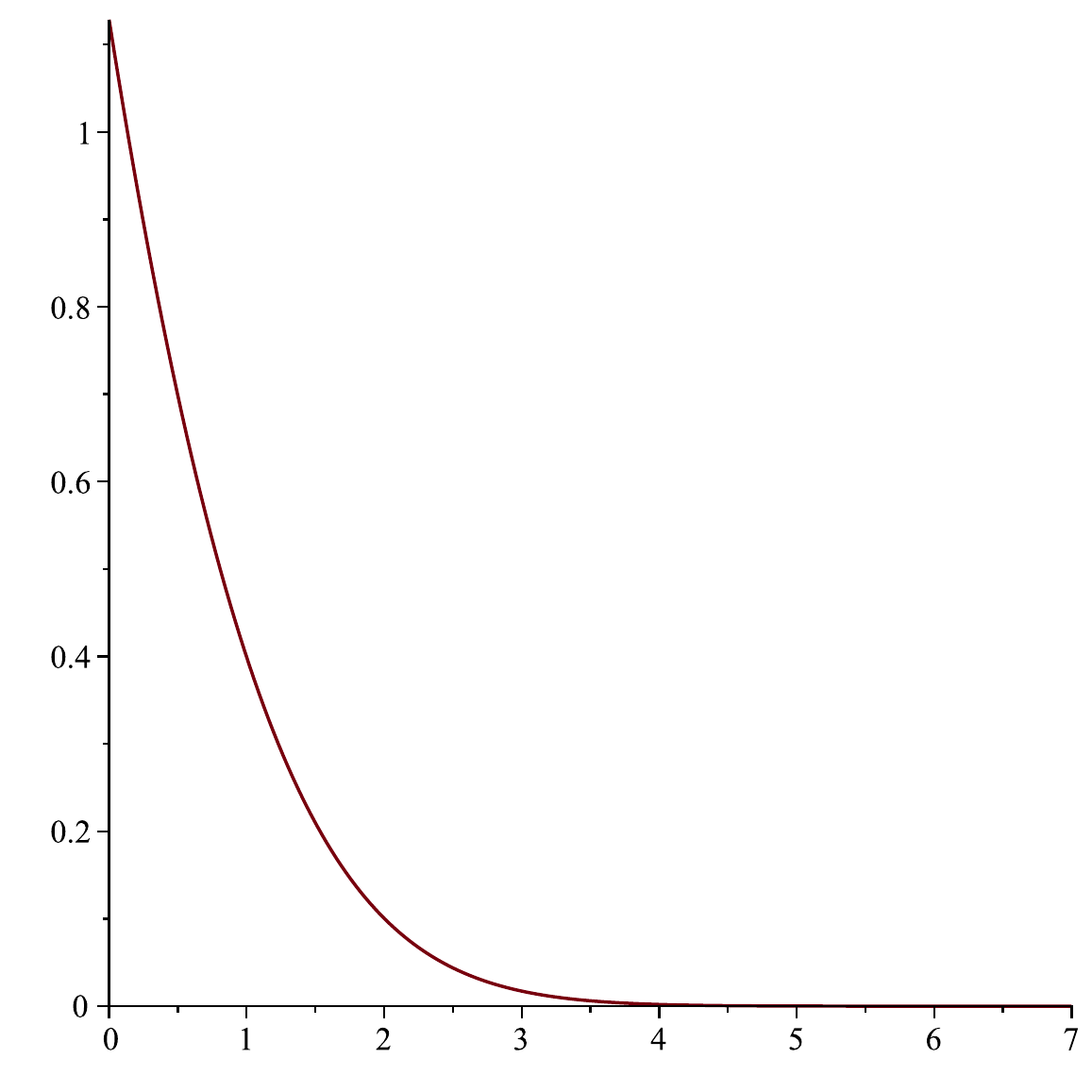} 
\\
$\ML(\frac{3}{4})$&   $\ML(\frac{1}{2},-\frac{1}{4})$  & $ \ML(\frac{1}{2},\frac{1}{3},1)$ &   $\ML(\frac{1}{3},\frac{1}{2},1)$ &   $ \ML(\frac{1}{2},\frac{1}{2},\frac{3}{2})$ 
\end{tabular}
\caption{Several instances of Mittag-Leffler distributions.
All of them have an exponential tail while the behaviour at~0 is more diverse:
One has $\ML(\alpha)\sim \frac{1}{\alpha|\Gamma(-\alpha)|}$, $\ML(\alpha,\beta)\sim Cst \cdot x^{\beta/\alpha}$, 
and $\ML(\alpha,\beta,\gamma) \sim  Cst \cdot  x^{\beta/\alpha-1}$.  %(for $\gamma>0$), implicit constraint here
}\label{ML0}
\end{figure}

\pagebreak
\begin{remark}[A link with electromagnetism]\label{electromagnetismremark}
The moment generating function of $Z=\BML(\alpha,\beta,\gamma)$ 
can be expressed 
via the function 
\begin{equation*}
E_{\alpha,\beta'}^{\gamma'}(t):=\sum_{k=0}^\infty \frac{\Gamma(k+\gamma')}{\Gamma(\alpha k+\beta') \Gamma(\gamma')} \frac{t^k}{k!};
\end{equation*}
one has indeed
\begin{equation*}
\begin{gathered}
\E(e^{tZ}) = \Gamma(\beta') E_{\alpha,\beta'}^{\gamma'}(t) 
\text{\qquad (with $\beta' = \gamma+ \beta$ and $\gamma'= \beta/\alpha$).}
\end{gathered}
\end{equation*}
Our article thus provides a unified probabilistic interpretation of the special function $E_{\alpha,\beta'}^{\gamma'}(t)$ which also appears in physics, 
where it is sometimes called the \textit{Prabhakar function}, 
or the \textit{three-parameter Mittag-Leffler function}~\cite{Prabhakar1971,GorenfloKilbasMainardiRogosin2014}.
This special function is also the inverse Laplace transform of the right-hand side 
of the Havriliak--Negami generalization of the Debye and Cole--Cole equations,
which are classical models of dielectric relaxation in electromagnetism~\cite{CapelasMainardiVaz2011,GarraGarrappa2018,GorskaHorzelaBratekDattoliPenson2018}.
What is more, it is proven by G\'orska et al.~in~\cite{GorskaHorzelaLattanziPogany2021} that $g(t):=E_{\alpha,\beta'}^{\gamma'}(-t)$ (with $0<\alpha\leq 1$, $\gamma'>0$, and $\beta'>\alpha\gamma'$, conditions which match exactly our probabilistic setting!) is \textit{totally monotone}: $(-1)^n \partial_t^n g(t)>0$ for $t\in \R^+$.
Bernstein's theorem on totally monotone functions then implies that this function is associated to a density. 
Our article unravels this mysterious density, and explains its universality.
\end{remark}

\begin{remark}[Other appearances of $\ML(\alpha,\beta,\gamma)$]
Since the initial publication of our work on arXiv, the three-parameter Mittag-Leffler distribution also 
appeared in two other articles.
In the first one, Möhle~\cite{Moehle2021Bar}, independently of our results, established that it is
the limit law appearing in a generalization of the Chinese restaurant process (see our Section~\ref{Chinese}).
In the second one, Sibisi~\cite{Sibisi2024FourparamML}, with a motivation coming from special functions, 
considered a four-parameter generalization $\ML_4(a,b,c,d)$. Comparing moments, we then see that 
\begin{flalign*} && 
\ML(\alpha,\beta,\gamma) = \ML_4(\alpha,\gamma,0,\beta) = \ML_4(\alpha,\beta+\gamma,\frac{\beta}{\alpha},0).&& \myqedhere
\end{flalign*}
\end{remark}

Another key ingredient that we shall need is the mixed Poisson distributions. 
These distributions were first introduced by Dubourdieu in 1938
for actuarial mathematics/insurance modelling~\cite{Dubourdieu1939}, and then also studied by Lundberg and others
(sometimes under the name ``compound Poisson processes'', a term that has a different meaning nowadays); 
they were also used for applications in bacteriology by Neyman \cite{Neyman1939}, %Neyman did just use MPo(Po)
in combinatorics by Kuba and Panholzer~\cite{KuPa2014},
for expectation-maximization algorithms by Karlis~\cite{Karlis},
or for the analysis 
of some point processes by Grandell~\cite{Grandell1997}.
Their unimodality properties are studied by Masse and Theodorescu~\cite{MasseTheodorescu2005}, 
and their tail asymptotics are analysed by Wilmot and Lin in~\cite{Willmot, WillmotLin2001}.

\begin{defi}[Mixed Poisson distributions]
\label{COMPSCHEMEdef1}
Let $X$ denote a nonnegative random variable with cumulative distribution function $U$.
We say that the discrete random variable $Y$ 
has a \emph{mixed Poisson distribution with mixing distribution} $U$ and scale parameter $\mppar \geq 0$, 
if its probability mass function is given for $\ell\ge 0$ by 
\begin{equation*}
\P\{Y=\ell\}=\frac{\mppar^\ell}{\ell!}\int_{\R^{+}} X^{\ell}e^{-\mppar X} dU= \frac{\mppar^{\ell}}{\ell!}\E( X^{\ell}e^{-\mppar X}).
\end{equation*}
This is summarized by the notation $Y\law \MPo(\mppar U)$, or, indifferently, $Y\law \MPo(\mppar X)$.
\end{defi}

Note that mixed Poisson distributions provide a common generalization of three major 
discrete laws ubiquitous in combinatorics~(see \cite[Figure~IX.5]{FlaSe2009}), namely
the Poisson, the geometric, and the negative binomial distributions, 
making them of great importance per se.

\pagebreak 

We emphasize that the \textit{factorial moments}\footnote{
Throughout this work we denote by $\fallfak{x}{n}$ the~\ith{n} falling factorial, $\fallfak{x}n=x(x-1)\cdots (x-n+1)$, $n\ge 0$, with $\fallfak{x}0=1$. It will be used for $\E(\fallfak{X}{n})$, the factorial moment of order $n$ of a random variable $X$.}
 of a mixed Poisson distribution are closely related to the classical \textit{raw moments} of its mixing distribution:
$\E(\fallfak{Y}{s})=\mppar^s \E(X^s),~s\ge 1.$
Additionally, like for any distribution, the factorial and raw moments of $Y$ are related via the Stirling set partition numbers $\Stir{s}{k}$ (also called Stirling numbers of the second kind):
\begin{equation*}
\E(Y^s) = \sum_{k=0}^r \Stir{s}{k} \E(\fallfak{Y}{k}).
\end{equation*}

Such relations are called Stirling transforms~\cite{BersteinSloane1995}. We refer to~\cite{KuPa2014} for more 
properties of the Stirling transform and mixed Poisson distributions,
like the following useful expression for the probability mass function of $\MPo(\mppar X)$ in terms of its \textit{factorial} moments.

\begin{prop}
\label{MOMSEQthe1}
Let $X$ denote a random variable with moment sequence given by $(\mu_s)_{s\in\N}$. 
If a random variable $Y$ has factorial moments given by $\E(\fallfak{Y}s)=\mppar^s\mu_s$,
then $Y\law \MPo(\mppar X)$. 
What is more, the sequence of moments of $Y$ is the Stirling transform of the moment sequence $(\mu_s)_{s\in\N}$, and the probability mass function of $Y$ is given by
\begin{equation*}
\P\{Y=\ell\}=\sum_{s\ge \ell}(-1)^{s-\ell}\binom{s}{\ell}\mu_{s}\frac{\mppar^s}{s!}, \quad \ell\ge 0.
\end{equation*}
\end{prop}

Let us give two short examples of mixed Poisson distributions, which, as we shall later see, correspond to ubiquitous cases in combinatorics.

\begin{example}[Mixed Poisson half-normal distribution]
\label{ExHaNo}
A half-normally distributed random variable $X \law \text{HN}(\sigma)$ with parameter $\sigma$ is 
the absolute value of a normally distributed random variable, i.e. $X\law |\mathcal{N}(0,\sigma^2)|$. Consequently, $X$ has the probability density function
\begin{equation*}
 f(x;\sigma) = \frac{\sqrt{2}}{\sigma\sqrt{\pi}} e^{-\frac{x^{2}}{2\sigma^{2}}}, \quad x \geq 0; 
\end{equation*}
alternatively, it is fully characterized by its moment sequence
 $\E(X^s)=\sigma^s 2^{s/2} {\Gamma(\frac{s+1}2)}/{\Gamma(\frac{1}2)}$.
Thus, a discrete random variable $Y$ with probability mass function
\begin{equation*}
\P\{Y=\ell\}=\frac{\mppar^\ell}{\ell!}\cdot\frac{\sqrt{2}}{\sigma\sqrt{\pi}}\int_0^\infty x^{\ell}e^{-\mppar x -\frac{x^2}{2\sigma^2}}\, dx,\quad \ell\ge 0,
\end{equation*}
has a mixed Poisson 
distribution: $Y\law \MPo(\mppar X)$ with $X\law \text{HN}(\sigma)$. Note that we can readily expand the exponential function and obtain various series representations of $\P\{Y=\ell\}$. 
\end{example}

\begin{example}[Mixed Poisson Rayleigh distribution]
\label{ExRay}
A Rayleigh distributed random variable $X \law \text{Rayleigh}(\sigma)$ with parameter $\sigma$ has the probability density function
\begin{equation*}
 f(x;\sigma) = \frac{x}{\sigma^{2}} e^{-\frac{x^{2}}{2\sigma^{2}}}, \quad x \geq 0;
\end{equation*}
alternatively, it is fully characterized by its moment sequence
 $ \E(X^{s}) = \sigma^{s} \, 2^{s/2} \,\Gamma\left(\frac{s}{2}+1\right)$.
Thus, a discrete random variable $Y$ with probability mass function
\begin{equation*}
\P\{Y=\ell\}=\frac{\mppar^\ell}{\ell! \sigma^2}\int_0^\infty x^{\ell+1}e^{-\mppar x -\frac{x^2}{2 \sigma^2}} \, dx,\quad \ell\ge 0,
\end{equation*}
has a mixed Poisson distribution: $Y\law \MPo(\mppar X)$ with $X\law \text{Rayleigh}(\sigma)$. 
%Another representation %valid for all $\mppar>0$ can be stated in terms of the incomplete gamma function $\Gamma(s,x) := \int_x^{\infty} t^{s-1}\,e^{-t}\,{\rm d}t$:
%\begin{flalign*} && 
%\P\{Y=\ell\}=\frac{(\mppar\sigma)^\ell}{\ell!}e^{\frac{(\mppar\sigma)^2}2}\sum_{i=0}^{\ell+1}\binom{\ell+1}i (-\mppar\sigma)^{\ell+1-i} \, 2^{\frac{i-1}2} \, %\Gamma\left(\frac{i+1}2,\frac{(\mppar\sigma)^2}2\right), \quad \mppar>0.
%&& \myqedhere
%\end{flalign*}
\end{example}
\pagebreak 

Last but not least, in our results, 
we shall also encounter another important family of discrete distributions: the \textit{Boltzmann distributions}. 
They were introduced in combinatorics by Duchon, Flajolet, Louchard, and Schaeffer 
in order to perform sampling of combinatorial structures~\cite{Boltz1}.
The starting point of these authors was the idea to give, like in statistical mechanics, a Gibbs measure/Boltzmann weight $x^n$ (for some fixed real number $x$) to each combinatorial object of size~$n$. Note that the objects of same size then follow a uniform distribution. 
It was then a nice surprise that, if one deals with an assemblage of combinatorial objects, the corresponding Boltzmann weights are given by very simple probabilistic laws (similarly to the symbolic method~\cite{FlaSe2009} which directly gives the generating functions of unions/products/cycles of objects).
This led to an outstanding generic linear time sampling algorithm: Its astonishing efficiency is partially due to the fact that, thanks to these Boltzmann weights, the sampling of a product 
of two combinatorial structures is simply obtained by {\em two} independent recursive subsamplings.
The sampling algorithm is thus essentially 
based on the following definition:

\begin{defi}[Boltzmann distribution]
\label{def:Boltzmann}
For any generating function $G(z)=\sum_{n\ge 0}g_n z^n$, 
and for any parameter $x>0$ smaller than the radius of convergence of $G$,
a random variable $X$ follows a \emph{Boltzmann distribution (associated with $G$) of parameter $x$}, denoted by $\mathcal{B}_{G}(x)$, if
\begin{equation*}
\P\{X=n\}=\smash{\frac{g_n x^n}{G(x)}},\quad n\ge 0.
\end{equation*}
\end{defi}
Then, the key idea behind Boltzmann sampling (of objects of size $n$) is to choose $x$ adequately to maximize $\P\{X=n\}$. If the object generated is not of size $n$, one rejects it and restarts the sampling.
This leads to a uniform sampling algorithm of optimal efficiency when $x$ is the unique real root of the equation $x G'(x) = n G(x)$.
This equation is reminiscent of many probabilistic results with mean $\mu= x G'(x)/G(x)$ (e.g., when $G$ encodes the offspring of a Galton--Watson process).
This is no coincidence: By design, Boltzmann sampling ``reverse-engineers'' these results
(see e.g.~\cite{BanderierKubaWagnerWallner2024}, or  the recent works of Sportiello~\cite{Sportiello2021} or Panagiotou, Ramzews, and Stufler~\cite{PanagiotouRamzewsStufler2021,Stufler2022} for more on these aspects).

As we shall see, these Boltzmann distributions also occur in our \textit{critical} composition schemes.
Retrospectively, it explains and puts in a unified framework earlier sporadic occurrences of such distributions
for the limit law of the degree of a random node in simply-generated trees, 
the root degree in simply-generated trees, as well as in \textit{subcritical} composition schemes; see~\cite[pages 460, 629--633]{FlaSe2009}.

\begin{figure}[h]
\includegraphics[width=.243\textwidth,page=1]{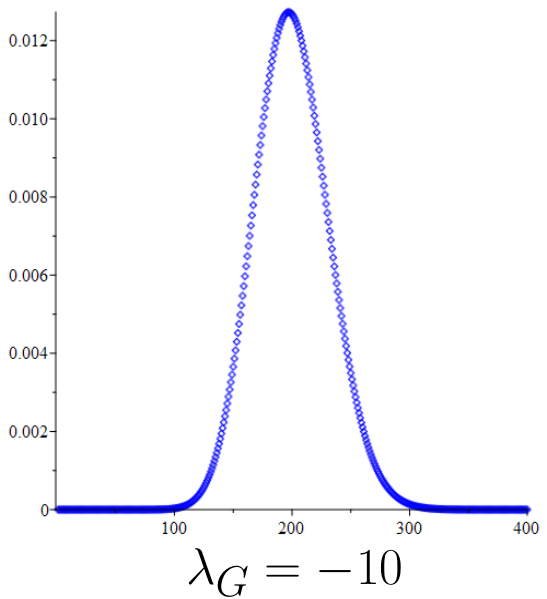} % singular exponent = -10
\includegraphics[width=.243\textwidth,page=2]{Boltzmann_legend} % singular exponent = -2
\includegraphics[width=.243\textwidth,page=3]{Boltzmann_legend} % singular exponent = -1
\includegraphics[width=.243\textwidth,page=4]{Boltzmann_legend} % singular exponent = 1/3
\caption{Boltzmann distributions $\mathcal{B}_G(x)$ have different shapes\protect\footnotemark, depending on the singular exponent $\lambda_G$ of the function~$G$: It goes from an asymptotic %in n
Gaussian shape if $\lambda_G \ll -1$ (and for entire functions) to a spread shape if $\lambda_G \gg -1$. 
The above distributions are drawn with a value of the parameter $x$ such that $\E(\mathcal{B}_G(x))=xG'(x)/G(x) =n$ (with $n=200$). It is interesting that they appear in the context of \textit{critical schemes}, independent\-ly of \textit{uniform random generation} motivations, but, as explained above, both aspects are in fact intimately related.
}
\end{figure}
\footnotetext{See \url{https://lipn.fr/~cb/Papers/CriticalSchemes/} for several animations of the different limit laws occurring in this article.}
\pagebreak

\section{Extended composition scheme\label{SecExtended}}

In the following we state and prove our main theorem on the extended critical composition scheme~\eqref{Eq4} for pure schemes.
For the terms critical and pure, we refer to Definitions~\ref{def:critical} and~\ref{def:pure}, respectively. 
This theorem shows the universality of the three-parameter Mittag-Leffler distribution (introduced in Definition~\ref{def:BetaMittagLeffler}).

\begin{theorem}[Extended composition scheme: pure case]
\label{TheExtended}
In a pure extended critical composition scheme $F(z,u)=G\big(u H(z)\big)M(z)$, the core size~$X_n$, rescaled, converges in distribution and in moments\footnote{We write 
$X_n  {\displaystyle \ensuremath{\myarrow}^{\smash{\raisebox{-3.0pt}{\scriptsize \textit{d\,\,}}}}_{\smash{\raisebox{2.4pt}{\textit{\scriptsize{m\,}}}}}} X$ 
to denote that $X_n$ converges in distribution to $X$, with convergence of all moments, i.e., $\E(X_n^s)\rightarrow \E(X^s)$ for all $s$. This notion of convergence in moments was, e.g., used in~\cite{JansonPouyanne2018,FuchsHwangNeininger2006}.
It is also indirectly used in~\cite{RoeslerRueschendorf2001}, which deals with more constrained models that offer a convergence in $L^p$, for all $p>1$. 
Note that in all our results involving convergence in distribution or in moments, we omit the speed of convergence, which is in fact easily obtained by considering 
the Puiseux expansions of order $2$ of $G/H/M$.} to a random variable $X$
distributed like the three-parameter Mittag-Leffler distribution:
\begin{align}
 \frac{X_n}{\kappa\cdot n^{\lambda_H}} \cmom X, 
 \qquad \text{ with } \qquad
 X \law &\BML(\alpha, \beta, \gamma),\label{MLd} 
\end{align}
where
\begin{align*}
 \alpha=\lambda_H, \quad 
 \beta=-\lambda_G \lambda_H, \quad
 \gamma =-\lambdaM = -\min(0,\lambda_M), \quad
 \text{ and } \quad
 \kappa=\frac{\tau_H}{-c_H}. 
\end{align*}

What is more, one has a local limit theorem 
\begin{equation*}%\label{locallimittheorem}
\P\{X_n=x\cdot \kappa n^{\lambda_H}\}\sim \frac{1}{\kappa n^{\lambda_H}}\cdot f_X(x),
\end{equation*}
with uniform convergence over any interval in $(0,+ \infty)$, and where $f_X(x)$ is the density of $X$:
\begin{equation} \label{density}
f_X(x)=\frac{\Gamma(\beta+\gamma)}{\Gamma(\beta/\alpha)}
\sum_{j\ge 0}\frac{(-1)^j }{j! \Gamma(\gamma-j\alpha)} x^{\beta/\alpha+j-1}.
\end{equation}
%NB: limit(GAMMA(theta)/GAMMA(theta/a),theta=0)=1/a, the density for BML(a,0,0) is consistent with the density of Mittag-Leffler(a)
\end{theorem}

\begin{remark}[Two simplifications of the three-parameter Mittag-Leffler distribution]
\label{CoExtended1}
In the above theorem,
 if $\lambda_M\geq 0$ (which includes the critical scheme $F(z,u)=G(uH(z))$ as $\lambda_M=+\infty$), then $\gamma=0$ and the three-parameter Mittag-Leffler distribution~\eqref{MLd} simplifies into a two-parameter Mittag-Leffler distribution:
\begin{equation*}
X\law \tilt_{-\lambda_{G}}(\ML(\lambda_{H}))
\law \ML(\lambda_H,-\lambda_G\lambda_H)
. 
\end{equation*}
In particular, for $\lambda_H=\frac12$ and $\lambda_G=-1$ 
the random variable $X$ follows a Rayleigh distribution of parameter $\sigma=\sqrt{2}$; see Example~\ref{ex:mittagleffler}.

Another noteworthy simplification occurs for $\lambda_M<0$, in the special case $\lambda_G=-1$ and $\lambda_H-\lambda_M=1$:
We then obtain a Mittag-Leffler distribution of parameter~$\lambda_H$:
\begin{equation*}
X\law \ML(\lambda_H).
\end{equation*}
In particular, for $\lambda_H=\frac12$ the random variable $X$ follows a half-normal distribution of parameter $\sigma=\sqrt{2}$; see again Example~\ref{ex:mittagleffler}.

Note that the cases with $\lambda_G=-1$ occur in a great many places in applied probability theory;
they indeed correspond to a natural combinatorial framework
where $\mathcal{F}$-objects are essentially \textit{sequences} of ${\mathcal H}$-components
(see~\cite{FlaSe2009} and examples in our Section~\ref{sec:examples}).
\end{remark} 

\begin{proof}[Proof of Theorem~\ref{TheExtended}]
The factorial moments satisfy\footnote{Throughout this work, we denote by $\partial_u$ the differentiation operator with respect to the variable $u$. Accordingly, we use the shorthand notation 
$\partial_u(F)(z,1)=\big(\partial_u F(z,u)\big)\big|_{u=1}$.}:
\begin{equation*}
\E(\fallfak{X_n}s)
=\frac{[z^n] \partial_u^s (F)(z,1)}{[z^n]F(z,1)}
=\frac{[z^n] H(z)^s G^{(s)}\big(H(z)\big)M(z)}{[z^n]G(H(z))M(z)}.
\end{equation*}

In the following we use the notation $\rho_F$, $\tau_F$, $\lambda_F$, and $c_F$ from Section~\ref{sec:prelimsingular} for the singular expansions of $F=G/H/M$.
As we are in a pure critical scheme (Definition~\ref{def:pure}), the unique singularity of $F(z)=G(H(z)) M(z)$ is at $z=\rho_H$. 
Then, we unify 
the three cases~\eqref{ExpansionF-a}, \eqref{ExpansionF-b}, and \eqref{ExpansionF-c} 
by using $\lambdaM = \min(0,\lambda_M)$ and choosing $C_M$ according to the specific case 
($C_M$ is for $\lambda_M \neq \lambda_G \lambda_H$ either $\tau_M$ or $c_M$; 
this quantity $C_M$ will anyway cancel in the end).
Note that the case \eqref{ExpansionF-d} does not hold in a pure scheme.
This gives
\begin{align*}
	F(z) &\sim C_M c_G \left(\frac{-c_H}{\rho_G}\right)^{\lambda_G}\left(1-\frac{z}{\rho_H}\right)^{\lambda_{G}\lambda_H + \lambdaM}.
\end{align*}
Therefore, using the transfer theorem from singularity analysis we get
\begin{equation}
f_n = [z^n]F(z,1)\sim \frac{C_M c_G (-c_H/\rho_G)^{\lambda_G}}{\rho_H^n}\cdot \frac{n^{-\lambda_{G}\lambda_H-\lambdaM-1}}{\Gamma(-\lambda_{G}\lambda_H-\lambdaM)}.
\label{ExpansionF2}
\end{equation}
Using singular differentiation (see \cite[Theorem~VI.8]{FlaSe2009} or~\cite{FFK2005})
for $G(z)$, we get the following singular expansion of the higher-order derivatives $G^{(s)}(z)$, for integer $s\ge 1$:
\begin{equation*}
G^{(s)}(z) \sim (-1)^{s} \frac{c_{G}}{\rho_G^s}\fallfak{\lambda_G}{s}\Big(1-\frac{z}{\rho_G}\Big)^{\lambda_{G}-s}. %+ \O\Big(1-\frac{z}{\rho_G}\Big)^{\lambda_{G}+\vartheta_G-s}
\end{equation*}
Next, from the singular expansion of $H(z)$ we directly get
\begin{equation*}
H(z)^s \sim \tau_H^s - s\tau_H^{s-1}\cdot (-c_{H})\Big(1-\frac{z}{\rho_H}\Big)^{\lambda_H}.%+ \O\Big(1-\frac{z}{\rho_H}\Big)^{2\lambda_H}+ \O\Big(1-\frac{z}{\rho_H}\Big)^{\lambda_H+\vartheta_H}.
\end{equation*}
Therefore, we get (recall that our scheme is critical, i.e., $H(\rho_H)=\tau_H=\rho_G$) 
\begin{equation*}
\begin{split}
G^{(s)}(H(z)) 
&\sim (-1)^{s} c_{G}\rho_G^{-\lambda_G}\fallfak{\lambda_G}{s}(-c_H)^{\lambda_G-s}\Big(1-\frac{z}{\rho_H}\Big)^{\lambda_{G}\lambda_H-s\lambda_H}.% + \O\Big(1-\frac{z}{\rho_H}\Big)^{\Lambda_1}.
\end{split}
\end{equation*}
Combining these expansions with the one of $M(z)$ gives the required expansion
\begin{equation*}
H(z)^s G^{(s)}\big(H(z)\big)M(z)\sim 
(-1)^{s} \tau_H^s C_M c_{G}\rho_G^{-\lambda_G}\fallfak{\lambda_G}{s} (-c_H)^{\lambda_G-s}\Big(1-\frac{z}{\rho_H}\Big)^{\lambdaM+\lambda_H\lambda_{G}-s\lambda_H}.
\end{equation*}
Next we rewrite $(-1)^s\fallfak{\lambda_G}{s}$ using the gamma function, that is, we use
%\begin{equation*}
$(-1)^s\fallfak{\lambda_G}{s}= \frac{\Gamma(s-\lambda_G)}{\Gamma(-\lambda_G)}.$
%\end{equation*}
Hence, we obtain by extraction of coefficients and singularity analysis
\begin{align}
\label{eq:pureasymptmoments}
\begin{aligned}
[z^n]H(z)^s G^{(s)}\big(H(z)\big)M(z)
& \sim \left(\frac{\tau_H}{-c_H}\right)^s \frac{C_M c_{G} (-c_H/\rho_G)^{\lambda_G}}{\rho_H^n} \times \\
& \qquad 
\frac{\Gamma(s-\lambda_G)}{\Gamma(-\lambda_G)} 
\cdot \frac{n^{-\lambda_{G}\lambda_H-\lambdaM-1+s\lambda_H}}{\Gamma(s\lambda_H-\lambda_{G}\lambda_H-\lambdaM)}.
\end{aligned}
\end{align}
Combining this expression with~\eqref{ExpansionF2} gives
%\begin{equation*} 
$\E(\fallfak{X_n}s)\sim n^{s\lambda_H}\kappa^s \cdot \mu_s,
$
%\end{equation*}
where
$\displaystyle \kappa:=\frac{\tau_H}{-c_H}$ and $\displaystyle \mu_s:= \frac{\Gamma(s-\lambda_G)\Gamma(-\lambda_G\lambda_H-\lambdaM)}{\Gamma(-\lambda_G)\Gamma(s\lambda_H-\lambda_G\lambda_H-\lambdaM)}$.
\pagebreak

What is more, one has $\E(X_n^s)\sim \E(\fallfak{X_n}s)$ since we can express the raw moments by using the factorial moments and the Stirling numbers of the second kind:
\begin{equation}
\E(X_n^s)=\E\Big(\sum_{j=0}^{s}\Stir{s}{j}\fallfak{X_n}{j}\Big)=\sum_{j=0}^{s}\Stir{s}{j}\E(\fallfak{X_n}{j}).
\label{COMPSCHEMEconversion}
\end{equation}

\noindent 
Consequently, we obtain the moment convergence to the stated moment sequence:
\begin{equation*}
\frac{\E(X_n^s)}{n^{s\lambda_H} \kappa^s}\to \mu_s.
\end{equation*}
By Stirling's formula for the gamma function, one has
%\begin{equation}
 $\Gamma(z) = \Bigl(\frac{z}{e}\Bigr)^{z}\frac{\sqrt{2\pi }}{\sqrt{z}} \Bigl(1+ \O\left(\frac{1}{z}\right)\Bigr);$
%\label{StirlingGamma}
%\end{equation}
this entails that for $s\to\infty$ the moments satisfy 
\begin{equation*}
(\mu_s)^{-1/(2s)}\sim \sqrt{e^{1-\lambda_H} \lambda_H^{\lambda_H}} s^{\frac{H-1}{2}}\left(1+ \O\left(\frac{1}{s}\right)\right).
\end{equation*}
As $0 < \lambda_H < 1$, the divergence in Carleman's criterion~\cite[pp.~189--220]{Carleman23} is satisfied: 
\begin{equation}
\sum_{s=0}^{\infty}\mu_s^{-1/(2s)}=+\infty;
\label{eq:Carleman}
\end{equation}
consequently, the moment sequence $(\mu_s)_{s\in\N}$ characterizes a unique distribution. 
Now, the Fr\'echet--Shohat theorem (see their original article~\cite{FrSh1931} or the book by Loève~\cite[Sec.~11.4]{Loeve1977}) states that
if $\E[Y_n^r] \rightarrow \E[Y^r]$ for all $r\in \N$, where $Y$ is uniquely characterized by its moments, then $Y_n$ converges to $Y$ in distribution.
In our case, this implies the convergence in distribution of the random variable $\frac{X_n}{\kappa\cdot n^{\lambda_H}}$ to a random variable $X$ with moment sequence $(\mu_s)_{s\in\N}$.

Concerning the local limit theorem, we have to analyse %~\eqref{PXnk}: 
$\P\{X_n=k\}=\frac{g_k[z^n]H(z)^kM(z)}{f_n}$, 
for $k=x\cdot\kappa\cdot n^{\lambda_H}$, with $x$ in a compact subinterval of $(0,\infty)$. 
First, we note that by~\eqref{EqSA1} applied to $G(z)$ we directly obtain 
\begin{equation}
g_k\sim \frac{c_G}{\rho_G^k}\cdot \frac{(\kappa x)^{-\lambda_G-1}n^{-\lambda_H\lambda_G-\lambda_H}}{\Gamma(-\lambda_G)}.
\label{ExpansionG4}
\end{equation}
It remains to determine the asymptotics of $[z^n]H(z)^k M(z)$.
We have
\begin{equation*}
[z^n]H(z)^k M(z)=\frac{1}{2\pi i}\oint \frac{H(z)^k M(z)}{z^{n+1}}\, dz.
\end{equation*}
Introducing the point $A$ of coordinates $(\frac{1}{n}, \rho_H (1+\frac{\log^2 n}{n}))$, 
this Cauchy integral can be transformed into an integral over a larger contour in the Delta-domain (in blue in Figures~\ref{fig:deltadomain} 
and~\ref{fig:Hankel}). Then, setting $z=\rho_H(1+t/n)$ leads to an integral asymptotically concentrated on the Hankel contour $\Ce$ (which starts and ends at $+\infty$):
\begin{equation*}
[z^n]H(z)^kM(z)\sim \frac{\tau_H^k C_M}{\rho_H^n \ n^{1+\lambdaM}}\cdot \frac{1}{2\pi i}\int_{\Ce} (-t)^{\lambdaM} e^{-t- x (-t)^{\lambda_H}} \, dt.
\end{equation*}

\noindent Then, expanding $e^{-x (-t)^{\lambda_H}}$ leads to 
\begin{equation}\label{expanded}
[z^n]H(z)^kM(z)\sim \frac{\tau_H^k C_M}{\rho_H^n \ n^{1+\lambdaM}}\cdot \frac{1}{2\pi i}\int_\Ce e^{-t} \sum_{j\ge 0}\frac{(-x)^j}{j!}(-t)^{j\lambda_H+\lambdaM} \, dt,
\end{equation}
in which we recognize the following Hankel contour representation~\cite[Section~12.22]{WhittakerWatson1927}: 
\begin{equation*}
\frac{1}{\Gamma(-z)} = \frac{1}{2 i \pi} \int_\Ce (-t)^{z} e^{-t} \, dt.
\end{equation*}

Recall that we use, by analytic continuation, $1/\Gamma(m)=0$ for any integer $m<0$; this
avoids heavier expressions relying on Euler's reflection formula 
%$\Gamma(z)\Gamma(1-z)=\frac{\pi}{\sin(\pi z)}$
$\frac{1}{\Gamma(-z)}=-\frac{1}{\pi} \sin(\pi z) \Gamma(1+z)$.
\pagebreak

Now, combining the expansions of $f_n$ from~\eqref{ExpansionF2}, $g_k$~from~\eqref{ExpansionG4}, 
and term-wise integration in~\eqref{expanded},
we get the desired local limit theorem:
\begin{align}
%\P\{X_n=k\}= \frac{\Gamma(-\lambda_{G}\lambda_H-\lambdaM)}{\kappa n^{\lambda_H}\Gamma(-\lambda_G)}
 %\sum_{j\ge 0}\frac{(-1)^j x^{j-\lambda_{G}-1}}{j!} \frac{\sin\big(\pi(1\!+\!j\lambda_H\!+\!\lambdaM)\big)}{\pi} \Gamma(1+j\lambda_H+\lambdaM).
\P\{X_n=k\}
 &=\frac{g_k}{f_n}[z^n]H(z)^kM(z) \notag \\[-1.4mm]
 &\sim \frac{\Gamma(-\lambda_{G}\lambda_H-\lambdaM)}{\kappa n^{\lambda_H}\Gamma(-\lambda_G)}
 \sum_{j\ge 0}\frac{(-1)^j}{j! \Gamma(-j\lambda_H-\lambdaM)} x^{j-\lambda_{G}-1}. \label{LLT}
\end{align}

\begin{figure}[!t]
 \centering
 \includegraphics[height=2.65cm]{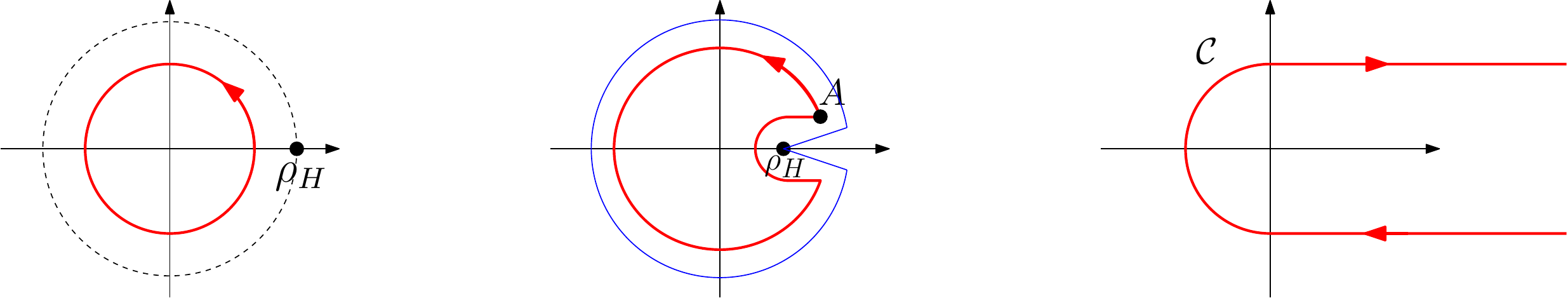}
 \caption{The integration contour of the Cauchy integral (left) 
 is transformed into a larger contour (in red in the middle image) inside a Delta-domain (in blue), becoming in the limit the Hankel contour $\Ce$ (right), allowing the identification of the gamma function.}
 \label{fig:Hankel}
\end{figure}

It remains to verify that $f_X(x)$ is the density function of $X$ with moments 
\begin{equation*}\mu_s=\frac{\Gamma(s-\lambda_G)\Gamma(-\lambda_G\lambda_H-\lambdaM)}{\Gamma(-\lambda_G)\Gamma(s\lambda_H-\lambda_G\lambda_H-\lambdaM)}.
\end{equation*}
Note that the gamma function is never zero and that its only singularities are simple poles at the negative integers.
Therefore $\mu_s$ (considered as a complex function of $s$)
has simple poles at $\rho_0=\lambda_G$ (if $\lambdaM \neq 0$, which entails $\lambda_G<0$ as we have a pure scheme)
 and at $\rho_j=\lambda_G-j$ (for $j\in\{1,2,\dots\}$).
Then, as $\mu_{s-1}$ is the Mellin transform of the density $f_X(x)$ of $X$,
an inverse Mellin computation\footnote{See~\cite[Appendix B.7]{FlaSe2009} for more on this method; see also~\cite[Theorem~5.4]{Jan2010} % with $\gamma=\gamma'=1-\lambda_H>0$,
and \cite[Equation~(6.12)]{Jan2010} for similar results on the class of functions with moments of Gamma type.}
implies that $f_X(x)$ is expressible in terms of the residues of $\mu_s$: 
\begin{equation*}
f_X(x)=\frac{1}{2 \pi i} \int_{\sigma-i\infty}^{\sigma+i\infty} \mu_{s-1} x^{-s} ds = \sum_{\substack{\text{$\sigma_j=1+\rho_j$ pole of $\mu_{s-1}$}\\ \sigma_j<\sigma}} \Res_{s=\sigma_j}(\mu_{s-1})x^{-\sigma_j},
\end{equation*}
which is valid for $x>0$, and where $\sigma=1+\rho_0+\epsilon$ is %any real 
in the fundamental strip of $ \mu_{s-1}$,
 and 
\begin{align*}
\Res_{s=\sigma_j}(\mu_{s-1})
&=\Res_{s=\rho_j}(\mu_s)
= \frac{(-1)^j}{j!}\cdot\frac{\Gamma(-\lambda_G\lambda_H-\lambdaM)}{\Gamma(-\lambda_G)\Gamma(-j\lambda_H-\lambdaM)}.
%\\&=\frac{\Gamma(-\lambda_G\lambda_H-\lambdaM)}{\pi \Gamma(-\lambda_G)}\cdot \frac{(-1)^n\sin\big(\pi(-n\lambda_H-\lambdaM)\big)}{n!}\Gamma(1+n\lambda_H+\lambdaM).
\end{align*}
%%%%% Maple %%%%%
%series(GAMMA(s-LG)/GAMMA(s*LH-LG*LH-LM),s=(LG-j),1) assuming j::posint;
%%%%%%%%%%%%%%
Summing for $j\geq 0$ (we can include $\rho_0$, even if it is not a pole,
as the corresponding residue is then 0) 
gives the same density function $f_X(x)$ as in the local limit theorem~\eqref{LLT}. 

Now, thanks to our probability distribution melting pot section~\ref{ssMP}, 
we identify the corresponding limit law by matching the parameters in the moments of 
the three-parameter Mittag-Leffler product given in Lemma~\ref{lem:MomentBetaStable}.
\end{proof}

Above, we have established the limit laws occurring for analytically {\em pure} composition schemes.
Next we deal with the analytically {\em confluent} and {\em degenerate} cases of Definition~\ref{def:pure}. 
They lead either to discrete distributions, or, more interestingly, to a mixture of discrete and continuous distributions. 
This generalizes the phenomenon observed in~\cite{BFSS2001}, where for $\lambda_H=\frac32$ the limit law also consists of a discrete part \textit{plus} a continuous part, a map-Airy distribution (this phenomenon also occurs for variants of 3-connected graphs; see~\cite{GimenezNoyRue2013}). 
%%see Th 5.1 in \url{https://arxiv.org/pdf/0907.0376.pdf}
The following two theorems explain the connection between this discrete part 
and Boltzmann distributions (usually used for random sampling; see, e.g., \cite{Boltz1,Boltz2}).
\pagebreak

\begin{theorem}[Extended composition scheme: degenerate case]
\label{TheExtendedDegenerate}
	Let a degenerate extended critical composition scheme $F(z,u)=G\big(u H(z)\big)M(z)$ with $0<\lambda_H<1$ be given.

	For $0<\lambda_G<1$ and $\lambda_M < \lambda_G \lambda_H$, the core size~$X_n$ converges for $k \geq 0$ and $n \to \infty$ to a Boltzmann distribution $\mathcal{B}_G(\rho_G)$ (see Definition~\ref{def:Boltzmann}):
		\begin{equation*}
		\P\{X_n=k\}\to \P\{\mathcal{B}_G(\rho_G)=k\}=\frac{g_k \rho_G^k}{G(\rho_G)}.
		\end{equation*}

	For $\lambda_G>1$, the core size~$X_n$ has for $k \geq 0$ and $n\to\infty$ the following behaviour:
	\begin{itemize}
		\item For $\lambda_M<\lambda_H$, the random variable $X_n$ converges to a Boltzmann distribution $\mathcal{B}_G(\rho_G)$:
		\begin{equation*}
		\P\{X_n=k\}\to \frac{g_k \rho_G^k}{G(\rho_G)}.
		\end{equation*}
		\item For $\lambda_M=\lambda_H$, the random variable $X_n$ converges to a convex combination of two discrete independent Boltzmann distributions:
		\begin{equation*}
		%X_n\claw 
		\text{Be}(p)\cdot \mathcal{B}_G(\rho_G) + (1-\text{Be}(p))\cdot \mathcal{B}_{G'}(\rho_G),
		\end{equation*}
		with 
		$p = \frac{c_M G(\rho_G)}{ c_M G(\rho_G) + c_H \tau_M G'(\rho_G) }$
		and where the Bernoulli distribution $\text{Be}(p)$ is independent of $\mathcal{B}_G(\rho_G)$ and $\mathcal{B}_{G'}(\rho_G)$. 
		We have
		\begin{equation*}
		\P\{X_n=k\}\to p\cdot \frac{g_k \rho_G^k}{G(\rho_G)} + (1-p)\cdot \frac{k g_k \rho_G^{k-1}}{G'(\rho_G)}. 
		\end{equation*}
			\item For $\lambda_M>\lambda_H$, the random variable $X_n$ converges to a Boltzmann distribution $\mathcal{B}_{G'}(\rho_G)$:
		\begin{equation*}
		\P\{X_n=k\}\to \frac{k g_k \rho_G^{k-1}}{G'(\rho_G)}. 
		\end{equation*}
	\end{itemize}
\end{theorem}

\begin{proof}
Let us start with the case $0<\lambda_G<1$. We study for arbitrary but fixed $k\in\N$ the probability $\P\{X_n=k\}$, as $n$ tends to infinity. 
Using~\eqref{PXnk} and applying the singular expansion of Lemma~\ref{lem:Fgenericasy} to $H(z)$ and $M(z)$, one gets
\begin{align}
\label{eqn:Boltz1}
\begin{aligned}
\P\{X_n=k\} &= g_k\frac{[z^n]H(z)^k\cdot M(z)}{f_n} \\
						&\sim \frac{g_k}{f_n} [z^n] \left( \left( \tau_H^k + k \tau_H^{k-1} c_{H} \big(1-\frac{z}{\rho_H}\big)^{\lambda_H}\right) \left( \tau_M + c_M(1-\frac{z}{\rho_H})^{\lambda_M} \right) \right).
\end{aligned}
\end{align}
As $\lambda_M<\lambda_G \lambda_H<\lambda_H$, the singular exponent $\lambda_M$ dominates the asymptotics
and applying singularity analysis~\eqref{EqSA1} leads to
\begin{equation*}
\P\{X_n=k\} \quad \sim \quad \frac{g_k}{f_n} [z^n] \tau_H^k c_M(1-\frac{z}{\rho_H})^{\lambda_M} \quad \sim \quad \frac{g_k \tau_H^k}{\tau_G}\quad=\quad \frac{g_k \rho_G^k}{G(\rho_G)}, 
\end{equation*}
where we simplified with the asymptotics $f_n\sim \tau_G c_M\frac{1}{\rho_H^{n}}\frac{n^{-\lambda_M-1}}{\Gamma(-\lambda_M)}$
coming from the expansion~\eqref{ExpansionF-c}, and where we used $\tau_H=\rho_G$ and $\tau_G=G(\rho_G)$ which follows by our assumptions on criticality.
\smallskip

Let us continue with the case $\lambda_G>1$.
From expansion~\eqref{ExpansionF-d} we see that we now have to distinguish three subcases:
$\lambda_M<\lambda_H$, $\lambda_M=\lambda_H$, or $\lambda_H<\lambda_M$.
The case $\lambda_M<\lambda_H$ is exactly the same as before.
The case $\lambda_M>\lambda_H$ is different, yet the results are derived analogously:
the asymptotics of $f_n$ depend now on $\lambda_H$ and are given by
\begin{equation*}
f_n \sim G'(\rho_G) \tau_M c_H \frac{1}{\rho_H^{n}}\frac{n^{-\lambda_H-1}}{\Gamma(-\lambda_H)}.
\end{equation*}
\pagebreak

\noindent  Then, we obtain from~\eqref{eqn:Boltz1} where now again the contribution from $H(z)$ dominates
\begin{equation*}
\P\{X_n=k\}\sim \frac{k g_k \tau_H^{k-1}}{ G'(\rho_G)}=\frac{ k g_k \rho_G^{k-1}}{G'(\rho_G)}.
\end{equation*} 

Finally, in the case $\lambda_M=\lambda_H$ the previous two contributions are mixed as the first two terms in~\eqref{ExpansionF-d} contribute:
The asymptotics of $f_n$ are given by 
\begin{equation*}
f_n \sim \left( G'(\rho_G) \tau_M c_H + c_M \tau_G \right) \frac{1}{\rho_H^{n}}\frac{n^{-\lambda_H-1}}{\Gamma(-\lambda_H)}.
\end{equation*}
Note that the coefficient is not zero, as both terms are negative: $G'(\rho_g),\tau_M,\tau_M>0$ and $c_H,c_M<0$ due to the sign property~\eqref{sign_property} for $0<\lambda_H<1$.
Then, we obtain from~\eqref{eqn:Boltz1} where again both contributions have to be taken into account
\begin{equation*}
\P\{X_n=k\}
\sim \frac{c_M g_k \rho_G^k + \tau_M c_H k g_k \rho_G^{k-1}}{ G'(\rho_G) \tau_M c_H + c_M \tau_G }
=p \cdot \frac{g_k \rho_G^k}{G(\rho_G)} + (1-p) \cdot \frac{ k g_k \rho_G^{k-1}}{G'(\rho_G)},
\end{equation*}
where $p = \frac{c_M G(\rho_G)}{ c_M G(\rho_G) + c_H \tau_M G'(\rho_G) } \in (0,1)$ 
by the sign property~\eqref{sign_property}. 
We thus get a linear combina\-tion of two Boltzmann distributions, 
weighted by a Bernoulli random variable $\text{Be}(p)$. 
\end{proof}

\begin{theorem}[Extended composition scheme: confluent case]
\label{TheExtendedConfluent}
	Let a confluent extended critical composition scheme $F(z,u)=G\big(u H(z)\big)M(z)$ with $0 < \lambda_H < 1$ be given (i.e., $0 < \lambda_G < 1$ and $\lambda_M=\lambda_G \lambda_H$). 
	Then the core size~$X_n$ is a convex combination of a Boltzmann distribution $\mathcal{B}_G(\rho_G)$ 
and an asymptotically continuous random variable $Z_n$:
		\begin{equation*}
		X_n \law \text{Be}(p)\cdot \mathcal{B}_G(\rho_G) + (1-\text{Be}(p))\cdot Z_n,\qquad \frac{Z_n}{\kappa\cdot n^{\lambda_H}}\cmom X,
		\end{equation*}
		where $\kappa$ and the limit law 
		$X\law \ML(\lambda_{H},-\lambda_{G}\lambda_{H})$ 
		are the same as in Remark~\ref{CoExtended1}, 
		and where $\text{Be}(p)$ (with $p = \frac{c_M G(\rho_G)}{c_M G(\rho_G) + \tau_M c_G (-c_H/\rho_G)^{\lambda_G}}$) is
independent of $\mathcal{B}_G(\rho_G)$, $Z_n$, and $X$. 
\end{theorem}

\begin{proof}
The proof follows the same lines as the one of Theorem~\ref{TheExtendedDegenerate}.
We start from expansion~\eqref{ExpansionF-c}.
Due to $\lambda_M=\lambda_G \lambda_H$ both terms contribute, and we get 
\begin{align}
\label{ExpansionF3}
f_n \sim \frac{\tau_M c_G (-c_H/\rho_G)^{\lambda_G} + \tau_G c_M}{\rho_H^{n}}\frac{n^{-\lambda_M-1}}{\Gamma(-\lambda_M)}.
\end{align}
Then we extract the asymptotics from~\eqref{eqn:Boltz1}, where now the contributions from $M(z)$ dominate, as $\lambda_M<\lambda_G$:
\begin{equation}\label{Pmiss}
\P\{X_n=k\} = \frac{c_M g_k \tau_H^k}{\tau_M c_G (-c_H/\rho_G)^{\lambda_G} + \tau_G c_M} +o(1)
= p \cdot \frac{g_k \rho_G^k}{G(\rho_G)} +o(1), 
\end{equation}
Using the sign property~\eqref{sign_property} we get $p \in (0,1)$.
So, for large $n$, $X_n$ behaves with probability $p$ like $\mathcal{B}_G(\rho_G)$, 
but what happens with probability $1-p$? Where is this missing mass in~\eqref{Pmiss}? 
For sure, it is sneakily spread in $\sum_k o(1)$: It turns out that more and more mass is distributed in the range $k \sim \Theta(\kappa n^{\lambda_H})$, leading to an asymptotic continuous distribution therein.

In order to identify this distribution, we compute the factorial moments of $X_n$ like in the proof of Theorem~\ref{TheExtended}.
We again use the singular expansions of $G^{(s)}\big(H(z)\big)$, $H(z)^s$, and $H(z)^s G^{(s)}\big(H(z)\big) M(z)$. 
Formula~\eqref{eq:pureasymptmoments} holds verbatim with $C_M = \tau_M$.
The big difference lies now in the asymptotics of $f_n$: It is given by~\eqref{ExpansionF2} in the pure case and by~\eqref{ExpansionF3} in the confluent case.
Thus, after rescaling~\eqref{eq:pureasymptmoments} by $f_n$, the factorial moments have the same shape, 
but with an additional prefactor $1-p$:
$\E(\fallfak{X_n}s)\sim (1-p)\cdot \E(\fallfak{Z_n}s)$, which proves the claim.
\end{proof}

\begin{remark}[Physical interpretation of the bimodal behaviour]
As illustrated by Table~\ref{tab:extended} on page~\pageref{tab:extended}, the confluent case gives a bimodal distribution.
The first mode is dictated by small values of $k$, where a Boltzmann distribution is dominant, 
while for larger values of $k\approx n^{\lambda_H}$, a second mode appears, which is associated with a continuous Mittag-Leffler distribution. 
This phenomenon explains the following behaviour:
If one picks an atom at random in a structure of size $n$, 
then with probability $p$ it lies in a large $\mathcal{H}$-block of size $\Theta(n)$ and with probability $1-p$ in a smaller $\mathcal{H}$-block of size $\Theta(n^{1-\lambda_H})$.
\end{remark}

In the next section, we refine these considerations 
by having a closer look at the distribution of the $\mathcal{H}$-blocks of any given size. 

\section{Size-refined composition scheme\label{SecRefined}}

In this section, we give the limit laws for the profile of combinatorial structures given by a critical composition scheme.
We focus here on schemes which are analytically pure (see Definition~\ref{def:pure}), while 
we handle the confluent and degenerate cases in our companion article~\cite{BanderierKubaWallner2021b} 
(as they require additional technical details and different families of limit laws, which also pop up for $\lambda_H>1$).
The profile is captured by the size-refined composition scheme~\eqref{eq:scheme_size-size-refined}. 
As we see in the theorem below, we get three successive asymptotic régimes,
each leading to its own limit law. Two of these limit laws are expressible 
in terms of the three-parameter Mittag-Leffler distribution of Theorem~\ref{TheExtended}
(see also Definition~\ref{ex:genmittagleffler} and Definition~\ref{def:BetaMittagLeffler}).

\begin{theorem}[Mixed Poisson limit behaviour for the size-refined scheme]
\label{TheRefined}
Consider a size-refined pure critical composition scheme
\begin{equation*}
F(z,v)=G\big(H(z) - (1-v)h_j z^j\big)M(z), \text{ with $j\in\N$.}
\end{equation*}
 Let $\mppar_{n,j}=\frac{\rho_H^{j}}{-c_H} h_j n^{\lambda_H}$, and $X$ be the three-parameter Mittag-Leffler distribution of Theorem~\ref{TheExtended}:
\begin{align*}
 X \law \BML(\alpha, \beta, \gamma),
\end{align*}
where $\alpha=\lambda_H$, $\beta=-\lambda_G\lambda_H$, and $\gamma=-\min(0,\lambda_M)$. 

 Then, the random variable $X_{n,j}$, which counts the number of $\HH$-components of size $j$
 possesses three successive asymptotic régimes, with a phase transition at $j=\Theta(n^{\frac{\lambda_H}{1+\lambda_H}})$:
 \begin{enumerate}[(i)]
 \item\label{it:limita} For $j\ll n^{\frac{\lambda_H}{1+\lambda_H}}$, we have $\mppar_{n,j}\to+\infty$ and convergence in distribution and in moments: 
 \begin{align*}
 \frac{X_{n,j}}{\mppar_{n,j}} \cmom X.
 \end{align*}
 \item\label{it:limitb} For $j\sim r\cdot n^{\frac{\lambda_H}{1+\lambda_H}}$, $r\in (0,\infty)$, we have $\mppar_{n,j}\to\mppar$ with 
	$
	\mppar=r^{-\frac{\lambda_H}{1+\lambda_H}}\cdot\frac{1}{-\Gamma(-\lambda_H)}
	$
	and 
	convergence in distribution and in moments towards a mixed Poisson distribution: 
 \begin{align*}
 X_{n,j} \cmom \MPo(\mppar X).
 \end{align*}
 \item \label{it:limitc} For $j\gg n^{\frac{\lambda_H}{1+\lambda_H}}$, we have $\mppar_{n,j}\to 0$, 
and $X_{n,j}$ converges to a Dirac distribution at 0. 
	\end{enumerate}
\end{theorem}
\pagebreak

\begin{remark}[Phase transition I]
The intuition behind the phase transition is as follows:
In the limit $n \to \infty$, there are many small ($j\ll n^{\frac{\lambda_H}{1+\lambda_H}}$), some giant ($j\sim r n^{\frac{\lambda_H}{1+\lambda_H}}$), 
and no super-giant ($j\gg n^{\frac{\lambda_H}{1+\lambda_H}}$) $\HH$-components of size $j$. 
 It is interesting to compare this situation with the birth of the giant component in Erdős--R\'enyi random graphs~\cite{JansonKnuthLuczakPittel1993}.
		Note that the case $j\in\N$ fixed (i.e., independent of $n$ as $n$ tends to infinity) falls into the case $j\ll n^{\frac{\lambda_H}{1+\lambda_H}}$.
\end{remark}

\begin{remark}[Phase transition II]
	For the often observed case of a square-root singularity of $H(z)$ (i.e., $\lambda_H=\frac12$), we reobtain the critical range $j=\Theta(n^{1/3})$, which was already observed in the mixed Poisson Rayleigh distributions in~\cite{KuPa2014}.
	 Furthermore, as one has 
	\begin{equation*}\E(\fallfak{X_{n,j}}s)=\mppar_{n,j}^s\cdot \E(X^s) \cdot (1+o(1)),\end{equation*}
	this offers en passant a link between the $\mppar_{n,j}$'s and the rescaling factor $\kappa n^{\lambda_H}$ in Theorem~\ref{TheExtended}:
	\begin{equation*}
		\sum_{j \geq 1} \mppar_{n,j} = \sum_{j \geq 1} \frac{\rho_H^{j}}{-c_H} h_j n^{\lambda_H} = \frac{H(\rho_H)}{-c_H} n^{\lambda_H} = \frac{\tau_H}{-c_H} n^{\lambda_H} = \kappa n^{\lambda_H}.
	\end{equation*}
	This link can be seen as an asymptotic avatar of the combinatorial relation $\sum_{j\geq 1} X_{n,j}=X_n$, implying 
	\begin{flalign*} && \sum_{j\geq 1} \E(X_{n,j})=\E(X_n).&& \myqedhere
          \end{flalign*}
\end{remark}

In order to prove Theorem~\ref{TheRefined}, we first need the following lemma
concerning the convergence of mixed Poisson distributions.

\begin{lem}[Factorial moments and limit laws of mixed Poisson type~\cite{KuPa2014}]
\label{COMPSCHEMElemmaMixedPoisson}
Let $(X_n)_{n\in\N}$ denote a sequence of random variables, whose factorial moments are asymptotically of \emph{mixed Poisson type}, i.e., 
they satisfy for $n \to \infty$ the asymptotic expansion
\begin{equation*}
\E(\fallfak{X_n}s)=\mppar_n^s \cdot \mu_s\cdot (1 + o(1)),\quad s\ge 1,
\end{equation*}
with $\mu_s\ge 0$ and $\mppar_n>0$. Furthermore, assume that the moment sequence $(\mu_s)_{s\in\N}$ determines a unique distribution 
$X$ satisfying Carleman's condition. Then, the following limit distribution results hold:
\begin{itemize}
\item[(i)] If $\mppar_n\to\infty$,
the random variable $\frac{X_n}{\mppar_n}$ converges in distribution, with convergence of all moments, to $X$. 

\item[(ii)] If $\mppar_n\to\mppar \in (0,\infty)$, 
the random variable $X_n$ converges in distribution, with convergence of all moments, to a mixed Poisson distributed random variable $Y\law \MPo(\mppar X)$.

\item[(iii)] If $\mppar_n\to 0$, 
the random variable $X_n$ converges to a Dirac distribution: $X_n\claw 0$.
\end{itemize}
\end{lem}

%\begin{remark}
%The second and third case could be grouped together, since for $\mppar=0$ we have $Y\law \MPo(0)\law 0$. 
%Furthermore, in the third case, for positive random variables $(X_n)_{n\in\N}$ the assumptions can be relaxed to simply
%$\E(X_n)\to 0$, without requiring the specific structure of the moments. 
%Note further that the discrete random variable $Y\law \MPo(\mppar X)$ converges, after scaling,  to its mixing distribution $X$: 
%One has $\frac{Y}{\mppar}\underset{\mppar \to \infty}{\clawlong}X$, with convergence of all moments.
%\end{remark}

\begin{proof}[Proof of Theorem~\ref{TheRefined}]
The factorial moments $\E(\fallfak{X_{n,j}}s)=\sum_{k\ge 0}\P\{X_{n,j}=k\}\fallfak{k}s$ of $X_{n,j}$ are obtained from $F(z,v)$ by repeated differentiation and evaluation at $s=1$:
\begin{equation*}
\E(\fallfak{X_{n,j}}s)= \frac{[z^n] \partial_v^s(F)(z,1)}{[z^n] F(z,1)}=h_j^s\frac{[z^{n-j\cdot s}]G^{(s)}\big(H(z)\big)M(z)}{f_n}.
\end{equation*}
We already know the asymptotics of $f_n$ from~\eqref{ExpansionF2}. 
For fixed $j$ we can proceed by extraction of coefficients, 
while for $j=j(n)$ tending to infinity, the asymptotic expansion of $h_j$ 
follows by singularity analysis applied to $H(z)$ (see Equation~\eqref{EqSA1}):
\begin{equation}
h_j=\frac{c_H}{\rho_H^j}\cdot \frac{j^{-\lambda_H-1}}{\Gamma(-\lambda_H)}\cdot \left(1+o(1)\right).
\label{COMPSCHEMEExpanHj}
\end{equation}
What is more, one has
\begin{equation}
G^{(s)}\big(H(z)\big)M(z) \sim
(-1)^{s} c_M c_{G}\rho_G^{-\lambda_G}\fallfak{\lambda_G}{s}(-c_H)^{\lambda_G-s}\Big(1-\frac{z}{\rho_H}\Big)^{\lambda_H\lambda_{G}-s\lambda_H+\lambdaM}.
\label{COMPSCHEMEExpanGAbleit}
\end{equation}
This implies that 
$X_{n,j}$ has factorial moments of mixed Poisson type:
\begin{equation*}
\E(\fallfak{X_{n,j}}s)\sim h_j^s (-c_H)^{-s}\rho_H^{js} \cdot \mu_s \cdot n^{s\lambda_H}
\text{\quad with \quad }
\mu_s:=\frac{\Gamma(s-\lambda_G)\Gamma(-\lambda_{G}\lambda_H-\lambdaM)}{\Gamma(-\lambda_G)\Gamma(-\lambda_H\lambda_{G}+s\lambda_H-\lambdaM)}.
\end{equation*}
We already observed that the moment sequence $(\mu_s)_{s\in\N}$ determines a unique distribution, 
by Carleman's criterion~\eqref{eq:Carleman}.
Thus, the limit laws follow by using Lemma~\ref{COMPSCHEMElemmaMixedPoisson}. 

Finally, the critical growth range is obtained via
the closed-form expression for~$\mppar_{n,j}$ (in which one inserts the expansion~\eqref{COMPSCHEMEExpanHj}):
Indeed, $\mppar_{n,j}$ is of growth order
$\frac{n^{\lambda_H}}{j^{1+\lambda_H}\Gamma(-\lambda_H)}$
and converges to a nonzero constant if and only if $j(n)\sim r\cdot n^{\frac{\lambda_H}{1+\lambda_H}}$, 
therefore the critical growth range is $\Theta(n^{\frac{\lambda_H}{1+\lambda_H}})$. 
Collecting all contributions from $\mppar_{n,j}$ (for $j=o(n)$ tending to infinity)
gives the constant~$\mppar$. The Dirac case is now finally obtained by an additional analysis
of the expected value in the remaining range 
$j\gg n^{ \frac{\lambda_H}{1+\lambda_H}}$.
There, we directly obtain $\E(X_{n,j})\to 0$, 
which proves the stated result. 
\end{proof}

\begin{remark}[A probabilistic approach]
Note that our Theorem~\ref{TheRefined} was recently revisited by Stufler in~\cite[Section~3.4]{Stufler2022} with a probabilistic point of view.
He also describes a nice link with a point process giving access to the distribution of the largest components; see~\cite[Proposition~3.16]{Stufler2022}.
Compare also with the results of Gourdon~\cite{Gourdon1998}. 
\end{remark}

Next we turn to the dependence between the number of $\HH$-components of size~$j_1$ and $j_2$, determining the covariance and the correlation coefficient. 

\begin{theorem}
\label{TheRefinedCov}
In a size-refined pure critical composition scheme, the covariance of
the random variables
$X_{n,j_1}$ and $X_{n,j_2}$, counting the
number of $\HH$-components of size $j_1$ and $j_2$
(with $j_1,j_2=o(n)$), 
satisfies
\begin{equation}\label{covariance}
\cov(X_{n,j_1},X_{n,j_2})\sim \mppar_{n,j_1}\mppar_{n,j_2}\cdot \V(X),
\end{equation}
where $\mppar_{n,j}=\frac{\rho_H^{j}}{-c_H} h_j n^{\lambda_H} >0$ and $X \law \BML(\alpha, \beta, \gamma)$ denotes the three-parameter Mittag-Leffler distribution from Theorem~\ref{TheExtended}. 
Furthermore, the correlation coefficient between 
$X_{n,j_1}$ and $X_{n,j_2}$ satisfies 
\begin{align*}
\rho(X_{n,j_1},X_{n,j_2}) \sim 
\frac{1}{\sqrt{1+\frac{\E(X)}{\mppar_1}}} \frac{1}{\sqrt{1+\frac{\E(X)}{\mppar_2}}},
\end{align*}
where, for $k=1,2$, one has $\mppar_k := \lim_{n} \mppar_{n,{j_k}}$ and $\sqrt{1+\frac{\E(X)}{\mppar_k}}\sim 1$
if $j_k \ll n^{\frac{\lambda_H}{1+\lambda_H}}$.
\end{theorem}
 
\begin{remark}
\label{RemarkRefinement1}
We observe that for small $j_1,j_2$ (e.g., if $j_1,j_2 =O(1)$), the random variables are asymptotically highly correlated: $\rho(X_{n,j_1},X_{n,j_2})\sim 1$.
\end{remark}
\pagebreak

\begin{proof}[Proof of Theorem~\ref{TheRefinedCov}]
The combinatorial scheme
\begin{equation*}
\mathcal{F}=\mathcal{G}\big(\mathcal{H}_{\neq j_1,j_2} + v_1\mathcal{H}_{= j_1}+ v_2\mathcal{H}_{= j_2}\big)\times \mathcal{M}
\end{equation*}
(where one takes distinct sizes $1\le j_1<j_2$) directly translates into 
\begin{equation*}
F(z;v_1,v_2)=G\big(H(z) - (1-v_1)h_{j_1} z^{j_1}- (1-v_2)h_{j_2} z^{j_2}\big) M(z).
\end{equation*}
Accordingly, the random variables $X_{n,j_1}$ and $X_{n,j_2}$ 
have the joint distribution 
\begin{equation*}
\P\{X_{n,j_1}=k_1,X_{n,j_2}=k_2\}= \frac{[z^n v_1^{k_1}v_2^{k_2}]F(z;v_1,v_2)}{[z^n]F(z;1,1)}.
\end{equation*}
%\footnote{Joint moments are the moments of a joint distribution. Some authors call them mixed moments, but we prefer here the terminology ``joint moments'',
%to avoid any confusion with the mixed Poisson and other mixture distributions.}

We already know the asymptotics of $f_n=[z^n]F(z;1,1)$, given in~\eqref{ExpansionF2}.
We get by differentiation and evaluation at $v_1=v_2=1$
\begin{equation*}
\E(X_{n,j_1}X_{n,j_2})=\frac{[z^n]\partial_{v_1}\partial_{v_2}(F)(z;1,1)}{[z^n]F(z;1,1)} = h_{j_1}h_{j_2}\frac{[z^{n-j_1-j_2}]G''(H(z))M(z)}{f_n}.
\end{equation*}
The asymptotics of $h_{j_1}$ and $h_{j_2}$ are given in~\eqref{COMPSCHEMEExpanHj}.
The singular expansion of $G''(H(z))M(z)$ is a special case of~\eqref{COMPSCHEMEExpanGAbleit}, 
so we obtain for $j_1,j_2=o(n)$:
\begin{equation*}
\E(X_{n,j_1}X_{n,j_2})\sim h_{j_1}h_{j_2}c_H^{-2}\rho_H^{j_1+j_2}\cdot\E(X^2)\cdot n^{2\lambda_H}.
\end{equation*}
Hence, using the explicit form of $\E(X)$ in~\eqref{MomentBetaStable}, we obtain for the covariance:
\begin{align*}
\cov(X_{n,j_1},X_{n,j_2})&= \E(X_{n,j_1}X_{n,j_2})-
\E(X_{n,j_1})\E(X_{n,j_2})\\
&\sim\frac{h_{j_1}h_{j_2}\rho_H^{j_1+j_2}}{c_H^2}\big(\E(X^2)-\E(X)^2\big)n^{2\lambda_H}.
\end{align*}
 By $\V(X)=\E(X^2)-\E(X)^2$ and the definition of
$\mppar_{n,j}$
, this implies~\eqref{covariance}.
For the correlation coefficient, we observe that Theorem~\ref{TheRefined} together with Lemma~\ref{COMPSCHEMElemmaMixedPoisson} implies
$\V(X_{n,j})\sim \mppar_{n,j}^2\V(X)+ \mppar_{n,j}\E(X)$.
Collecting all contributions, this implies that
\begin{align*}
\rho(X_{n,j_1},X_{n,j_2}) 
&= \frac{\cov(X_{n,j_1},X_{n,j_2})}{\sqrt{\V(X_{n,j_1})} \sqrt{\V(X_{n,j_2})}} 
\sim \frac{1}{ \sqrt{1+\frac{\E(X)}{\mppar_{n,j_1}}}\sqrt{1+\frac{\E(X)}{\mppar_{n,j_2}}} }. 
\end{align*}
Taking the limit of this expression for $n$ going to infinity concludes the proof.
\end{proof}

Before to present in Section~\ref{sec:outlook} a multivariate generalization of these results,
we now list several applications to a variety of combinatorial structures.

\pagebreak

\section{Applications and examples} \label{sec:examples}
Our main results, Theorems~\ref{TheExtended} and~\ref{TheRefined}, 
can be readily applied to the problems considered by Drmota and Soria~\cite{DS97}, Flajolet and Sedgewick~\cite{FlaSe2009}, Dumas, Flajolet and Puyhaubert~\cite{FlaDumPuy2006}, Janson~\cite{Jan2006b}, Meir and Moon~\cite{MeiMoo1970}, 
Panholzer and Seitz~\cite{PanholzerSeitz2011} (see also~\cite{KuPa2014} for many additional pointers to the literature)
once the required singular expansions of the involved generating functions are established. This includes returns to record-subtrees in Cayley trees, edge-cutting in Cayley trees, returns to zero in Dyck paths, 
cyclic points and trees in graphs of random mappings, all leading to (mixed Poisson) Rayleigh laws, as well as block sizes in $k$-Stirling permutations. 

In the following we discuss several new results for the distribution of  different statistics such as returns to zero and sign changes in walks with arbitrary steps, 
the number of subtrees satisfying some constraint in different fundamental families of trees,
as well as the table sizes in the Chinese restaurant process,
and the evolution of the number of balls in triangular urn models.
These examples illustrate that composition schemes 
$\mathcal F= \mathcal G(\mathcal H) \times \mathcal M$ lead universally 
to three-parameter Mittag-Leffler distributions (Definition~\ref{def:BetaMittagLeffler}),
which simplify into two-parameter Mittag-Leffler distributions
if $M(z)$ has a nonnegative singular exponent.

\subsection{Core size of supertrees and a confluent example}
\label{sec:Supertrees}

Let $\Ce$ denote the family of plane trees (i.e., trees with all arities allowed and embedded into the plane) defined by
\begin{equation*}
\Ce = \Z\times\Seq(\Ce),\quad \text{ which translates to } \quad C(z)=\frac{1-\sqrt{1-4z}}{2}.
\end{equation*}
Then, following~\cite[pp.~412--414, 714]{FlaSe2009},
we consider supertrees, or ``trees of trees'', defined by
\begin{equation*}
\mathcal{K}=\Ce\Big((\Z_\text{red}+\Z_\text{blue})\times \Ce\Big),\quad \text{ which translates to } \quad K(z)=C\big(2zC(z)\big).
\end{equation*}
Note that this is a critical scheme as one has $\rho_C=1/4$ and $\tau_C=C(\rho_C)=1/2$. 

The tree family~$\mathcal{K}$ corresponds to trees where onto each node we graft a red or blue tree; 
one can also draw them like done in Figure~\ref{fig:supertrees}.
By Lagrange inversion, these supertrees~$\mathcal{K}$ are thus enumerated by a neat combinatorial sum:
\begin{equation*}
K_n=\sum_{k=1}^{\lfloor n/2 \rfloor}\frac{2^k}{n-k}\binom{2k-2}{k-1}\binom{2n-3k-1}{n-k-1},
\end{equation*}
thus the sequence $K_n$ for $n\geq 2$ starts like $2,2,8,18,64,188,656,2154,\dots$, constituting sequence \href{https://oeis.org/A168506}{A168506} in the \textsc{Oeis}
(On-Line Encyclopedia of Integer Sequence: \href{https://oeis.org/}{https://oeis.org}). 
%\footnote{\textsc{Oeis} stands for the On-Line Encyclopedia of Integer Sequences, accessible via \url{https://oeis.org}.}. 

\begin{figure}[b]
\begin{center}
\includegraphics[height=4.95cm]{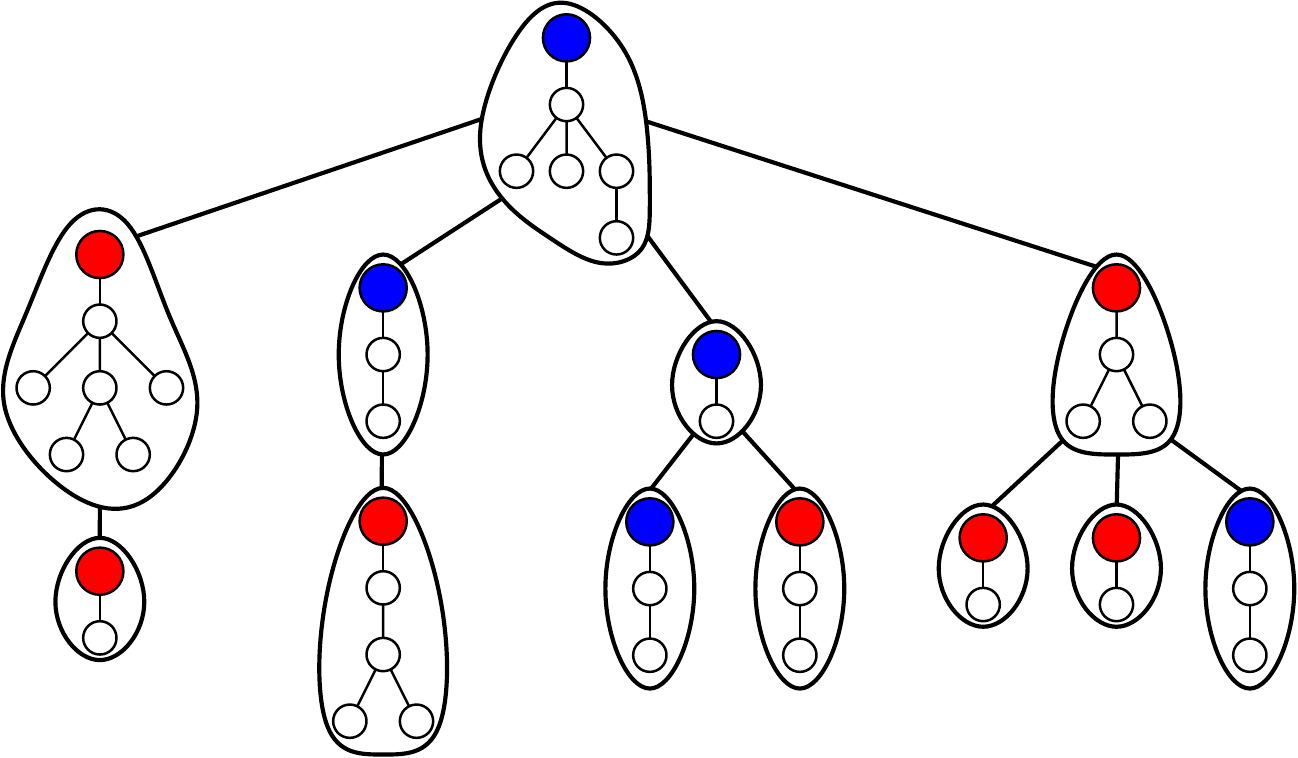} 
\end{center}
\caption{A bicoloured supertree is a ``tree of trees'': Each node (here drawn in a potato shape) of some initial plane tree is replaced by a red or a blue node to which one attaches another plane tree.
}
\label{fig:supertrees}
\end{figure}

\pagebreak
By the Laplace method or by singularity analysis, this directly leads to the asymptotic expansion
$
K_n\sim\frac{4^n}{8\Gamma(3/4)n^{5/4}}.
$
(See also DeVries~\cite{DeVries2011,PemantleWilson2013} for an alternative approach to this expansion via multivariate analysis.)
This asymptotic behaviour is noteworthy, because one sees here an unusual occurrence of the exponent $-\frac54$, 
while most tree-like structures in combinatorics usually involve the exponent $-\frac32$. 
In fact, one could similarly define super-supertrees, super-super-supertrees, and so on, by further iterations of the critical scheme: 
$\Ce_{k+1} = \Ce_{k}\Big( 2 \Z \Ce_k\Big)$ with $\Ce_0=\Ce$. 
This gives an interesting family of combinatorial structures
whose asymptotic enumeration involves a dyadic exponent $-1-1/2^{k+1}$;
see~\cite{BanderierDrmota2015} for a complete characterization 
of the possible singular exponents for \textit{$\N$-algebraic functions} (i.e., generating functions of any structure which can be generated by a context-free language).

With respect to supertrees, the critical scheme is 
\begin{equation*}\label{supertreesGHK}
K(z)=G(H(z)),\quad G(z)=C(z), \quad H(z)=2zC(z),
\end{equation*}
where $H(z)$ has the following Puiseux expansion at $z\sim\frac14$:
\begin{equation*}
H(z)\sim \frac14 - \frac14\sqrt{1-4z}.
\end{equation*}
Now, we can study the core size~$X_n$ via the bivariate generating function
\begin{equation*}
K(z,u)=C\big(u\cdot 2z C(z)\big),
\end{equation*}
as well as the number of $\mathcal{H}$-components of size $j$, as captured by 
\begin{equation*}
K(z,v)=C\big(2z C(z) + (v-1)2c_{j-1}z^j\big).
\end{equation*}
We can then apply our main Theorems~\ref{TheExtended} and~\ref{TheRefined} 
(with $\lambda_G=\lambda_H=1/2$, $\tau_H=\frac14$, $c_H=-\frac14$, and $\kappa=1$). This directly gives
the following corollaries.
\begin{coroll}\label{cor6.1}
The core size~$X_n$ in supertrees of size $n$ has factorial moments 
\begin{equation*}
\E(\fallfak{X_n}s)\sim n^{s/2}\cdot \mu_s,\quad
\quad \mu_s=\frac{\Gamma(s-\frac12)\Gamma(-\frac14)}{\Gamma(-\frac12)\Gamma(\frac{s}2-\frac14)}.
\end{equation*}
The scaled random variable $X_n/n^{1/2}$ converges in distribution with convergence of all moments to
a $c=-1/2$ moment-tilted stable distribution (a tilted Mittag-Leffler distribution of index~$1/2$):
\begin{equation*} 
\frac{X_n}{n^{1/2}}\cmom X, \qquad \text{ with } \qquad 
X\law \tilt_{-\frac{1}{2}}\left(\ML\left(\frac{1}{2}\right)\right)
\law \ML\left(\frac{1}{2},-\frac{1}{4}\right).
\end{equation*}
Moreover, we have the local limit theorem
\begin{equation*}
\P\{X_n=x\cdot n^{1/2}\}\sim \frac{1}{n^{1/2}}\cdot f_X(x),
\end{equation*}
with $f_X(x)$ denoting the density of the random variable $X$.
\end{coroll}
Note that by Legendre's duplication formula one has
$
\mu_s
=2^s \Gamma(\frac{s}2+\frac14)/\Gamma(\frac14),
$
so the random variable $X$ can also be seen as equal in law to 
the chi distribution $\chi(\frac12)$ of parameter~$\frac12$, 
which is itself a generalized gamma distribution~\cite{BanderierMarchalWallner2020}.
%%% Interesting dictionary %%%
%$\chi(1) = HN(1)$,
%$\chi(2)=Rayleigh(1)$, 
%$\chi(3)=Maxwell(1)$, 
%$\chi(4)$ appears in a later example. 
%%%%%%%%%%%%%
\pagebreak

\begin{coroll}
The number of coloured trees of size $j-1$ in supertrees of size $n$ has factorial moments of mixed Poisson type given by
\begin{equation*}\label{cor6.2}
	\E(\fallfak{X_{n,j}}s)=\mppar_{n,j}^s\cdot \mu_s(1+o(1)),
\end{equation*}
with 
\begin{equation*}
\mppar_{n,j}=2\cdot (\frac14)^{j-1} c_{j-1}\cdot n^{1/2}
\end{equation*}
 and mixing distribution $X=\chi(\frac12)$ with $\E(X^s)=\mu_s$.

\smallskip
Furthermore, the random variable $X_{n,j}$ possesses the three successive asymptotic régimes
of Theorem~\ref{TheRefined}, with a phase transition at $j=\Theta(n^{1/3})$.
\end{coroll}

\smallskip

We note in passing that a very similar result holds for the family of binary supertrees $\mathcal{S}$, occurring in Bousquet-M\'elou's study~\cite{Bou2006} of the integrated super-Brownian excursion. 
This family is defined in terms of the family of complete binary trees $\mathcal{B}$:
\begin{equation*}
\mathcal{S}=\mathcal{B}(\Z\times \mathcal{B}),\qquad \mathcal{B}=\Z+\Z\times\mathcal{B}\times\mathcal{B},
\end{equation*}
with initial values of $S_n$ given by $1,1,3,8,25,80,267,911,\dots$, constituting sequence \href{https://oeis.org/A101490}{A101490} in the \textsc{Oeis}. 
These functional equations lead to $B(z)=\frac{1-\sqrt{1-4z^2}}{2z}$, 
and to a critical scheme leading to limit laws similar to the ones of supertrees in Corollaries~\ref{cor6.1} and~\ref{cor6.2}.

\medskip

{\bf A confluent example.} Next, we consider pairs of supertrees in which we mark the number of $\mathcal H$-trees in the first part of the pair. 
This translates to the scheme $F(z,u)=G(u H(z)) K(z)$ (with $G$, $H$, and~$K$ defined as in~\eqref{supertreesGHK}), which is confluent because we have
$\lambda_H\lambda_G=\lambda_M=\frac{1}{4}$. 
This case is interesting as $\P\{X_n=k\}$ (the probability that a pair of supertrees has $k$ $\mathcal H$-trees in its first part) 
converges to the sum of a discrete law and a continuous law. 
More precisely, as given by Theorem~\ref{TheExtendedConfluent}, and illustrated in Figure~\ref{fig:supertrees_confluent}, the limit is a linear combination of the Boltzmann distribution $\mathcal{B}_C(\frac{1}{4})$\footnote{$C$ denoting the Catalan generating function, it could be natural to call this Boltzmann distribution 
$\mathcal{B}_C(\frac{1}{4})$ the {\em Catalan distribution}, but this name is already used for some other distributions popping up in random matrix theory, 
and having moments related to the Catalan numbers, like the Marchenko--Pastur distribution.}
and the two-parameter Mittag-Leffler distribution~$\ML\left(\frac{1}{2},-\frac{1}{4}\right)$.

\begin{figure}[!hb]
\begin{center}\hspace{-2mm}
\includegraphics[width=0.32\textwidth]{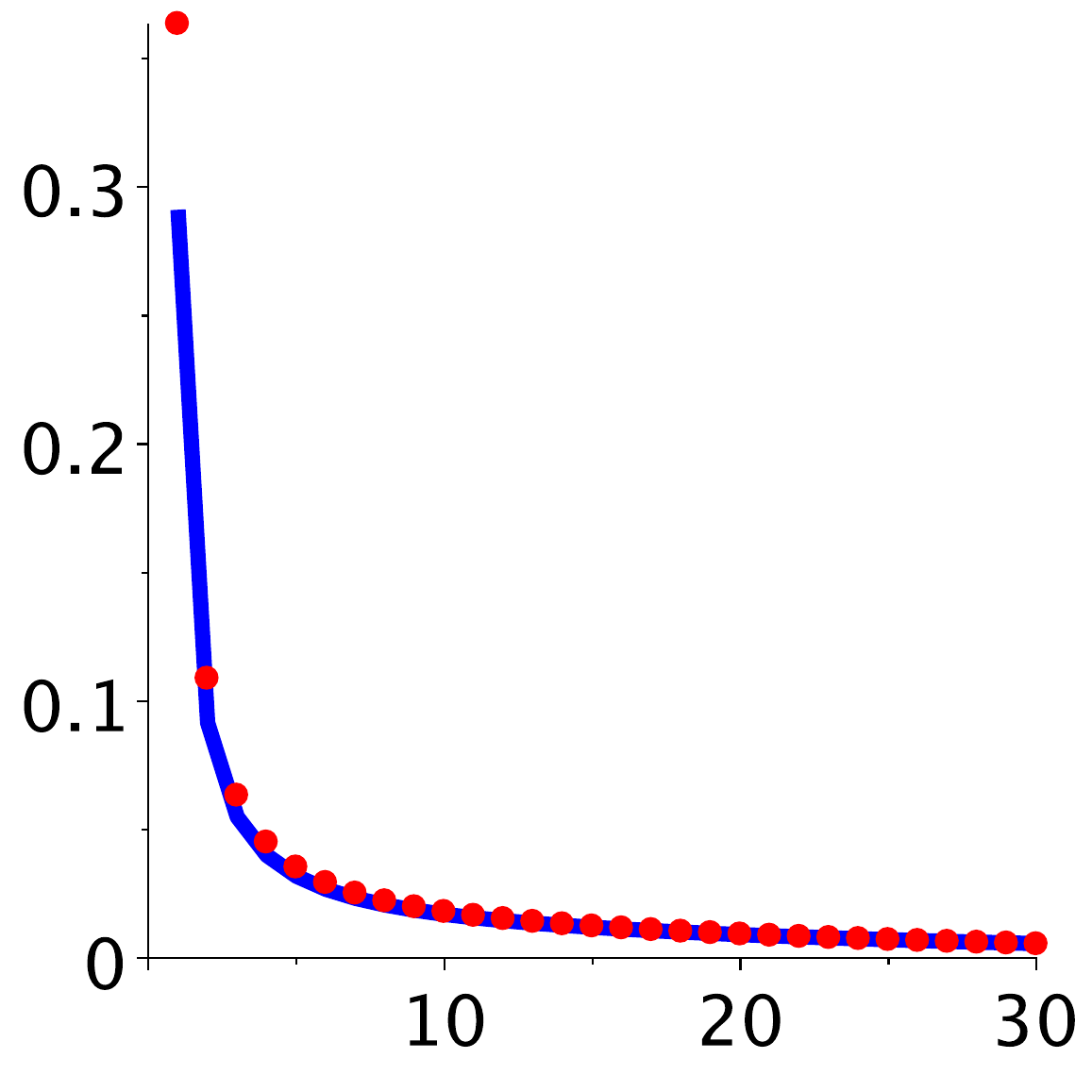} 
\ \ 
\includegraphics[width=0.32\textwidth]{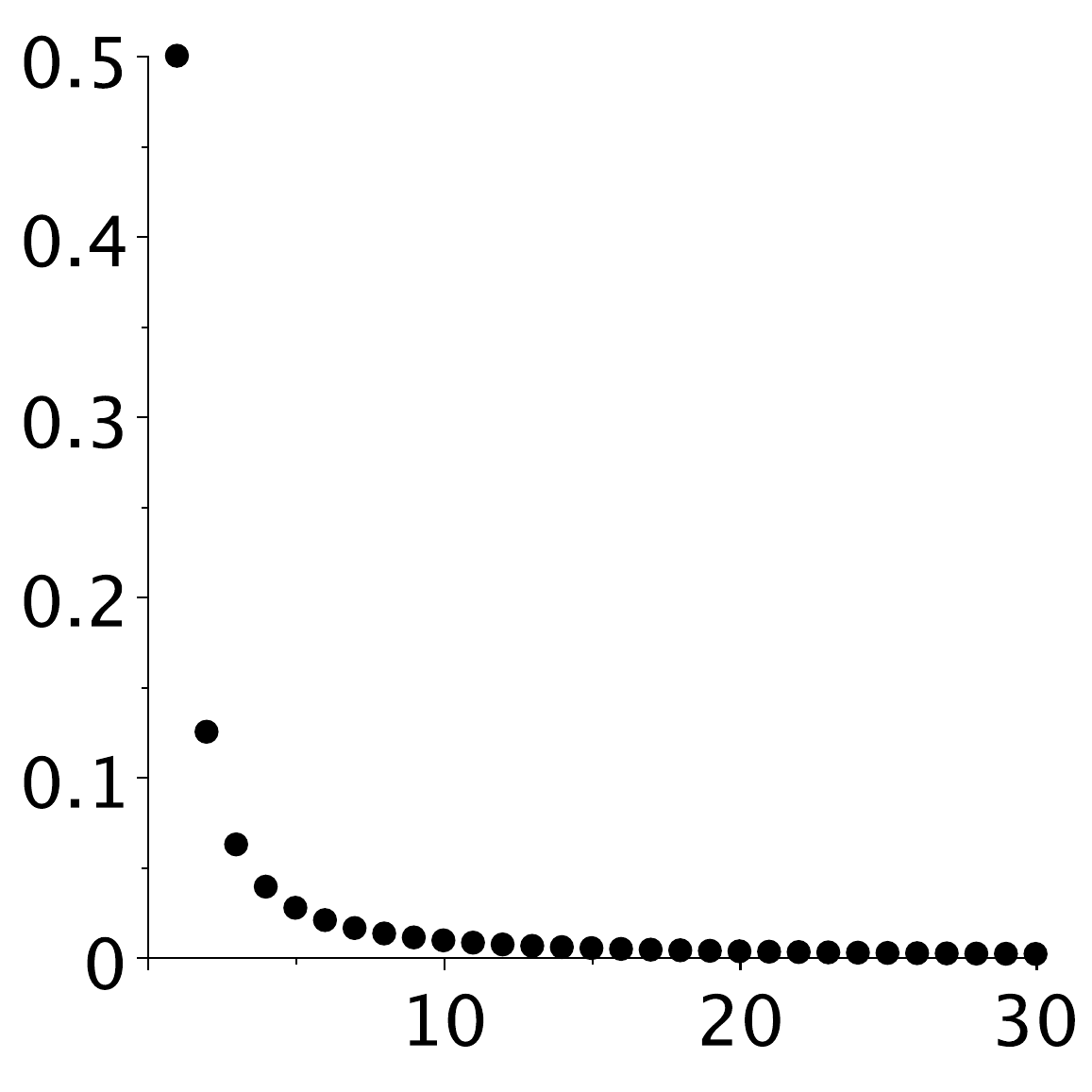} 
\ \ 
\includegraphics[width=0.32\textwidth]{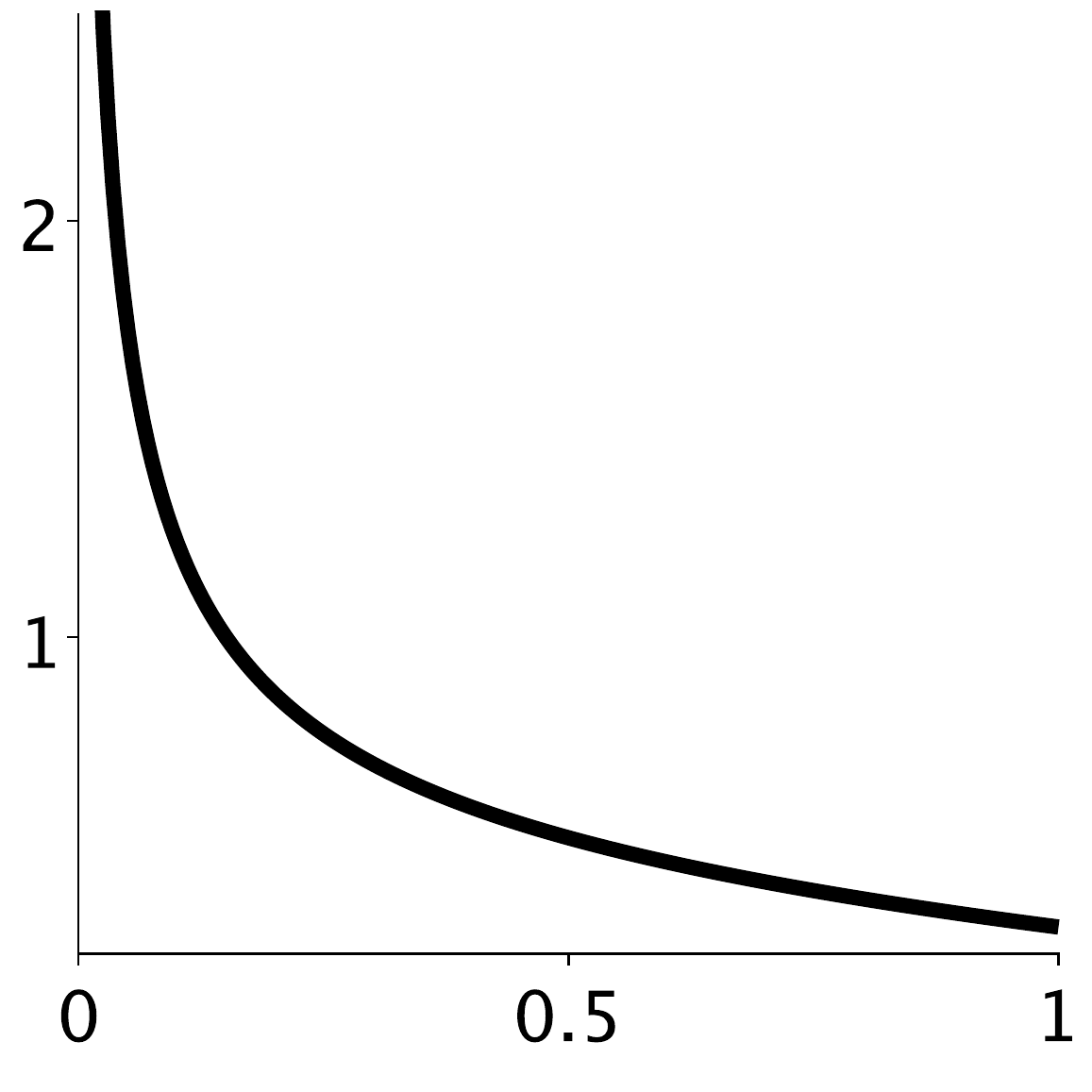} 
\end{center}

\caption{The empirical distribution $\P\{X_n=k\}$ (drawn with red dots), 
and its theoretical limiting distribution (in blue).
This blue curve is a linear combination of 
a discrete and a continuous distribution (the middle and right curves drawn in black): 
$\frac{1}{2}\,\mathcal{B}_C\left(\frac{1}{4}\right) + \frac{1}{2} \,\sqrt n \ML\left(\frac{1}{2},-\frac{1}{4}\right)$.
The blue theoretical limit curve agrees almost perfectly with the red empirical distribution even for small values of $n$ (here, $n=500$).
} 
\label{fig:supertrees_confluent}
\end{figure}

\newpage
\subsection{Root degree and branching structure in bilabelled increasing trees\label{ExBilabelled}}
Bilabelled increasing trees are a natural generalization of increasing trees~\cite{BucketPanKu2009} 
where every node is assigned two labels instead of just one. 
General families of bilabelled trees are in bijection with increasing diamonds, which are a natural type of directed acyclic graphs; see~\cite{PK2020+} for the general statement and Figure~\ref{fig:increasingdiamond} for a concrete example.
Increasing diamonds model partial orders and their linear extensions, as well as
computational processes and their executions in parallel computing~\cite{BodiniDienFontaineGenetriniHwang2016}.
They possess nice combinatorial properties and are enumerated by variants of hook-length formulas~\cite{KubPan2012,KubPan2016,PK2020+}.

\begin{figure}[t]
\begin{center}
\includegraphics[width=0.55\textwidth]{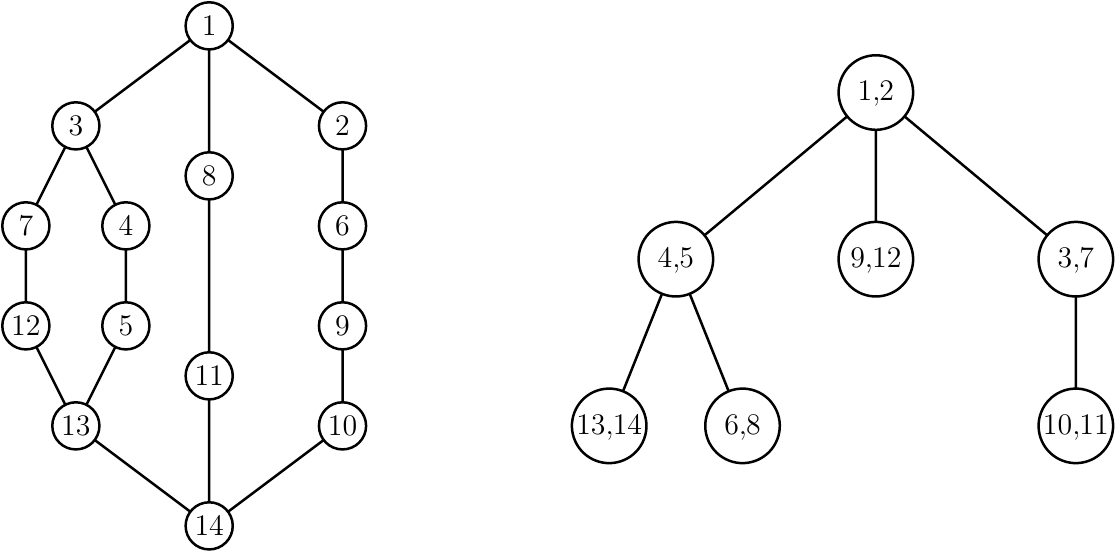} 
\end{center}
\caption{An increasing diamond and the bijectively equivalent bilabelled increasing plane-oriented recursive tree.}
\label{fig:increasingdiamond}
\end{figure}

Given a \emph{degree-weight sequence} $(\varphi_{j})_{j \ge 0}$, the corresponding \emph{degree-weight generating function} is defined as $\varphi(t) = \sum_{j \ge 0} \varphi_{j} t^{j}$.
The family $\mathcal{T}$ of 
bilabelled increasing trees can be described by the following symbolic equation
%\begin{equation*}
$ \mathcal{T} = 
 \mathcal{Z}^{\Box} \ast \left(\mathcal{Z}^{\Box} \ast \varphi\big(
 {\mathcal{T}}\big)\right),$
%\end{equation*}
where $\mathcal{Z}$ denotes single unilabelled nodes,
$\mathcal{A}^{\Box} \ast \mathcal{B}$ denotes the \emph{boxed product}
 (i.e., the smallest label is constrained to lie in the $\mathcal{A}$ component) of the combinatorial classes $\mathcal{A}$ and~$\mathcal{B}$, 
 and $\varphi(\mathcal{A}) = \varphi_{0} \cdot \{\epsilon\} + \varphi_{1} \cdot \mathcal{A} + \varphi_{2} \cdot \mathcal{A}^{2} + \dots$ denotes the class containing all \emph{weighted finite labelled sequences} of objects of $\mathcal{A}$ (i.e., each sequence of length $k$ is weighted by $\varphi_{k}$; $\epsilon$ denotes the \emph{neutral object} of size $0$); see \cite{FlaSe2009}.
Note that increasing diamonds are associated with~$\varphi(t) = \frac{1}{1-t}$.
For the exponential generating function 
$T(z) = \sum_{n \ge 1} T_{n} \frac{z^{n}}{n!}$
 where $n$ counts the number of labels, the above construction translates into
\begin{equation}
T''(z)=\varphi(T(z)),\quad T(0)=0,\quad T'(0)=0. \label{eq:inctree}
\end{equation}

We now focus on the case of $3$-bundled bilabelled increasing trees; see Figure~\ref{fig:increasingtree}. 

\begin{figure}[b]
\begin{center}
\includegraphics[width=0.95\textwidth]{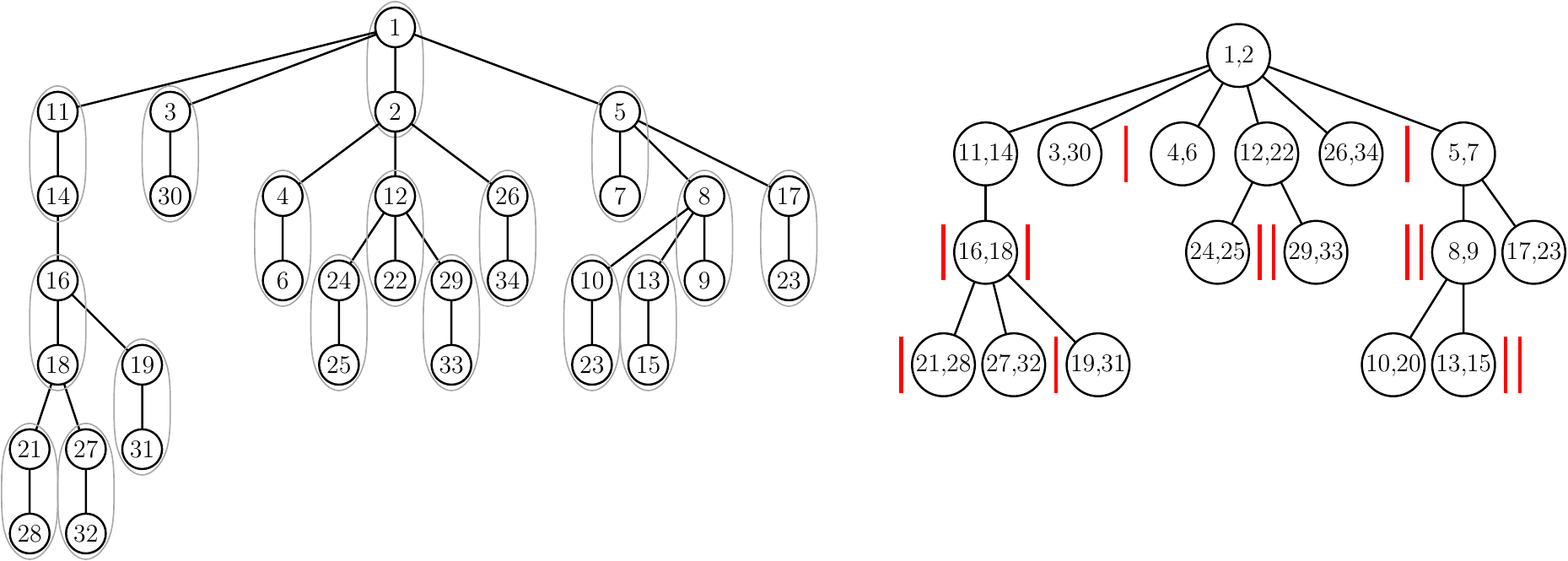} 
\end{center}
\caption{An increasing plane-oriented recursive tree and the bijectively equivalent $3$-bundled bilabelled increasing plane-oriented recursive tree. 
(The grey ellipses are just here to sketch the bijective correspondences.)}
\label{fig:increasingtree}
\end{figure}

The family of $3$-bundled trees is defined by Equation~\eqref{eq:inctree} with the following degree-weight generating function 
\begin{equation}
	\label{eq:phi3bundled}
	\varphi(t)=\frac{1}{(1-t)^{3}}=\sum_{k\geq 0} \binom{k+2}{2} t^k.
\end{equation} 
In other words, each node may have any number $k$ of children, and 
the binomial indicates two bars between these children, 
thus creating $3$ (possibly empty) sequences (or {\em bundles}) of children. 
From~\cite{KubPan2016} we get the remarkably simple closed form
\begin{align}
\label{eq:3bundledclosed}
T(z)&=1-\sqrt{1-z^2}
=\sum_{n\ge 1}(2n-1)!!(2n-3)!!\frac{z^{2n}}{(2n)!},
\end{align}
where the double factorials are defined as
\begin{equation*}
(2n-1)!!=\prod_{k=1}^{n}(2k-1)=\frac{(2n)!}{n!\cdot 2^n},
\end{equation*}
with initial values of $T_{2n}$ given by $1, 3, 45, 1575, 99225, 9823275, 1404728325,\dots$, constituting sequence \href{https://oeis.org/A079484}{A079484} in the \textsc{Oeis}. 

We are interested in the random variable $X_n$ counting the root degree of 
these $3$-bundled bilabelled increasing trees of size $n$, under the uniform random tree model.
Note that by definition there are no 
 bilabelled trees with an odd number of labels, so $T_{2n+1}=0$ and, consequently, $X_{2n+1}=0$.
In the following we use the notation $R_n=X_{2n}$ for the root degree. 

By~\eqref{eq:inctree} and \eqref{eq:phi3bundled}, the generating function
$
T(z,u)=\sum_{n\ge 1}T_n\E(u^{X_n})\frac{z^{n}}{n!}
$
satisfies
\begin{equation*}
\frac{\partial^2}{\partial z^2}T(z,u)=\varphi\big(u T(z)\big)
=\frac{1}{\big(1-u(1-\sqrt{1-z^2})\big)^3}.
\end{equation*}
%%%%% Maple (start) %%%%%
%%%series(1/(1-u*(1-sqrt(1-z^2)))^3,z,12):
%%%convert(%,polynom):tmp int(%,z$2):
%%%for i from 0 to degree(tmp,z) do seq(coeff(factorial(i)*coeff(tmp,z,i),u,k),k=1..degree(coeff(tmp,z,i),u)) end;
%%%%% Maple (end) %%%%%
%Therefore, the Taylor expansion of $T(z,u)$ starts as follows
%\begin{align*}
	%T(z,u) &= 1 + 3u\frac{z^2}{2!} + (3u + 9u^2)\frac{z^4}{4!} + (36u + 135u^2 + 540u^3)\frac{z^6}{6!} 
	%+ \dots
%\end{align*}
Since a derivative with respect to $z$ is simply a shift in the coefficient sequence of exponential generating functions, we obtain
\begin{equation*}
\E(u^{R_{n+1}})
=\frac{(2n)!}{T_{2n+2}}[z^{n}]\frac{1}{\big(1-u(1-\sqrt{1-z})\big)^3}.
\end{equation*} 
Now we observe, by Stirling's formula for the gamma function %~\eqref{StirlingGamma} 
and singularity analysis~\eqref{EqSA1} applied to~\eqref{eq:3bundledclosed}, that 
\begin{equation}
\frac{T_{2n+2}}{(2n)!}\sim \frac{2\sqrt{n}}{\sqrt{\pi}}\sim[z^n]\frac{1}{(1-z)^{3/2}}.
\label{BilabelledExpansion1}
\end{equation}
This implies that, except for the non-standard shift in the random variable, the problem is equivalent (for first-order asymptotics) to the composition scheme $(1-u(1-\sqrt{1-z}))^{-3}$, i.e., $\rho_H=1$, $\lambda_H=\frac12$, $\lambda_G=-3$, and $\lambda_M=+\infty$ (as $M(z)=1$ is entire).
We note in passing that in this special case, it is also possible to obtain a quite simple closed-form expression for the probability mass function, as well as the (factorial) moments, due to the explicit expressions for the involved generating functions:
\begin{equation*}
\E(\fallfak{X_{2n+2}}{s})=\frac{(2n)!}{T_{2n+2}}[z^{n}]\left.\partial_u^s\frac{1}{\big(1-u(1-\sqrt{1-z})\big)^3} \right|_{u=1}.
\end{equation*}

However, here we use our general scheme and Theorem~\ref{TheExtended}. 
We apply Legendre's duplication formula and obtain
\begin{equation*}
\frac{\Gamma(s+3)\Gamma(\frac32)}{\Gamma(3)\Gamma(\frac{s+3}2)}
=2^s\cdot\Gamma\left(\frac{s+4}{2}\right).
\end{equation*}
This leads to the following result.
\begin{coroll}
The random variable $R_n$, counting the root degree in a random 
strict bilabelled increasing $3$-bundled tree with $2n$ labels, with tree generating function given by $\varphi(t)=(1-t)^{-3}$, satisfies 
\begin{equation*}
\E(\fallfak{R_{n}}{s})\sim n^{s/2}\cdot 
2^s\cdot\Gamma\left(\frac{s+4}{2}\right).
\end{equation*}
The random variable $R_{n}/n^{1/2}$ converges (in distribution and in moments) to a multiple of a chi-distributed random variable $X\law \chi(4)$, with four degrees of freedom:
\begin{equation*}
\frac{R_n}{n^{1/2}}\cmom \sqrt{2}\cdot X, \quad X\law \chi(4). 
\end{equation*}
\end{coroll}

We can refine the root degree by looking at the branching structure. 
We denote by $R_{n,j}=X_{2n,j}$ the random variable counting the number of branches (i.e., subtrees) with $2j$ labels, $1\le j\le n$, attached to the root:
\begin{equation*}
R_n=\sum_{j\ge 1}R_{n,j}.
\end{equation*}
Such random variables naturally arise in the context of the Chinese restaurant process~\cite{KuPa2014,Pitman1995,Pitman2006} and generalized plane-oriented recursive trees~\cite{KubPan2006-L}. See also Feng et al.~\cite{HuFenSu2006} for the analysis of the branching structure of recursive trees.
The generating function $T(z,v)=\sum_{n\ge 1}T_n\E(v^{X_{n,j}})\frac{z^{n}}{n!}$
satisfies
\begin{equation*}
\frac{\partial^2}{\partial z^2}T(z,v)=\varphi\left(T(z)-(1-v)\frac{T_{2j}}{(2j)!}z^{2j}\right).
\end{equation*}
Consequently,
\begin{equation*}
\E(u^{R_{n+1,j}})=\frac{(2n)!}{T_{2n+2}}[z^{2n}]\Big(1-\big((1-\sqrt{1-z^2})-(1-v)\frac{T_{2j}}{(2j)!}z^{2j}\big)\Big)^{-3}.
\end{equation*}
We can use the asymptotics~\eqref{BilabelledExpansion1} and apply Theorem~\ref{TheRefined} to obtain the following result.
\begin{coroll}
The random variable $R_{n,j}$ counting the number of size $2j$ branches attached to the root in a random 
strict bilabelled increasing $3$-bundled tree with $2n$ labels, 
%with tree generating function given by 
associated with
$\varphi(t)=(1-t)^{-3}$, 
has factorial moments of mixed Poisson type,
\begin{equation*}
		\E(\fallfak{R_{n,j}}s)=\mppar_{n,j}^s\cdot \E(X^s) (1+o(1)),
\end{equation*}
with $\mppar_{n,j}=\sqrt{2}\cdot \frac{T_{2j}}{(2j)!}\cdot n^{1/2}$ and mixing distribution $X=\chi(4)$. 

Furthermore, the random variable $R_{n,j}$ possesses the three successive asymptotic régimes
of Theorem~\ref{TheRefined}, with a phase transition at $j=\Theta(n^{1/3})$.
\end{coroll}

\subsection{Returns to zero: walks and bridges with drift zero\label{ExReturns}}
A \emph{lattice path} of length~$n$ is a sequence $(s_1,\ldots, s_n)$ of steps $s_i \in \Sc$ for a fixed finite subset $\Sc \subseteq \ZZ$ called step set.
Geometrically, we fix the starting point $0$ and consider the partial sums $\sum_{i=1}^k s_i$ which can be interpreted as appending the steps one after another. 
Each step $s_i$ gets a weight $p_i>0$ and the weight of a path is the product of the weights of its steps. 
The \emph{step polynomial} 
\[P(u) = \sum_{i} p_i u^i\]
connects the weights and the steps.
We call a step set \emph{periodic} if there exist integers $b,p \in \ZZ$, $p \geq 2$ such that $P(u) = u^b P(u^p)$; otherwise we call it \emph{aperiodic}.
\pagebreak

Here and in the next section, we assume that the step set is aperiodic.
Note that this is not a  major constraint as the asymptotics of walks with periodic steps can be deduced from the ones with aperiodic ones; see~\cite{BaWa2019}. 
We call a lattice path a walk if it is unconstrained, and a bridge if it ends at zero, i.e., $\sum_{i=1}^n s_i = 0$.
A \emph{return to zero} is a point in the path such that $\sum_{i=1}^k s_i =0$; see Figure~\ref{fig:returnstozerowalk}.

\begin{figure}[t]
\begin{center}
\includegraphics[width=.75\textwidth]{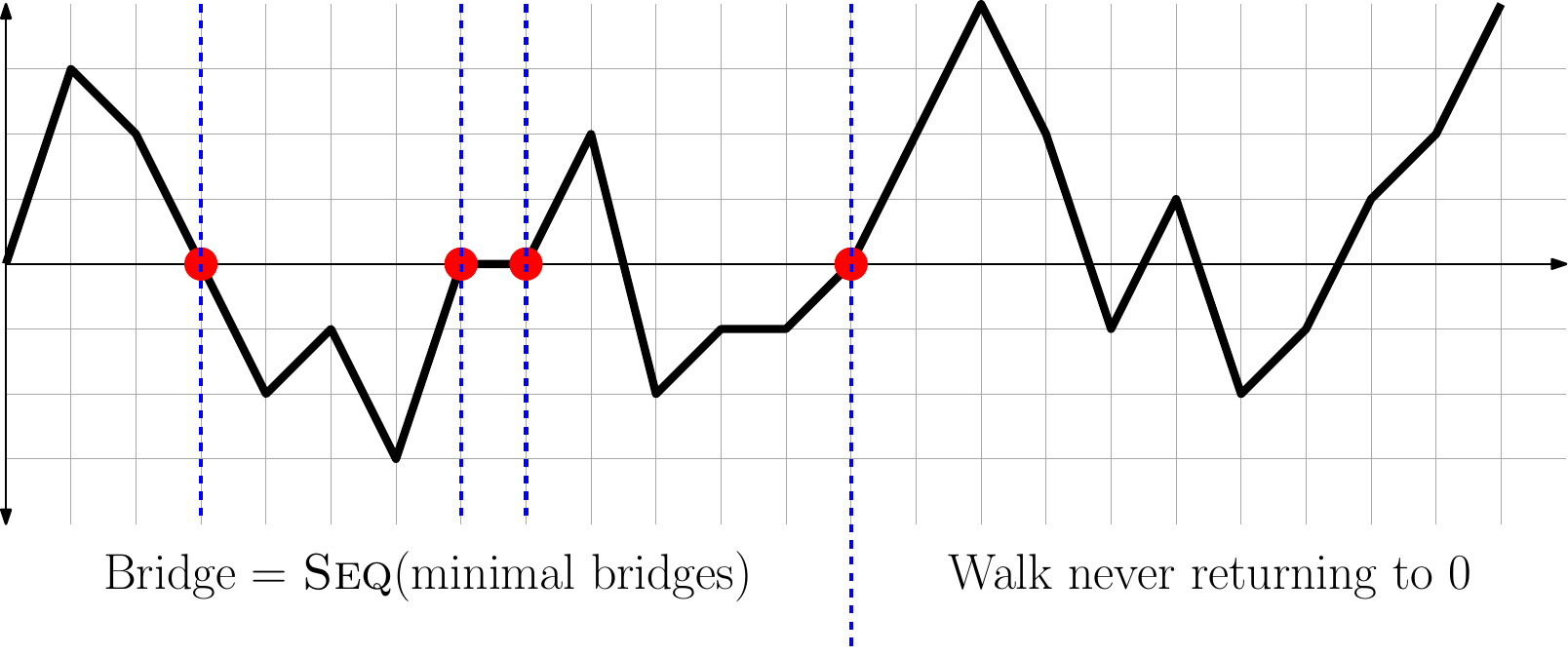} 
\end{center}
\caption{A walk consists of an initial bridge which contains all returns to zero (red dots) and a final walk never returning to zero. The bridge is further decomposed into minimal bridges touching zero only twice. }
\label{fig:returnstozerowalk}
\end{figure}

Generalizing results from Feller~\cite[Problems 9--10]{Feller} and Barton~\cite[Discussion p.~115]{SkellamShenton1957}, 
it was shown in~\cite[Section~3.2]{Wallner2020} that for drift $P'(1) = 0$ the law of the number of returns to zero follows a Rayleigh distribution for bridges, while it follows a half-normal distribution for walks. 
This result follows easily from our Theorem~\ref{TheExtended}.

Let $w_{n,k}$ be the number of walks of length $n$ with $k$ returns to zero.
The bivariate generating function of walks $W(z,u) = \sum_{n,k\geq0} w_{n,k} z^n u^k$ is given by 
\begin{align}
	\label{eq:returnswalks}
	W(z,u) &= \frac{1}{1-u\left(1-1/B(z)\right)} \frac{W(z)}{B(z)},
\end{align}
where $B(z)$ and $W(z)=1/(1-zP(1))$ are the generating functions of bridges and walks, respectively; see~\cite[Equation~(3.3)]{Wallner2020}.
To explain~\eqref{eq:returnswalks}, observe that every bridge is a sequence of \emph{minimal bridges}, 
which are bridges that never return to the $x$-axis between the start- and endpoint; see Figure~\ref{fig:returnstozerowalk}. 
Therefore, minimal bridges are enumerated by $1-1/B(z)$.
Hence, this is exactly the situation of the extended composition scheme~\eqref{Eq4} with $G(z)=1/(1-z)$, $H(z)=1-1/B(z)$, and $M(z)=W(z)/B(z)$.
Now, for zero drift, %\cite[Equation~(3.4)]{Wallner2020} shows that 
one has
\begin{align}
	\label{eq:Bsing}
	B(z) = \frac{c_B}{\sqrt{1-zP(1)}} + \O(1) \quad \text{with} \quad 
	c_B = \sqrt{\frac{P(1)}{2P''(1)}}.
\end{align}
Hence, one has $\lambda_G=-1$, $\lambda_H=\frac{1}{2}$, and $\lambdaM=-\frac{1}{2}$, therefore we are in the pure régime; see Definition~\ref{def:pure}.
This is exactly the situation in Remark~\ref{CoExtended1} and the number of returns to zero in walks 
thus follows a half-normal distribution with parameter $\sigma=\sqrt{2}c_B = \sqrt{P(1)/P''(1)}$.

Now, the generating function for bridges is nearly the same as $W(z,u)$ from~\eqref{eq:returnswalks} except that the last factor $W(z)/B(z)$ is omitted. 
So, by Corollary~\ref{CoExtended1}, the number of returns to zero here follows a Rayleigh distribution with the same parameter $\sigma=\sqrt{P(1)/P''(1)}$.

We can refine this result for the random variable~$X_{n,j}$ 
counting the number of distance-$j$-zeroes (which were introduced in~\cite{KuPa2014}). 
These are the number of returns to zero which have a distance of exactly $j$ steps to the \textit{previous} zero contact. 
The union over $j$ of distance-$j$-zeroes gives all returns to zero, and they therefore clearly represent a partition of all returns to zero.
Using Theorem~\ref{TheRefined}, we then get the following limit theorem.
 
\begin{coroll}
Let $X_{n,j}$ be the number of distance-$j$-zeroes in lattice paths of length~$n$.
For walks (resp.\ bridges) with zero drift (i.e., $P'(1)=0$), $X_{n,j}$ 
has factorial moments of mixed Poisson half-normal type (resp.\ mixed Poisson Rayleigh type)
\begin{align}
	\label{eq:Xnjwalksbridges}
	\E(\fallfak{X_{n,j}}s)=\mppar_{n,j}^s\cdot \E(X^s) \left(1+o(1)\right),
\end{align}
with $\mppar_{n,j}=\sqrt{\frac{P(1)}{2P''(1)}} \frac{h_{j}}{P(1)^{j}} \cdot n^{1/2}$, where $X$ is given by
\begin{align*}
	X &= 
	\begin{cases}
		\text{HN}(\sigma) & \text{ for walks},\\
		\text{Rayleigh}(\sigma) & \text{ for bridges},
	\end{cases}
	&
	\sigma &= \sqrt{\frac{P(1)}{P''(1)}}.
\end{align*}
Furthermore, the random variable $X_{n,j}$ possesses the three successive asymptotic régimes of Theorem~\ref{TheRefined}, with a phase transition at $j=\Theta(n^{1/3})$.
\end{coroll}

\begin{remark}[Universality of the rescaling factor]
\label{rem:walksbridges}
Note that the rescaling factor $\mppar_{n,j}$ in~\eqref{eq:Xnjwalksbridges} is the same for walks and bridges, 
while the moment sequence $\mu_s$ changes. 
This independence of $\mppar_{n,j}$ can be explained:
In Figure~\ref{fig:returnstozerowalk}, the last factor of the walk is a walk not touching zero and is encoded by $M(z) = W(z)/B(z)$. 
Then by Theorem~\ref{TheRefined} we know that $\mppar_{n,j}$ is independent of this factor, and thus has the same value for walks and bridges.

Furthermore, this factor $M(z)$ is also responsible for the often observed dichotomy between half-normal and Rayleigh distributions in the extended composition scheme which we will also observe in the next examples of initial returns and sign changes. 

Note that Formula~\eqref{eq:Xnjwalksbridges} offers a neat factorization for the moments: 
One can regroup in one factor the quantities with a probabilistic flavour (involving the variance $P''(1)$ of the allowed steps, and $\mu_s$), 
while the remaining factor ($h_j$, the number of minimal bridges of length $j$) corresponds to a quantity with a combinatorial flavour.
This could also be explained using renewal~theory.
\end{remark}

\subsection{Initial returns in coloured bridges}\label{sec:mcolouredbridges}
We generalize the previous model by introducing $m$-coloured bridges $B_m$ (see Figure~\ref{fig:colouredbridge}): 
We append $m$ non-empty bridges one after the other (and each one with a different colour): 
$B_m = (B - 1)^m.$ 
Then, we are interested in the number of returns to zero in the first bridge, i.e., the initial one that we uniformly coloured. 
We call such returns the \emph{initial returns}.

\begin{figure}[b]
\begin{center}
\includegraphics[width=.75\textwidth]{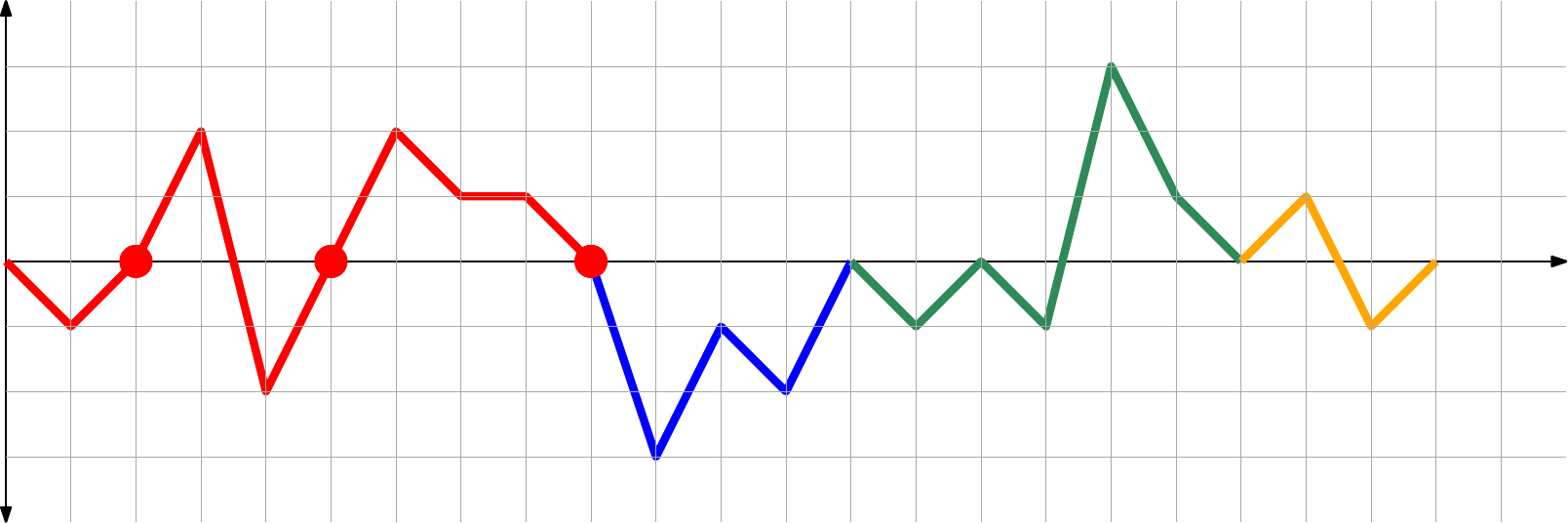} 
\end{center}
\caption{A $4$-coloured bridge, with all its initial returns to zero marked by red dots. }
\label{fig:colouredbridge}
\end{figure}

Reusing the combinatorial constructions of the previous section, this gives for the bivariate generating function $B_m(z,u)$ the following decomposition
\begin{align}
	\label{eq:mcoloredbridgeBGF}
	B_m(z,u) &= \left(\frac{1}{1-u\left(1-1/B(z)\right)}-1\right) (B(z)-1)^{m-1}.
\end{align}
For $m=1$ this is~\eqref{eq:returnswalks} except the factor $W(z)/B(z)$ and the constraint to be non-empty. Asymptotically, and therefore for the law, the non-emptiness is negligible.
The generating function $W_m(z,u)$ of $m$-coloured walks ($m$-tuples of bridges with a few more steps coloured in the same colour as the final bridge) is given by
\begin{align*}
	W_m(z,u) &= (1 + B_{m}(z,u)) \frac{W(z)}{B(z)}.
\end{align*}
Now, we can directly apply Theorem~\ref{TheExtended}. 
From the reasoning above we see that $\lambda_G=-1$, $\lambda_H=\frac{1}{2}$, and $\lambda_M=-\frac{{m-1}}{2}$ for bridges and $\lambda_M=-\frac{m}{2}$ for walks. 
\begin{coroll}
\label{coro:mcolouredlimitlaw}
The random variable $X_{n}$ counting the number of initial returns in a $m$-coloured walk (resp.\ bridge) of length $n$ satisfies
\begin{align*}
\E(\fallfak{X_n}s) &\sim n^{s/2} \left(\frac{\sigma}{\sqrt{2}}\right)^s \mu_s, &
 \sigma &= \sqrt{\frac{P(1)}{P''(1)}}, &
 \mu_s &=
\begin{cases}
	\frac{\Gamma(s+1) \Gamma((m+1)/2)}{\Gamma((m+s+1)/2)},& \text{ for walks,}\\
	\frac{\Gamma(s+1) \Gamma(m/2)}{\Gamma((m+s)/2)}, & \text{ for bridges.}
\end{cases}
\end{align*}
The random variable $X_n/n^{1/2}$ converges in distribution with convergence of all moments 
\begin{align*}
\frac{X_n}{n^{1/2}}&\cmom X,
&	X \law 
	\begin{cases}
	\frac{\sigma}{\sqrt{2}} \BML\left( \frac{1}{2}, \frac{1}{2}, \frac{m}{2} \right), & \text{ for walks,}\\
	\frac{\sigma}{\sqrt{2}} \BML\left( \frac{1}{2}, \frac{1}{2}, \frac{m-1}{2} \right), & \text{ for bridges.}
	\end{cases}
\end{align*}
These two limit laws can also be seen as the product of independent random variables, 
namely a Rayleigh and a scaled beta distribution (see Definition~\ref{ex:beta} and Example~\ref{ExRay}):
\begin{align*}
X & \law \text{Rayleigh}(\sigma) \cdot B^{1/2}, 
\quad \text{ with} \quad
	B = 
	\begin{cases}
	\operatorname{Beta}\left(\frac{1}{2}, \frac{m}{2}\right), & \text{ for walks,}\\
	\operatorname{Beta}\left(\frac{1}{2}, \frac{m-1}{2}\right), & \text{ for bridges,}
	\end{cases}
\end{align*}
where $\operatorname{Beta}(\alpha,0)=1$. Moreover, we have the local limit theorem
\begin{equation*}
\P\{X_n=x \cdot n^{1/2}\} \sim n^{-1/2} \cdot f_X(x),
\end{equation*}
where, for bridges, the density $f_X(x)$ 
is given by 
\begin{align*}
	f_X(x) &= \sqrt{\frac{2}{\pi \sigma^2}} \, \Gamma\left(\frac{m}{2}\right) e^{-\frac{x^2}{2\sigma^2}} \, U\left(\frac{m}{2}-1,\frac{1}{2},\frac{x^2}{2\sigma^2}\right),
\end{align*}
where $U(a,b,x)$ is the confluent hypergeometric function of the second kind which is the solution of $zy'' + (b-z)y'-ay=0$ such that $U(a,b,x) \sim z^{-a}$ for $z \to \infty$ and ${|\arg(z)| < 3\pi/2}$; see~\cite[\href{https://dlmf.nist.gov/13.2}{Section~13.2}]{NIST:DLMF}.
For walks, one replaces $m$ by $m+1$.
\end{coroll}

Observe the special cases $U(-1/2,1/2,x)=\sqrt{x}$ and $U(0,1/2,x)=1$ which nicely give the density functions of a Rayleigh (see Example~\ref{ExRay}) and a half-normal distribution (see Example~\ref{ExHaNo}).
Hence, for $m=1$ we recover the results of the previous section and uncover a large family of connected probability distributions. 
It is interesting that this family also appears in the context of preferential attachments in graphs~\cite[Formula~(1.1)]{PekoezRoellinRoss2013}.

It is also interesting to consider multicoloured bridges, where we allow any number of colours. We still mark by $u$ the initial returns. The corresponding %bivariate 
generating function is 
\begin{align*}
	B(z,u) &= \sum_{m \geq 1} B_m(z,u) 
	 = \left(\frac{1}{1-u\left(1-1/B(z)\right)}-1\right) \frac{1}{2-B(z)}.
\end{align*}

\pagebreak

\noindent The generating function for the number of multicoloured bridges is thus
$B(z,1) 		= \frac{1}{2-B(z)}-1$. 
From~\eqref{eq:Bsing} we see that $B(z)$ possesses a singularity of order $-1/2$ at $z=1/P(1)$, and hence $B(z,1)$ becomes singular at some $z_0>0$ which is the unique solution of $B(z_0)=2$. 
Hence, the probability generating function reveals a geometric distribution of parameter $1/2$:
\begin{align*}
	\frac{[z^n] B(z,u)}{[z^n] B(z,1)} &
		\sim \frac{1}{1-u\left(1-1/B(z_0)\right)}-1
		= \frac{u/2}{1-u/2}.
\end{align*}
As the truncated sum $\sum_{m=1}^{m_0} B_m(z,u)$ behaves asymptotically like $B_{m_0}(z,u)$,
we see here a phase transition from a continuous law (for any finite $m_0$) to a discrete law (when $m_0$ goes to infinity). 
Note that this phenomenon holds verbatim for walks.

Finally, let us apply the size-refined scheme Theorem~\ref{TheRefined}, counting initial returns in $m$-coloured bridges (or walks) which are a certain distance apart:
\begin{coroll}
\label{coro:mcolouredlimitlawsize-size-refined}
Let $X_{n,j}$ be the number of initial returns at distance $j$ from the previous zero in $m$-coloured walks or bridges of length $n$.
Then, $X_{n,j}$ has mixed Poisson type moments
\begin{align*}
	\E(\fallfak{X_{n,j}}s)=\mppar_{n,j}^s\cdot \E(X^s) \left(1+o(1)\right),
\end{align*}
with $\mppar_{n,j}=\sqrt{\frac{P(1)}{2P''(1)}} \frac{h_{j}}{P(1)^{j}} \cdot n^{1/2}$, $h_j = [z^j] (1-1/B(z))$, and where 
$X \law \frac{\sigma}{\sqrt{2}} \BML\left( \frac{1}{2}, \frac{1}{2}, \frac{m}{2} \right)$ for walks and $ X \law \frac{\sigma}{\sqrt{2}} \BML\left( \frac{1}{2}, \frac{1}{2}, \frac{m-1}{2} \right)$ for bridges, as given in Corollary~\ref{coro:mcolouredlimitlaw}. 
Furthermore, $X_{n,j}$ possesses the three successive asymptotic régimes
of Theorem~\ref{TheRefined}, with a phase transition at $j=\Theta(n^{1/3})$.
\end{coroll}

These results also hold for other variants of paths that are
in bijection with sequences 
that already appeared in the literature; 
see Table~\ref{tab:mcolouredDyck}.

\begin{table}[ht]
\centering
\begin{tabular}{@{}llll@{}}
	\toprule
	Steps & GF & Sequence & OEIS \\
	\midrule
	$\{\mathsf{U},\mathsf{D}\}$ & 
	$\frac{8z^2-2- \sqrt{1 - 4z^2}}{16z^2 - 3}$ &
	$1, 0, 2, 0, 10, 0, 52, 0, 274, 0, 1452, 0, 7716, \dots$ &
	\OEISs{A075436}\\
	$\{\mathsf{U},\mathsf{D},\mathsf{H}_1\}$ & 
	$\frac{z + \sqrt{1-4z^2}}{1-5z^2}$ &
	$1, 1, 3, 5, 13, 25, 61, 125, 295, 625, 1447, \dots$ &
	\OEISs{A098615}\\
	$\{\mathsf{U},\mathsf{D},\mathsf{H}_2\}$ & 
	$\frac{z^2 + \sqrt{1-4z^2}}{1-4z^2-z^4}$ &
	$1, 0, 3, 0, 11, 0, 43, 0, 173, 0, 707, 0, 2917, \dots$ &
	\OEISs{A026671}\\
	\bottomrule
\end{tabular}
\caption{
Multicoloured bridge models: 
they end at $0$ and use up steps $\mathsf{U}=(1,1)$, down steps $\mathsf{D}=(1,-1)$,\\ and 
horizontal steps $\mathsf{H}_i=(i,0)$ allowed only at altitude 0.
The limit laws of initial returns to zero\\ 
in these models are all the same and special cases of Corollaries~\ref{coro:mcolouredlimitlaw} and \ref{coro:mcolouredlimitlawsize-size-refined}. 
}
\label{tab:mcolouredDyck}
\end{table}

\subsection{Sign changes in walks} 
\label{ExSign}

Using the same notation as in Example~\ref{ExReturns},
we now define the \emph{sign} of the path $(s_1,\dots,s_n)$ after $k$ steps as $\operatorname{sgn}(\sum_{i=1}^k s_i) \in \{-1,0,1\}$. 
Thereby every lattice path is associated with a sequence of signs. 
A \emph{sign change} is therein any subsequence 
$(-1,0^*,1)$ or $(1,0^*,-1)$ where $0^*$ denotes a (possibly empty) sequence of~$0$s; see~Figure~\ref{fig:signchanges}.

In this section we consider \emph{Motzkin paths}.
They are composed of up steps $+1$, down steps~$-1$, and horizontal steps $0$; see again Figure~\ref{fig:signchanges}.
Their step polynomial is therefore given by $P(u)=\frac{p_{-1}}{u} + p_0 + p_1 u$ (with $p_{-1} p_0 p_1\neq 0$).
In the case of zero drift, let us show how to apply our results to get that the number of sign changes follows asymptotically a Rayleigh distribution for bridges and a half-normal distribution for walks, while for nonzero drift it follows a geometric distribution; see~\cite{Wallner2020}.

\begin{figure}[t]
\begin{center}
\includegraphics[width=.73\textwidth]{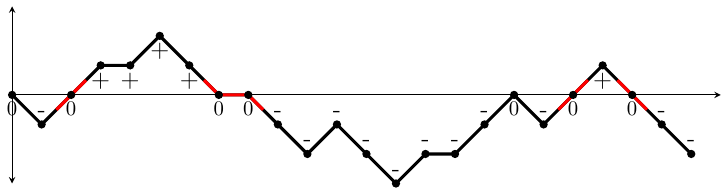} 
\end{center}
\caption{A Motzkin walk (i.e., step set $\Sc = \{-1,0,1\}$) with $4$ sign changes marked in red.}
\label{fig:signchanges}
\end{figure}

Combinatorially, we see that the bivariate generating function of bridges is 
\begin{align*}
	B(z,u) &= S(z)\left(1+\frac{2 H(z)}{1-uH(z)}\right), \quad \text{ where } 
	S(z) = \frac{1}{1-p_0 z} \quad \text{ and } \quad
	H(z) = \frac{E(z)}{S(z)}-1.
\end{align*}
Here, $S(z)$ counts sequences of horizontal steps, $E(z)$ counts excursions (bridges constrained to be nonnegative; see~\cite{BaFla2002}), and $H(z)$ counts excursions which start with an up or a down step (and not with a horizontal step). 

We now give the main corresponding Puiseux expansions. First one has 
\begin{align*}
	H(z) &= 1 - 2 \sqrt{\frac{2 P(1)}{P''(1)}} \sqrt{1-zP(1)} + \O(1-zP(1)).
\end{align*}
Then, as the radius of convergence $1/p_0$ of $S(z)$ is strictly larger than $1/P(1)$, which is the one of $H(z)$, 
we see that the additive term $S(z)$ is negligible for the limit law.
Thus, we have a composition scheme~\eqref{Eq4} where $M(z) = 2S(z)H(z)$ has the asymptotic expansion
\begin{align*}
	M(z) = 2E\left(\frac{1}{P(1)}\right) + \O\left(\sqrt{1-zP(1)}\right). 
\end{align*}
Hence, we have $\lambdaM =0$, which means that the factor $M(z)$ is asymptotically negligible for the law. 
The asymptotic dominant part arises from 
$\frac{1}{1-uH(z)}$ 
and we get from Corollary~\ref{CoExtended1} the expected convergence to a Rayleigh distribution with parameter $\sigma= -\sqrt{2} \frac{\tau_H}{c_H} = \frac12 \sqrt{\frac{P''(1)}{P(1)}}$.

A similar reasoning (and Remark~\ref{CoExtended1}) allows us to prove that the number of sign changes in walks asymptotically follows a half-normal distribution with the same parameter $\sigma$. 
We now refine the analysis by counting sign changes which are $j$ steps apart. 
Then we can apply Theorem~\ref{TheRefined} to get the following refined result which strongly depends on $H(z)=\sum_{j \geq 0} h_j z^j$. 
Note that a statement analogous to Remark~\ref{rem:walksbridges} also applies here.

\begin{coroll}
For walks of length $n$ of Motzkin paths, 
let the random variable $X_{n,j}$ be the number of sign changes at distance $j$ from the previous sign change or the origin.
For walks (resp.\ bridges) with zero drift (i.e., $P'(1)=0$), $X_{n,j}$ 
has factorial moments of mixed Poisson half-normal type (resp.\ mixed Poisson Rayleigh type)
\begin{equation*}
	\E(\fallfak{X_{n,j}}s)=\mppar_{n,j}^s\cdot \E(X^s) \left(1+o(1)\right),
\end{equation*}
with $\mppar_{n,j}=\frac12 \sqrt{\frac{P''(1)}{2P(1)}} \frac{h_{j}}{P(1)^{j}} \cdot n^{1/2}$ and mixing distributions 
\begin{align*}
	X &
	\law
	\begin{cases}
		\text{HN}(\sigma) & \text{ for walks},\\
		\text{Rayleigh}(\sigma) & \text{ for bridges}, 
	\end{cases}
	&
	\sigma &= \frac12 \sqrt{\frac{P''(1)}{P(1)}}.
\end{align*}
Furthermore, the random variable $X_{n,j}$ (for walks and for bridges) possesses the three successive asymptotic régimes
of Theorem~\ref{TheRefined}, with a phase transition at $j=\Theta(n^{1/3})$.
\end{coroll}
\pagebreak 

\subsection{Tables in the Chinese restaurant process} \label{Chinese}
Following Aldous, Pitman, and Dubins (see~\cite{Aldous1983,Pitman1995,Pitman2006}), 
we now consider the Chinese restaurant process. 
This a discrete-time stochastic process having as value at time $n$ one of the $B_n$ partitions of the set $[n]=\{1, 2, \dots, n\}$
(where $B_n$ denotes the Bell numbers found as sequence \OEISs{A000110} in the \textsc{Oeis}).
One fancifully imagines a Chinese restaurant with an infinite number of tables, where each table has a possibly infinite number of seats. 
In the beginning the first customer takes a seat at the first table. At each discrete time step a new customer arrives and either joins one of the existing tables, or takes a seat at the next empty table.
Each table corresponds to a block of a random partition. 
The process thus starts at time $n = 1$ with the partition $\{ \{1\} \}$. 
Now, given a partition $T=\{t_1,\dots,t_k\}$ of $[n]$ with $|T|=k$ parts~$t_i$,
at time $n + 1$ the element $n + 1$ is either added to one of the existing parts $t_i\in T$ with probability 
\begin{equation*}
\P\{n+1 \prec t_i\}=\frac{|t_i|-\alpha}{n+\beta},\quad 1\le i\le k,
\end{equation*}
where $n +1 \prec t_i$ denotes that $n+1$ is a costumer sitting at table $t_i$, 
or as a new singleton block with probability
\begin{equation*}
\P\{n+1 \prec t_{|T|+1}\}=\frac{\beta+k \cdot \alpha}{n+\beta}.
\end{equation*}

This model (parametrized by the two parameters $0< \alpha <1$ and $\beta>-\alpha$)
thus assigns a probability to any particular partition~$T$ of~$[n]$. 
We are interested in the random variable $C_{n}$, counting the total number of tables in the Chinese restaurant process, 
as well as the random variable $C_{n,j}$, counting the number of parts of size $j$ in a partition of $[n]$.
As pointed out in~\cite{KuPa2014}, this process can be embedded into a variant
of the growth process of generalized plane-oriented recursive trees with two different connectivity parameters $a>0$ and $b>-1$. 
This allows us to study properties of the Chinese restaurant process using analytic combinatorial tools. For the reader's convenience, we restate this embedding below.

\begin{figure}[b]
\begin{center}
\includegraphics[width=.9\textwidth]{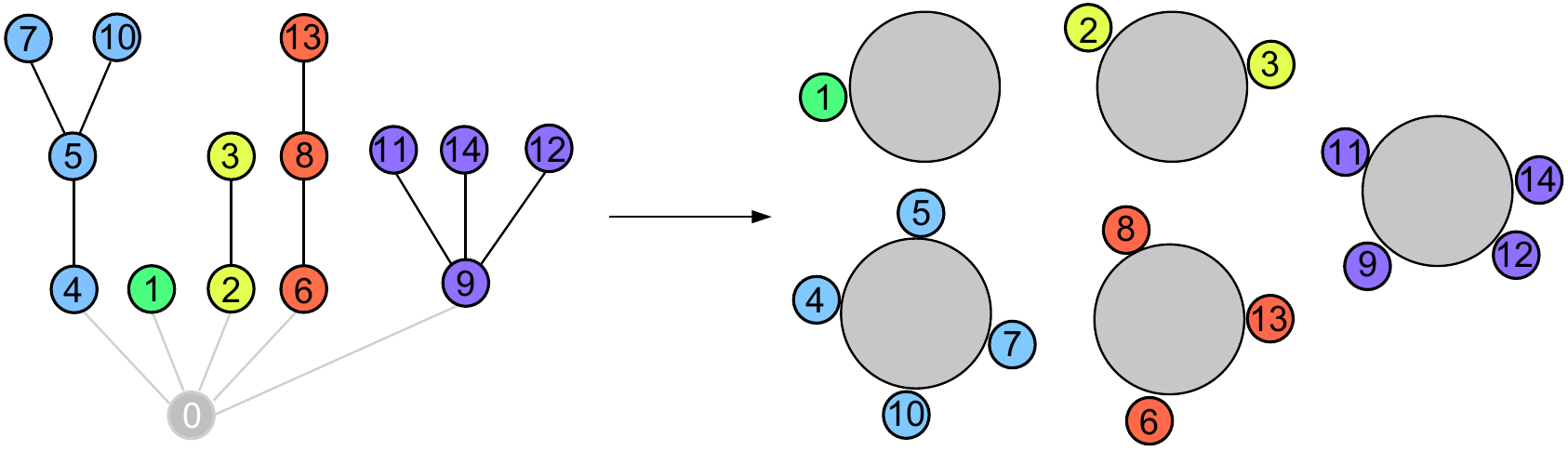} 
\end{center}
\caption{A plane-oriented recursive tree of size $15$ and the corresponding tables in the Chinese restaurant model.}
\label{fig:PortSubtreeSizesb}
\end{figure}

We collect the results of~\cite{KuPa2014} and complement them by extending the constraint $b>0$ to the full range $b>-1$ 
as well as by providing the missing identification of the limit law as a (moment-shifted) stable law.

Combinatorially, we consider a family $\mathcal{T}_{a,b}$ of generalized plane-oriented recursive trees, where the degree-weight generating function $\psi(t)=\frac{1}{(1-t)^b}$, $b>0$, associated with the root of the tree, is different to the one for non-root nodes in the tree, $\varphi(t)=\frac{1}{(1-t)^a}$, $a>0$. Then, the family $\mathcal{T}_{a,b}$ is closely related to the corresponding family $\mathcal{T}$ of 
generalized plane-oriented recursive trees with degree-weight generating $\varphi(t)=\frac{1}{(1-t)^a}$, $a>0$, via the following formal recursive equations (see Section~\ref{ExBilabelled} for the definition of the boxed product):
\begin{equation*}
\begin{split}
 \mathcal{T}_{a,b} =\mathcal{Z}^{\Box} \ast \psi(\mathcal{T}),
\qquad \mathcal{T} = \mathcal{Z}^{\Box} \ast \varphi(\mathcal{T}).
\end{split}
\end{equation*}
\pagebreak

The weight $w(T)$ of a tree $T \in \mathcal{T}_{a,b}$ is then defined by
\begin{equation*}
w(T) = \psi_{d(\text{root})}\prod_{v\in T\setminus\{\text{root}\}} \varphi_{d(v)},
\end{equation*}
where $d(v)$ denotes the outdegree of node~$v$. Thus, the generating functions
\begin{equation*}T_{a,b}(z)=\sum_{n\ge 1}T_{a,b;n}\frac{z^n}{n!} \text{\quad and \quad} T(z)=\sum_{n\ge 1}T_n\frac{z^n}{n!}\end{equation*} of the total weight of size-$n$ trees in $\mathcal{T}_{a,b}$ and $\mathcal{T}$, respectively, satisfy the differential equations
\begin{equation*}
T'_{a,b}(z)=\psi(T(z)) \text{\quad and \quad} T'(z)=\varphi(T(z)).
\end{equation*}

The ordinary tree evolution process to generate a random tree of arbitrary given size in the family $\mathcal{T}$ (see~\cite{PanPro2007} for a detailed discussion) can be extended in the following way to generate a random tree in the family $\mathcal{T}_{a,b}$.
The process, evolving in discrete time, starts with the root labelled zero.
At step $n+1$, with $n\ge 0$, the node with label $n+1$ is attached as a new child to any previous
node~$v$ (this is denoted by $n+1 \prec v$) with probabilities
\begin{equation*}
\P\{n+1\prec v \}=
\begin{cases}
\frac{d(v)+b}{b+(a+1)n} & \text{if $v$ is the root},\\
\frac{d(v)+a}{b+(a+1)n} & \text{if $v$ is not the root}.
\end{cases}
\end{equation*}
We recall the following result from~\cite{KuPa2014}.

\begin{prop}[Chinese restaurant process and generalized plane-oriented recursive trees]
\label{ChineseProp}
A random partition of $\{1,\dots,n\}$ generated by the Chinese restaurant process with parameters $0<\alpha<1$ and $\beta>0$ can be generated equivalently by the growth process of the family of generalized
plane-oriented recursive trees $\mathcal{T}_{a,b}$ when generating such a tree of size $n+1$. The parameters $\alpha,\beta$ and $a,b>0$, respectively, are related via
\begin{equation*}
	\alpha=\frac{1}{1+a},\qquad \beta=\frac{b}{1+a}.
\end{equation*}

The random variable $C_n$ is distributed as the outdegree $X_{n+1}$ of the root of a random generalized plane-oriented recursive trees of size $n+1$ from the family $\mathcal{T}_{a,b}$:
$C_{n}\law X_{n+1}$.

The random variable $C_{n,j}$ is distributed as the number $X_{n+1,j}$ of branches of size~$j$ attached to the root of a random tree
 of size $n+1$ from the family $\mathcal{T}_{a,b}$:
$C_{n,j}\law X_{n+1,j}$.
\end{prop}
Note that in the above relation, $\beta$ cannot be negative, since $b$ is assumed to be positive. 
As already observed in~\cite{KuPa2014}, the correspondence can be extended to the full range $0< \alpha <1$ and $\beta>-a$, where one has 
$a=\frac{1}{\alpha}-1>0 \text{ and } b=\frac{\beta}{\alpha}>-1.$
For $-1<b\le 0$, we cannot directly use the degree-weight generating function $\psi(t)=(1-t)^{-b}$. 
Indeed, for $-1<b<0$ we would have $\psi(t)=1+b t + \dots$, involving a negative weight; while for $b=0$ we would have $\psi(t)=1$, a degenerate case. 
However, we can use a modified generating function, leading to a correct model of the Chinese restaurant process in the range $-1<b\le 0$ (see~\cite{PanPro2007,KuPa2014} for more details on the growth process): 
\begin{equation*}
\psi(t)=1+\int_0^{t}\frac{1}{(1-x)^{1+b}} \, dx
= 1 + \frac{1}{b}\Big(\frac1{(1-t)^{b}}-1\Big)=1+\sum_{k\ge 1}\binom{b+k}{k-1}\frac{t^k}k,
\end{equation*}
for $-1<b<0$, while for $b=0$ one uses
\begin{equation*}
\psi(t)=1-\log(1-t)=1+\sum_{k\ge 1}\frac{t^k}k. 
\end{equation*}

\pagebreak

Thus, we have some generalized plane-oriented recursive trees attached to a root with a different tree-weight generating function $\psi(t)$. 
Summarizing, we have
\begin{equation*}
\psi(t)=
\begin{cases}
\frac{1}{(1-t)^b} & \text{ if } b>0,\\
1-\log(1-t) & \text{ if } b=0,\\
 1 + \frac{1}{b}\Big(\frac1{(1-t)^{b}}-1\Big) & \text{ if } -1<b<0.
\end{cases}
\end{equation*}
Here (except for the special case $b=0$, which is handled by a slightly 
different approach, detailed later in Theorem~\ref{COMPSCHEMECycleThe1}),
we can directly apply our results from Theorem~\ref{TheExtended} to
\begin{equation*}
T'_{a,b}(z,u)=\sum_{n\ge 1} T_n \E(u^{X_{n}}) \frac{z^{n-1}}{(n-1)!}=\psi(u\cdot T(z)),\qquad T'(z)=\varphi(T(z)),
\end{equation*}
for the total number of tables, and from Theorem~\ref{TheRefined} to 
\begin{equation*}
R'_{a,b}(z,v)=\sum_{n\ge 1} T_n \E(v^{X_{n,j}}) \frac{z^{n-1}}{(n-1)!}=\psi(T(z)-(1-v)z^j\frac{T_j}{j!}),
\end{equation*}
for the number of tables of size $j$.
This allows us to extend the corresponding result of~\cite{KuPa2014} to the full range of $b>-1$, 
also providing the missing identification of the limit law as a (moment-tilted) stable law:
\begin{theorem}
\label{ChineseThe}
Let $a>0$, $b> -1$. The random variable $X_{n,j}$ counting the number of branches of size $j$ in a random ${\mathcal T}_{a,b}$ tree of size $n$
(or, equivalently, the number of tables with~$j$~seated customers in a Chinese restaurant process of parameter $\alpha=1/(1+a)$ and $\beta=b/(1+a)$, with a total of $n-1$ customers) possesses the three successive asymptotic régimes of Theorem~\ref{TheRefined}, with a phase transition at $j=\Theta(n^{1/(a+2)})$:
\begin{enumerate}[(i)]
 \item For $j\ll n^{\frac{1}{a+2}}$ we have $\mppar_{n,j}=\frac{\alpha n^{\alpha}}{j} \binom{j-1-\alpha}{j-1}\to\infty$ and $\frac{X_{n,j}}{\mppar_{n,j}}$ 
 converges in distribution, with convergence of all moments, to 
 a two-parameter Mittag-Leffler distribution:
\begin{equation*}
\frac{X_{n,j}}{\mppar_{n,j}}\cmom X \qquad
\text{ with } \qquad X\law \ML(\alpha,\beta).
\end{equation*}
 \item For $j\sim r\cdot n^{\frac{1}{a+2}}$, $r\in (0,\infty)$, we have $\mppar_{n,j}\to\mppar$, and the random variable $X_{n,j}$ converges in distribution, with convergence of all moments, to a mixed Poisson distribution:
\begin{equation*}
X_{n,j}\cmom \MPo(\mppar X).
\end{equation*}
 \item For $j\gg n^{\frac{1}{a+2}}$ we have $\mppar_{n,j}\to 0$, so $X_{n,j}$ converges to a Dirac distribution at 0.
	\end{enumerate}
\end{theorem}
\begin{remark}
Our result above implies that there are only a few giant tables in the Chinese restaurant process
(a mixed-Poisson number of tables with a number of customers proportional to $n^{\frac{1}{a+2}}$).
In contrast, there are much more tables with a smaller number of customers, and
an asymptotically negligible number of tables of size $\gg n^{\frac{1}{a+2}}$.
\end{remark}

%%\begin{remark}
%%Closed formulas for the factorial moments $\E(\fallfak{X_{n,j}}{s})$, as well as a formula for the probability mass function of $X_{n,j}$ are readily obtained from the generating functions by extraction of 
%%coefficients.
%%\end{remark}

\begin{remark}
Our results also allow recovering the limit theorem in~\cite{Pitman2006} for the total number of tables $C_n$ in the Chinese restaurant process (via $X_n$), albeit with a totally different proof,
as the normalized random variable $X_{n}/n^\alpha$ converges in distribution with convergence of all moments 
to a random variable $X$, with $X$ given in the theorem before. For the reader's convenience, we state the moments in terms of $\beta$ and $\alpha$, compare with~\cite[Theorem 3.8]{Pitman2006}:
\begin{flalign*}
&&\E(X^s)=\frac{\Gamma(s+\frac{\beta}{\alpha})\Gamma(\beta)}{\Gamma(\beta+s\cdot \alpha)\Gamma(\frac{\beta}{\alpha})}. &&\myqedhere
\end{flalign*}
\end{remark}

\pagebreak

\begin{proof}[Proof of Theorem~\ref{ChineseThe}]
We follow very closely~\cite{KuPa2014} and sketch the remaining steps. We solve the differential equation $T'(z)=\varphi(T(z))$ and get
%\begin{equation*}
$T(z)=1-(1-(a+1)z)^{\frac{1}{a+1}}. $
%\end{equation*}
Thus, the probability generating function is given by 
\begin{equation*}
 \E(v^{X_{n+1,j}})=\frac{n!}{T_{n+1}}[z^n]\psi\left(T(z)-(1-v)z^j\frac{T_j}{j!}\right),
\end{equation*}
where the coefficient $\frac{T_{n+1}}{n!}=[z^n]\psi\big(T(z)\big)$
is computed by standard singularity analysis. 
Therefore, except for the non-standard shift, we can readily apply our scheme to 
\[\psi\left(T(z)-(1-v)z^j\frac{T_j}{j!}\right).\]
Here, $T=H$ and $\psi=G$, whose Puiseux exponents are 
$\lambda_{H}=\frac{1}{a+1}$ and $\lambda_{G}=-b$ for $b\neq 0$, where $a>0$ and $b >-1$. 
Hence, the critical range is given by $j(n)=\Theta  (n^{\frac{\lambda {H}}{1+\lambda {H}}} )=\Theta  (n^{\frac{1}{a+2}} )$.

In this model, no additional factor $M(z)$ is present, so
$\lambda _{M}^{\mkern - 1. mu{{{-}}}}=0$, and if $b=0$ we
apply the cycle scheme of Theorem~\ref{COMPSCHEMECycleThe2}. This gives
\begin{equation*}
\E(X^s)=\mu_s=
\begin{cases}
 \frac{\Gamma(s+b)\Gamma(\frac{b}{a+1})}{\Gamma(b)\Gamma(\frac{b+s}{a+1})} & \text{ if } b\neq 0,\\
 \frac{\Gamma(s+1)}{\Gamma(\frac{s}{a+1}+1)} & \text{ if } b=0. 
\end{cases}
\end{equation*}
Finally, we unify both expressions by simply using $\Gamma(x+1)=x \Gamma(x)$.
\end{proof}

\begin{remark}[Chinese restaurant with a bar]
Recently, Möhle introduced in~\cite{Moehle2021Bar} a generalization of the classical Chinese restaurant process, in which in addition to the tables, the customers 
can sit at an infinite bar. 
After proper rescaling, the number of customers at the bar converges to a beta-distributed random variable, 
and the number of occupied tables to the three-parameter Mittag-Leffler distribution.
\end{remark}

\subsection{Triangular urn models and the three-parameter Mittag-Leffler distribution}\label{SubSecTriangular}
Two-colour triangular urns are instances of generalized P\'olya urn models~\cite{Jan2004,FlaDumPuy2006,M2008}. 
At each time step $n \ge 1$, a ball is drawn uniformly at random, reinserted, and depending on the observed colour, balls of both colours are added to the urn:
If a white ball was drawn, we add $a$ white and $b$ black balls, whereas, 
if a black ball was drawn, we add $c$ white and $d$ black balls.
The addition/replacement of balls can be described by the so-called \emph{ball replacement matrix} 
$M = \left(\begin{matrix} a & b \\ c & d \end{matrix}\right),$
where for \emph{balanced} urn models it holds that $a+b=c+d$, such that the total number $\sigma=a+b$ of added balls in each step is independent of the observed colour. The initial configuration of the urn consists of $w_0$ white balls and $b_0$ black balls, 
and the random variable $\mathcal{W}_n$ counts the number of white balls in the urn after $n$ draws.
For balanced urns with replacement matrix
\begin{equation*}
M = \left(\begin{matrix} a & b \\ 0 & d \end{matrix}\right), \quad a, b >0, \quad d =\sigma=a+b,
\end{equation*}
it was shown by Flajolet, Dumas, and Puyhaubert~\cite{FlaDumPuy2006} (and also by Janson~\cite{Jan2006,Jan2010} 
via different analytic methods)
that $\frac{\mathcal{W}_n}{a n^{a/\sigma}}\claw \mathcal{W}$, for a random variable with moments
\begin{equation*}\label{FlDuPu}
\E(\mathcal{W}^s)=\frac{\Gamma(\frac{b_0+w_0}{\sigma})}{\Gamma(\frac{w_0}{a})}\cdot \frac{\Gamma(s+\frac{w_0}{a})}{\Gamma(s\cdot \frac{a}{\sigma}+\frac{b_0+w_0}{\sigma})}.
\end{equation*}
The limit law thus depends on the initial composition of the urn, as it is often the case for urn models.
In the special case $(w_0,b_0) = (a,b)$, and thus $w_0+b_0= \sigma$, $\mathcal{W}$ follows a Mittag-Leffler distribution $\ML(a/\sigma)$.
For $b_0 > 0$ and either $w_0 = 0$ or $w_0 = b$, Janson observed a moment-tilted stable law, leaving the other cases open; see~\cite[Theorem~1.8 and Problem~1.15]{Jan2006}. 

\pagebreak

Let us now show how the general case 
%(i.e., a balanced triangular P\'olya urn with any initial conditions) 
can be solved using our extended scheme. 
The key tool is the history generating function $ F(z,u) = \sum_{n,k \geq 0} f_{n,k} u^k \frac{z^n}{n!}$ 
where $f_{n,k}$ is equal to the number of transitions (or histories) leading to a configuration with $k$ white balls after $n$ steps. 

The closed form of this history generating function was derived in~\cite[Proposition~14]{FlaDumPuy2006}:
%Dictionary of parameters
%Janson a = \sigma Puyhaubert
%Janson w_0 = a_0 Puyhaubert
%Janson b_0 = b_0 Puyhaubert
%Janson d = \alpha Puyhaubert
%Janson c = \sigma-\alpha Puyhaubert
\begin{equation*}
F(z,u)=u^{w_0}(1-\sigma z)^{-b_0/\sigma}\Big(1-u^{a}\big(1-(1-\sigma z)^{a/\sigma}\big)\Big)^{-w_0/a}.
\end{equation*}

Putting aside the prefactor $u^{w_0}$, and after a change of variable $u^a \mapsto u$, 
this equation can be interpreted as an extended critical composition scheme 
\begin{equation}\label{equation_urns_comp} F(z,u)=M(z)\cdot G\big(uH(z)\big),\end{equation}
involving the exponential generating functions with nonnegative integer coefficients
\begin{equation*}
M(z)=(1-\sigma z)^{-b_0/\sigma}, \ \ G(z)=(1-a z)^{-w_0/a},
\text{ and } H(z)=(1-(1-\sigma z)^{a/\sigma})/a. 
\end{equation*} 

The fact that the singular exponents depend on $b_0$ and $w_0$
explains en passant why the limit distribution of ${\mathcal W}_n$ differs
according to the initial composition of the urn.
Indeed, as the number of white balls at time $n$ satisfies 
\begin{equation*} \P\{ \mathcal{W}_n=a k+w_0\} =\frac{n![z^n u^k]F(z,u)}{n![z^n]F(z,1)}
=\frac{g_k}{k!} \frac{[z^n] H(z)^k M(z)}{[z^n] F(z,1)}, \end{equation*}
we can apply Theorem~\ref{TheExtended} and we then get the following limit distributions of $\mathcal{W}_n$ 
for balanced triangular urn models, completing and extending earlier results~\cite[Theorem~1.8]{Jan2006}:
\begin{coroll}
Let $\mathcal{W}_n$ be the random variable for the number of white balls in a balanced triangular P\'olya urn with initially $w_0>0$ white and $b_0\ge 0$ black balls. 
Then, we have a convergence in distribution, with convergence of all moments,
towards a three-parameter Mittag-Leffler distribution (see Definition~\ref{def:BetaMittagLeffler})
\begin{equation*}
	\frac{\mathcal{W}_n}{a n^{a/\sigma}} \cmom \BML\left(\frac{a}{\sigma},\frac{w_0}{a},\frac{b_0}{a}\right).
\end{equation*} 
\end{coroll}
\begin{remark}[Almost sure convergence and beyond]\label{remark:almostsure}
This limit was also recently identified by Goldschmidt, Haas, and Sénizergues~\cite{GoldschmidtHaasSenizergues2020} for urns with non-integer weights: 
A link with the Chinese restaurant model (for $b_0=0$) leads to a Mittag-Leffler distribution, then they show that the impact of $b_0>0$ on the process leads to a distribution with an additional beta law factor. It is interesting to stress that their approach implies an almost sure convergence.
Note that the fluctuations around this almost sure limit are known:
A second-order central limit theorem (that is, the random variable minus its almost sure limit converges, rescaled, to a Gaussian distribution), as well as a law of the iterated logarithm was obtained using a discrete martingale~\cite{KubaSulzbach2017}. What is more, following Gouet~\cite{Gouet1993}, a continuous-time reparametrization 
leads to a functional second-order limit theorem for balanced urn models. 
There is currently no systematic way to obtain an almost sure convergence for all combinatorial models covered by our composition schemes; however, in a few cases (e.g., for walks, trees, and maps), some ad-hoc clever constructions entail this almost sure convergence~\cite{Marchal2003, Marchal2010,Miermont2009,Marckert2021,RoeslerRueschendorf2001}.
\end{remark}

Note that applying Theorem~\ref{TheRefined} to the size-refined version of the
composition scheme~\eqref{equation_urns_comp}, we get factorial moments of mixed Poisson type for a size-refined random variable $X_{n,j}$
and the corresponding limit laws. However, the combinatorial interpretation of the random variable(s) $X_{n,j}$ is more involved and will be given elsewhere.
%%seq(a,b)^3= M^2 : shuffle/OGF/EGF. M* (factor M with w0) = histories : splitting them in pairs, etc. 
%%more general combinatorial explanations of factorizations for urns, also for PGF/MGF/laws, etc. 

We stress the fact that the methods and results presented in this article are thus holding both for ordinary generating functions (typically used for unlabelled structures) and 
for exponential generating functions (typically used for labelled structures); see~\cite{FlaSe2009}.

This concludes the list of applications for our results on the extended and size-refined composition schemes.
We now give some extensions of our work to other schemes. 
\pagebreak

\section{Further extensions} \label{sec:outlook}
\subsection{Critical cycle scheme}
Many combinatorial structures are cycles of more basic building blocks 
(e.g., cyclic permutations or functional applications are cycles of Cayley trees).
If one marks the number of such basic building blocks, this corresponds to 
\begin{equation*}
\mathcal{F}=\mathcal{G}(\mathcal{H})=\Cyc(\mathcal{H})\quad\Longrightarrow\quad F(z,u)=-\log\big(1-u H(z)\big),
\end{equation*}
where $\mathcal{G}=\Cyc$ denotes the cycle operator. 
This scheme is analysed in Flajolet and Sedgewick's magnum opus~\cite[page 414]{FlaSe2009} in the supercritical case, 
and we now extend this analysis to the critical case (i.e., by Definition~\ref{def:critical} for $H(\rho_H)=1$). 
Note that the previous sections were assuming Puiseux-like expansions for the generating function $F(z,1)$ 
at its dominant singularity $z=\rho=\rho_H$. Now, for critical cycle schemes, $F$ does not have a Puiseux expansion, 
so the previous results need to be adapted.

Let us begin with an example: For $H(z)=1-\sqrt{1-2z}$ we get the sequence
\begin{equation*}
n![z^n]F(z) = n![z^n]\frac12\log\left(\frac{1}{1-2z}\right)=2^{n-1} (n-1)!=(2n-2)!!,\quad n\ge 1,
\end{equation*}
which starts with $1,2,8,48,384,3840,\dots$, and constitutes the entry \href{https://oeis.org/A000165}{A000165} in the \textsc{Oeis}. 
Here, the moments $\E(\fallfak{X_n}s)$ are of order $n^{s/2}$, so the scaling with $1/\sqrt{n}$ leads directly to moment convergence.
This is just one instance of the following more general result.
\smallskip

\begin{theorem}[Critical schemes with a log]
\label{COMPSCHEMECycleThe1}
In a critical cycle composition scheme 
\begin{equation} F(z,u)=-\log\big(1-u H(z)\big), \label{7Fu}\end{equation} 
if $H(z)$ has a singular exponent $0<\lambda_H<1$, 
the core size~$X_n$ (i.e., the number of $\HH$-components in structures of size $n$) has factorial moments given by 
\begin{equation*}
\E(\fallfak{X_n}s)\sim \kappa n^{s\lambda_H} \mu_s, \quad \text{ with } \qquad \kappa=\frac{1}{-c_H}
\quad \text{ and } \quad \mu_s=\frac{\Gamma(s+1)}{\Gamma(s\lambda_H+1)}.
\end{equation*}
The scaled random variable $X_n/(\kappa n^{\lambda_H})$ converges in distribution with convergence of all moments to
a Mittag-Leffler distributed random variable $X\law M_{\lambda_H}$.
\end{theorem}
\begin{remark}
\label{LambdaGNull}
Observe that this scheme leads to a distribution similar (except for a shift in the moments) to the one obtained
for the scheme involving the sequence operator $\mathcal{G}=\Seq$, i.e., $G(z)=\frac{1}{1-z}$, for which one has $c_G=1$, $\lambda_G=-1$, and $\rho_G=1$. 
Alternatively, we may think of this cycle scheme as the limit case of Theorem~\ref{TheExtended} when $\lambda_G\to 0$.
Indeed, for $\lambdaM=0$ the moments~\eqref{MomentBetaStable} of the extended composition scheme can be rewritten into
\begin{equation*}
\E(X^s)=\frac{\Gamma(s-\lambda_G+1)\Gamma(-\lambda_G\lambda_H+1)}{\Gamma(s\lambda_H-\lambda_G\lambda_H+1)\Gamma(-\lambda_G+1)}. 
\end{equation*}
Thus, for $\lambda_G\to 0$ the moments of the random variable $X$ converge to the moments of an ordinary Mittag-Leffler distribution; see Definition~\ref{ex:positivestable}.
Similarly, taking the limit $\lambda_G\to 0$ in Remark~\ref{CoExtended1} gives $\lambda_G=0$ as the tilting parameter, resulting again in the ordinary Mittag-Leffler distribution.
\end{remark}

Now, for $j \in \N$, we can also look at the size-refined scheme 
\begin{equation*}
\mathcal{F}=\Cyc(v\mathcal{H}_{=j} + \mathcal{H}_{\neq j}),
\end{equation*} for which we get the following theorem.
\pagebreak

\begin{theorem}[Size-refined critical schemes with a log]
\label{COMPSCHEMECycleThe2}
In the size-refined critical cycle composition scheme
\begin{equation}
F(z,v)=-\log\Big(1 - \big(H(z) - (1-v) h_j z^j\big)\Big),\quad j\in\N,\label{7Fv} 
\end{equation} if $H(z)$ has a singular exponent $0<\lambda_H<1$, 
then the number $X_{n,j}$ of $\HH$-components of size~$j$ in structures of size $n$ has factorial moments of mixed Poisson type, 
\begin{equation*}
	\E(\fallfak{X_{n,j}}s)=\mppar_{n,j}^s\cdot \mu_s \cdot (1+o(1)),
\end{equation*}
with $\mppar_{n,j}=\frac{\rho_H^{j}}{-c_H} h_j n^{\lambda_H}$ and Mittag-Leffler mixing distribution $X\law M_{\lambda_H}$.
The random variable $X_{n,j}$ converges to one of the three limit laws given in Theorem~\ref{TheRefined},
depending on whether $j=j(n)$ is smaller, equal, or larger than the critical growth range $j=\Theta(n^{\frac{\lambda_H}{1+\lambda_H}})$.
\end{theorem}
\begin{proof}[Proofs of Theorem~\ref{COMPSCHEMECycleThe1} and~\ref{COMPSCHEMECycleThe2}]
The proofs are analogous to those of Theorems~\ref{TheExtended} and~\ref{TheRefined}, and we only point out the differences next.
Let us start with the proof of Theorem~\ref{COMPSCHEMECycleThe1}. 
The factorial moments of order $s$ satisfy
\begin{equation*}
\E(\fallfak{X_n}s)
=\frac{[z^n] \partial_u^s (F)(z,1)}{[z^n]F(z,1)},
\end{equation*}
where $F$ is defined by Equation~\eqref{7Fu}.
As the scheme is critical (i.e.~$H(\rho_H)=1$), one has 
\begin{align*}
F(z,1)=-\log(1-H(z))&= -\log\Big((1-z/\rho_H)^{\lambda_H}(1+o(1))\Big)
\sim 
-\lambda_H\cdot \log\Big(1-z/\rho_H\Big).
\end{align*}
Using the transfer theorems of~\cite{FlaSe2009} we directly obtain 
$[z^n] F(z,1) \sim \lambda_H \smash{\frac{\rho_h^{-n}}{n}}.$
It remains to compute $\partial_u^{s} F$. 
Note that the log function can be replaced by a quasi-inverse using 
\begin{equation*}
\partial_u^{s} \log \left( \frac{1}{1-u} \right) = \partial_u^{s-1}\frac{1}{1-u}.
\end{equation*}
Thus, the $s$th factorial moment is obtained from the asymptotics in~\eqref{eq:pureasymptmoments} computed for $s-1$ and with $G(z)=\frac{1}{1-z}$. 
Then, we obtain the final result:
the normalized moments converge to the moments of a Mittag-Leffler distribution.

For the proof of Theorem~\ref{COMPSCHEMECycleThe2} one replaces~\eqref{7Fu} by~\eqref{7Fv} and $\partial_u^{s}F$ by~$\partial_v^{s} F$.
\end{proof}

\subsection{Multivariate critical composition schemes}
It is possible to generalize the critical composition scheme by looking at combinatorial constructions of the form
\begin{equation*}
\mathcal{F}=\mathcal{M}\times\mathcal{G}_1\big(\mathcal{H}_1\big)\times \mathcal{G}_2 \big(\mathcal{H}_2 \big) \times \dots \times \mathcal{G}_m \big(\mathcal{H}_m\big)
= \mathcal{M}\times \prod_{\ell=1}^m \mathcal{G}_{\ell} \big(\mathcal{H}_{\ell}\big).
\end{equation*}
We measure the size of the $\mathcal{G}_\ell$ component by the variable $u_{\ell}$; accordingly this gives % the multivariate generating function
\begin{equation}\label{MVF}
F(z,u_1,\dots, u_m) = M(z) \cdot \prod_{\ell=1}^{m}G_{\ell}\big(u_{\ell}H_{\ell}(z)\big).
\end{equation}
The random vector $\mathbf{X}_n=(X_{n,1},\dots, X_{n,m})$ measures the sizes of the $\mathcal{G}_\ell$-components,
\begin{equation*}
\P\{X_{n,1}=k_1,\dots, X_{n,m}=k_m\}=\frac{[z^n\, u_1^{k_1}\dots u_m^{k_m}]F(z,u_1,\dots,u_m)}{[z^n]F(z,1,\dots,1)}.
\end{equation*}
\pagebreak

We now introduce a suitable extension of the terms \textit{critical} (Definition~\ref{def:critical}) and \textit{pure} (Definition~\ref{def:pure}) for multivariate schemes. 
We call a multivariate scheme \emph{critical} if all functions $H_{\ell}(z)$ have the identical radius of convergence $\rho_{H_{\ell}}=\rho_H$ such that $\tau_{\ell}:=H_{\ell}(\rho_H)=\rho_{G_{\ell}}$ and $M(z)$ has radius of convergence $\rho_M\geq\rho_H$.
We call a multivariate scheme \emph{pure} if
\begin{itemize}
\item $H_{\ell}(z)$ has a singular exponent $0<\lambda_{H_\ell}<1$ for $1\le \ell \le m$; 
\item $G_{\ell}(z)$ has a singular exponent $\lambda_{G_{\ell}}<0$ for $1\le \ell \le m$;
\item $M(z)$ has a singular exponent $\lambda_M\le 0$ \emph{or} $M(z)$ is analytic at $\rho_H$.
\end{itemize}
We can now state our multivariate result.

\begin{theorem}
\label{TheMV}
In a multivariate pure extended critical composition scheme~\eqref{MVF}, 
the joint moments of the random vector $\mathbf{X}_n=(X_{n,1},\dots, X_{n,m})$
satisfy
\begin{align*}
\E(X_{n,1}^{s_1}\dots X_{n,m}^{s_m})\sim \mu_{s_1,\dots,s_m} \prod_{\ell=1}^m n^{s_{\ell}\lambda_{H_{\ell}}}\kappa_{\ell}^{s_{\ell}},
\end{align*}
with $\kappa_{\ell}=\smash{\frac{\tau_{H_{\ell}}}{-c_{H_{\ell}}}}$ and $\mu_{s_1,\dots,s_m}$ given by
\begin{equation}\label{eq:momDir}
\mu_{s_1,\dots,s_m}=\frac{\Gamma\big(-\sum_{\ell=1}^{m}\lambda_{G_{\ell}}\lambda_{H_{\ell}}-\lambda_M\big)}{\Gamma\big(\sum_{\ell=1}^{m}s_{\ell}\lambda_{H_{\ell}}-\sum_{\ell=1}^{m}\lambda_{G_{\ell}}\lambda_{H_{\ell}}-\lambda_M\big)}\prod_{\ell=1}^{m}\frac{\Gamma(s_{\ell}-\lambda_{G_{\ell}})}{\Gamma(-\lambda_{G_{\ell}})}.
\end{equation}
Consequently, one gets a convergence in distribution and in moments
\begin{equation*} \left( \frac{X_{n,1}}{\kappa_1 n^{\lambda_1}},\dots,\frac{X_{n,m}}{\kappa_m n^{\lambda_m}} \right) \cmom \mathbf{X}, \end{equation*}
where $\mathbf{X}$ is determined by its joint moment sequence $\mu_{\mathbf{s}}=\mu_{s_1,\dots,s_m}$. 
Moreover, the random vector $\mathbf{X}$ has a scaled Dirichlet-stable product distribution,
\begin{equation}\label{MLRV}
\mathbf{X}=(X_1,\dots,X_m)\law (V_{1}\cdot W_1^{\lambda_{H_{1}}},\dots,V_{m}\cdot W_m^{\lambda_{H_{m}}}),
\end{equation}
where $\mathbf{W}=(W_1,\dots,W_m,W_{m+1})$ follows a Dirichlet distribution 
\begin{equation*} %\label{WDir}
\mathbf{W} \law \Dir(-\lambda_{G_{1}}\lambda_{H_{1}},\dots,-\lambda_{G_{m}}\lambda_{H_{m}},-\lambda_M),
\end{equation*}
and where the $V_\ell$'s, for $1\le\ell\le m$, are $m$ independent two-parameter Mittag-Leffler distributions 
$V_{\ell}\law \ML(\lambda_{H_{\ell}}, -\lambda_{G_{\ell}}\lambda_{H_{\ell}})$, 
also independent of $\mathbf{W}$.
\end{theorem}

\begin{proof}We proceed similarly to the proof of the first part of Theorem~\ref{TheExtended}. First, the mixed factorial moments of $\mathbf{X}_n$, which are obtained by differentiation and extraction of coefficients:
\begin{equation*}
\E\big(\fallfak{X_{n,1}}{s_1}\cdots \fallfak{X_{n,m}}{s_m}\big)=\frac{[z^n]\partial_{u_1}^{s_1}\dots \partial_{u_m}^{s_m}(F)(z,1,\dots, 1)}{[z^n]F(z,1,\dots, 1)}.
\end{equation*}
The differentiation with respect to $u_{\ell}$ only affects the factor $G_{\ell}\big(u_{\ell}H_{\ell}(z)\big)$, 
leading to a singular expansion covered in Section~\ref{sec:prelimsingular}.
Extraction of coefficients then gives an asymptotic expansion of the factorial moments. Converting all the factorial moments into moments using~\eqref{COMPSCHEMEconversion} gives the desired asymptotics and moments in~\eqref{eq:momDir}. 

It remains to identify the distribution. 
To this aim, note that a Dirichlet distributed random vector $(W_1,\dots,W_{m+1})\law\Dir(a_1,\dots,a_{m+1})$ with positive parameters $a_1,\dots,a_{m+1}$
has a 
density function 
supported on the $m$ simplex 
$\{(x_1,\dots,x_{m+1}) \in {\mathbb R}_{\geq 0}^{m+1} ~|~ \sum_{j=1}^{m+1}x_j=1 \}$:
\begin{equation*}
f(x_1,\dots,x_{m+1})=\frac{\Gamma(\sum_{j=1}^{m+1}a_j\big)}{\prod_{j=1}^{m+1}\Gamma(a_j)}
 \prod_{j=1}^{m+1}x_j^{a_j-1}.
\end{equation*}
Accordingly, the joint moments are given by
\begin{equation*}
\E\big(W_1^{s_1}\cdots W_{m+1}^{s_{m+1}}\big)
= \frac{\Gamma(\sum_{j=1}^{m+1}a_j)}{\Gamma\big(\sum_{j=1}^{m+1}(s_j+a_j)\big)}
\frac{\prod_{j=1}^{m+1}\Gamma(s_j+a_j)}{\prod_{j=1}^{m+1}\Gamma(a_j)}.
\end{equation*}

Now, consider a random vector $(Z_1,\dots,Z_m)$ satisfying
\begin{equation*}
(Z_1,\dots,Z_m)\law (V_{1}\cdot W_1^{\alpha_1},\dots,V_{m}\cdot W_m^{\alpha_m})
\end{equation*}
with two-parameter Mittag-Leffler distributions $V_{\ell}\law \ML(\alpha_{\ell},a_{\ell})$ for ${\ell}=1,\dots,m$, such that all random variables are mutually independent.
Using the closed form~\eqref{momgenML} we get
\begin{flalign*}&&
\E\big(Z_1^{s_1}\cdots Z_m^{s_m}\big)
=\frac{\Gamma(\sum_{\ell=1}^{m+1}a_{\ell})}{\Gamma\big(a_{m+1}+\sum_{\ell=1}^{m}(\alpha_{\ell} s_{\ell}+a_{\ell})\big)}
\prod_{\ell=1}^m\frac{\Gamma(s_{\ell}+a_{\ell}/\alpha_{\ell}) }{\Gamma(a_{\ell}/\alpha_{\ell})}.
&& 
\end{flalign*}
Comparing this expression with the moments~\eqref{eq:momDir}, the claim follows.
\end{proof}

\begin{remark}
The marginals $X_{\ell}$ of the random vector $\mathbf{X}$ are also covered by Theorem~\ref{TheExtended}. The random vector $\mathbf{X}$ is closely related to Poisson--Dirichlet distributions $\text{PD}(\alpha,\beta)$, \cite{James2013,James2015,PY1997}
and the joint limit law of node degrees in preferential attachment trees or generalized plane-oriented recursive trees~\cite{Mor2005}; see also the subsequent example.
\end{remark}

Now, if one considers the multivariate size-refined scheme
\begin{equation*}
\mathcal{F}=\mathcal{M}\times \prod_{\ell=1}^m \mathcal{G}_{\ell}\big(\mathcal{H}_{\ell,\neq j_{\ell}} + v_{\ell}\mathcal{H}_{\ell,= j_{\ell}}\big),
\end{equation*}
one gets the following multivariate version of Theorem~\ref{TheRefined}.

\begin{theorem}[Multivariate pure size-refined critical scheme]
\label{TheMVRefined}
In a multivariate pure size-refined critical composition scheme 
\begin{equation*}
F(z,v_1,\dots, v_m)=M(z)\cdot\prod_{\ell=1}^{m}G_{\ell}\big(H_\ell(z) - (1-v_\ell)h_{\ell,j_\ell} z^{j_\ell}\big)
\end{equation*}
the random variables $X_{n,\ell,j_\ell}$, which count the number of $H_\ell$-components of size $j_\ell$, have joint factorial moments of mixed Poisson type:
{%
\begin{equation*}
	\E\big(\fallfak{X_{n,1,j_1}}{s_1}\dots\fallfak{X_{n,m,j_m}}{s_m}\big)=\mu_{s_1,\dots,s_m}\cdot\prod_{\ell=1}^m\mppar_{n,\ell,j_\ell}^{s_\ell}\cdot (1+o(1)),
\end{equation*}}%
with $\mppar_{n,\ell,j_\ell}=\frac{\rho_{H_\ell}^{j_\ell}}{-c_{H_\ell}} h_{\ell,j_\ell} n^{\lambda_{H_\ell}}$ and joint mixing distribution $\mathbf{X}=(X_1,\dots,X_m)$ as in Equation~\eqref{MLRV}.
\smallskip

Let $X_\ell$, for $1\le \ell \le m$, denote the marginal distribution of the $\ell$th coordinate of $\mathbf{X}=(X_1,\dots,X_m)$.
For $n\to\infty$, the limiting distributions of $X_{n,\ell,j_\ell}$ jointly undergo mixed Poisson type phase transitions with mixing distributions $X_\ell$. The phase transitions depend on the growth of $j_\ell=j_\ell(n)$, with critical growth ranges given by $j_\ell=j_\ell(n)=\Theta(n^{\frac{\lambda_{H_\ell}}{1+\lambda_{H_\ell}}})$. 

In particular, for $j_\ell(n)\sim \mppar_\ell\cdot n^{\frac{\lambda_{H_\ell}}{1+\lambda_{H_\ell}}}$, 
the random vector $\mathbf{X}_{n,\mathbf{j}}=( X_{n,1,j_1},\dots,X_{n,m,j_m})$ converges
in distribution with convergence of all (factorial) moments to a multivariate 
distribution $\MPo(\boldsymbol{\mppar}\mathbf{X})$. 
\end{theorem}	
\pagebreak

For properties of multivariate mixed Poisson distributions we refer to~\cite{FLT2004} or~\cite{KuPa2014}.
The key tool for proving Theorem~\ref{TheMVRefined} is the following multivariate extension of Lemma~\ref{COMPSCHEMElemmaMixedPoisson}.

\begin{lem}[Joint factorial moments and limit laws of mixed Poisson type]
\label{COMPSCHEMElemmaMixedPoissonMV}
Let $(\mathbf{X}_n)_{n\in\N}$ denote a sequence of $m$-dimensional random vectors, whose factorial moments are asymptotically of \emph{mixed Poisson type}, i.e., they satisfy for $n \to \infty$ the asymptotic expansion
\begin{equation*}
\E(\fallfak{\mathbf{X}_n}{\mathbf{s}})=
\E\big(\fallfak{X_{n,1}}{s_1}\dots\fallfak{X_{n,m}}{s_m}\big)
=\mu_{s_1,\dots,s_m}\cdot\prod_{\ell=1}^m\mppar_{n,\ell}^{s_\ell}\cdot (1+o(1)),\quad s\ge 1,
\end{equation*}
with $\mu_{s_1,\dots,s_m}\ge 0$, and $\mppar_{n,\ell}>0$ for $1\le \ell \le m$. Furthermore, assume that the sequence of 
joint moments $(\mu_{\mathbf{s}})_{\mathbf{s}\in\N^m}$ determines a unique distribution 
$\mathbf{L}=(L_1,\dots,L_m)$. Then, for $n\to\infty$, one has the following joint limit distributions:
\begin{itemize}
\item[(i)] If $\mppar_{n,\ell}\to\infty$, the random variable $\frac{X_{n,\ell}}{\mppar_{n,\ell}}$ converges in distribution, with convergence of all moments, to $L_\ell$. 

\item[(ii)] If $\mppar_{n,\ell}\to\mppar\in (0,\infty)$, the random variable $X_{n,\ell}$ converges in distribution, with convergence of all moments, to 
$Y\law \MPo(\mppar L_\ell)$.

\item[(iii)] If $\mppar_{n,\ell}\to 0$, 
$X_{n,\ell}$ converges to a Dirac distribution: $X_{n,\ell}\claw 0$.
\end{itemize}
\end{lem}

\begin{proof}
We follow the proof of the univariate case~\cite[Lemma 2]{KuPa2014}. 
Assume that the set $\{1,\dots,m\}$ decomposes into two disjoint sets $C$ and $D$ such that for indices 
$k\in C$ we have $\lambda_{n,k}\to\infty$, whereas for an index $k\in D$ it holds $\lambda_{n,k}\to\rho_k$ for $n\to\infty$.
We observe that the mixture of joint raw moments and joint factorial moments converge:
\begin{equation*}
\E\left(\Big(\prod_{k\in C}\frac{X_{n,k}^{s_{k}}}{\lambda_{n,k}^{s_k}}\Big)\cdot\Big(\prod_{k\in D}\fallfak{X_{n,k}}{s_{k}}\bigg)\right)
\to \mu_{s_1,\dots,s_m}\cdot\prod_{k\in D}\rho_{k}^{s_k}.
\end{equation*}
The latter joint moment sequence, both raw and factorial moments, is exactly the joint moment sequence of a random vector $\mathbf{Z}=(Z_1,\dots,Z_m)$, with 
\begin{equation*}
\forall k\in C\colon \ Z_k\law L_k,\qquad \forall k\in D\colon \ Z_k\law \MPo(\rho_k L_k),
\end{equation*}
such that
\begin{equation*}
\E\left(\Big(\prod_{k\in C}Z_k^{s_{k}}\Big)\cdot\Big(\prod_{k\in D}\fallfak{Z_{k}}{s_{k}}\bigg)\right)
= \mu_{s_1,\dots,s_m}\cdot\prod_{k\in D}\rho_{k}^{s_k}.\qedhere
\end{equation*}
\end{proof}

\begin{proof}[Proof of Theorem~\ref{TheMVRefined}]
The proof is very similar to the proofs of Theorems~\ref{TheRefined} and~\ref{TheMV}, so we will be brief again. 
The mixed factorial moments of $\mathbf{X}_{n,j_\ell}$ are obtained by differentiation and extraction of coefficients:
\begin{equation*}
\E\big(\fallfak{X_{n,1,j_1}}{s_1}\cdots \fallfak{X_{n,m,j_m}}{s_m}\big)=\frac{[z^n]\partial_{v_1}^{s_1}\dots \partial_{v_m}^{s_m}(F)(z,1,\dots, 1)}{[z^n]F(z,1,\dots, 1)}.
\end{equation*}
The differentiation with respect to $v_{\ell}$ only affects the factor 
\begin{equation*}
G_{\ell}\big(H_\ell(z) - (1-v_\ell)h_{\ell,j_\ell} z^{j_\ell}\big),
\end{equation*}
leading to singular expansions covered in Section~\ref{sec:prelimsingular} and additional factors
$h_{\ell,j_\ell}^{s_\ell}$, ${1\le \ell \le m}$. The asymptotics of these factors for $j_\ell\to\infty$ are all governed by~\eqref{COMPSCHEMEExpanHj}. The individual singular expansions are similar to~\eqref{COMPSCHEMEExpanGAbleit}. 
Thus, extracting the coefficient $z^{n-s_1j_1-\dots-s_mj_m}$ from 
%\begin{equation*}
$M(z)\cdot\prod_{\ell=1}^{m}G^{(s_\ell)}_{\ell}\big(H_\ell(z)\big)$
%\end{equation*}
 leads to the asymptotic expansion of the joint factorial moments. 
Lemma~\ref{COMPSCHEMElemmaMixedPoissonMV} then yields the stated limit law. 
\end{proof}
\pagebreak

We end this section with four examples covered by the multivariate scheme~\eqref{MVF}.

\begin{example}[$m$-bundled plane-oriented recursive trees]
\label{COMPSCHEMEMulVar-Examplekbund}
We have previously encountered $3$-bundled trees in Section~\ref{ExBilabelled} in the framework of bilabelled trees. 
In the following we study ordinary increasing trees~\cite{BerFlaSal1992,Drmota2009,JKP2011,KuPa2007,MSS1993,PanPro2007}, where each node has only one label. As before, given a degree-weight sequence $(\varphi_{j})_{j \ge 0}$, the corresponding degree-weight generating function is defined by 
\begin{equation*}\varphi(t) = \sum_{j \ge 0} \varphi_{j} t^{j}.\end{equation*}
The associated family $\mathcal{T}$ of increasing trees can be described by the following symbolic equation using the boxed product (see Section~\ref{ExBilabelled}):
\begin{equation*}
 \mathcal{T} = \mathcal{Z}^{\Box} \ast \varphi\big({\mathcal{T}}\big).
\end{equation*}
For the exponential generating function 
\begin{equation*} T(z) = \sum_{n \ge 1} T_{n} \frac{z^{n}}{n!},\end{equation*} 
we thus have the following equation
\begin{equation*}
T'(z)=\varphi(T(z)),\quad T(0)=0. 
\end{equation*}
We are interested in families of generalized plane-oriented recursive trees with degree-weight generating function
$\varphi(t)=1/(1-t)^m$, with $m\in\N$. 

For $m=1$ we get the ordinary plane-oriented recursive tree, whereas for $m>1$ we get the so-called $m$-bundled trees, with generating function
\begin{equation*}
T(z)=1-\big(1-(m+1)z\big)^{1/(m+1)}.
\end{equation*}
One may think of each node holding $m-1$ additional separation bars~\cite{JKP2011}, which can be regarded as a special edge type. 
Naturally, this refines the root degree $X_n$ in a random tree of size $n$, since we may look at the number of nodes attached to the root in a specific cluster, induced by the $m-1$ bars:
\begin{equation*}
X_n=\sum_{\ell=1}^{m}X_{n,\ell}.
\end{equation*}

By construction, the random variables $X_{n,\ell}$ are exchangeable, but not independent. 
Writing $\mathbf{u}^\mathbf{X_n} = {u_1}^{X_{n,1}} \cdots {u_m}^{X_{n,m}}$,
the generating function 
\begin{equation*}
T(z,\mathbf{u})=\sum_{n\ge 1} T_n \E(\mathbf{u}^{\mathbf{X_n}} ) \frac{z^n}{n!}
\end{equation*}
of the random vector $\mathbf{X}_n=(X_{n,1},\dots,X_{n,m})$ 
is given by
\begin{equation*}
\frac{\partial}{\partial z}T(z,\mathbf{u})= \prod_{\ell=1}^{m}\frac{1}{1-u_{\ell}T(z)}.
\end{equation*}
This refinement is covered by the multivariate scheme~\eqref{MVF} 
(with a shift due to the derivative:
Theorem~\ref{TheMV} thus gives here the asymptotic behaviour of $X_{n+1,\ell}$, which, of course,
then gives the asymptotic behaviour of $X_{n,\ell}$). Similarly, one may study
the outdegree of a node labelled~$j$, as well as multiple nodes, leading
to closely related generating functions~\cite{James2015,KuPa2007,Mor2005}.
\end{example}
\pagebreak

\begin{example}[Bilabelled increasing trees and refined root degree]
Continuing from Section~\ref{ExBilabelled}, we can refine the root-degree $X_n$ in $3$-bundled bilabelled increasing trees of size~$2n$, such that
$X_{n}=X_{n,1}+X_{n,2}+X_{n,3}$,
where $X_{n,\ell}$ is the root-degree in the $\ell$th bundle.
Note that the random variables $X_{n,\ell}$ are exchangeable. 
The corresponding generating function
{%
\begin{equation*}
T(z,u_1,u_2,u_3)=\sum_{n\ge 1}T_n\E(u_1^{X_{n,1}}u_2^{X_{n,2}}u_3^{X_{n,3}})\frac{z^{n}}{n!}
\end{equation*}}%
then satisfies
\begin{equation*}
\frac{\partial^2}{\partial z^2}T(z,u_1,u_2,u_3)=\prod_{\ell=1}^{3}\frac{1}{1-u_{\ell}(1-\sqrt{1-z^2})}.
\end{equation*}
Except for the non-standard shift in the random variable, the problem corresponds directly to a multivariate pure critical composition scheme.
\end{example}

\begin{example}[Returns in coloured bridges and walks]
Consider $k$-coloured bridges $B_k$ as defined in Section~\ref{sec:mcolouredbridges}:
A bridge is coloured in exactly $k$ colours, where each colour consists of a non-empty bridge. 
Combinatorially, we append $k$ non-empty bridges one after the other.
We are interested in the individual number of returns to zero in each of the first $k_1$ bridges, followed by $k_2$ additional bridges, such that $k_1+k_2=k$.
Reusing the combinatorial construction~\eqref{eq:mcoloredbridgeBGF}, the multivariate generating function $B_k(z,\mathbf{u})$ satisfies 
\begin{align*}
	B_k(z,u_1,\dots,u_{k_1}) &= \left(\prod_{j=1}^{k_1}\left(\frac{1}{1-u_j\left(1-\frac{1}{B(z)}\right)}-1\right)\right) (B(z)-1)^{k_2}. 
\end{align*}
By our previous reasoning we see that the corresponding random variables are exchangeable and the multivariate scheme directly applies.
Moreover, we can also consider walks and a refined generating function $W_k(z,\mathbf{u})$ of $k$-coloured walks with the tail coloured in the same colour as the final bridge, 
keeping track of the individual returns to zero of the first $k_1$ bridges,
\begin{align*}
	W_k(z,\mathbf{u}) &= (1 + B_{k}(z,\mathbf{u})) \frac{W(z)}{B(z)}.
\end{align*}
Again, the corresponding random variables are exchangeable and the multivariate scheme directly applies again.
\end{example}

\begin{example}[Triangular urn models and node degrees in generalized recursive trees]
We discuss a specific balanced triangular urn model with $m$ colours, which generalizes the case with $2$ colours from Section~\ref{SubSecTriangular}.
Our urn model is specified by the following $m \times m$ balanced ball replacement matrix $M$ with $\myalpha_k,\myybeta_k \in \N$ and 
$\myalpha_k+\myybeta_k=\sigma$, for  $1\le k\le m$: 
\begin{equation}
M=
\left(
\begin{matrix}
\myalpha_1 &0 & \dots &0 &\myybeta_1\\
0 &\myalpha_2 &\dots &0& \myybeta_2\\
\vdots & \ddots &\ddots & &\vdots\\
0 & \dots & 0 &\myalpha_{m-1} & \myybeta_{m-1}\\
0& 0&\dots &0 &\myalpha_{m} +\myybeta_{m}
\end{matrix}
\right).
\label{COMPSCHEME-MV-UrnM}
\end{equation}
We consider the random variables $A_{n,k}$ which count the number of balls of colour $k$ after $n$ draws.
%We remind the reader that 
%\[T_n=n\cdot \sigma +T_0 = \sum_{k=1}^{m+1}A_{n,k},\]  %T_n not yet defined, even if implicit
%such that $A_{n,m+1}$ is completely determined by the random vector $(A_{n,1},\dots,A_{n,m})$. 
To analyse this urn model, it is natural to introduce the history generating function 
\begin{equation*}
F(z,x_1,\dots,x_m) = \sum_{n,k_1,\dots,k_{m} \geq 0} f_{n,k_1,\dots,k_{m}} x_1^{k_1} \dots x_{m}^{k_{m}}\frac{z^n}{n!},
\end{equation*}
where $f_{n,k_1,\dots,k_{m}}$ is equal to the number of transitions leading after $n$ steps to a configuration with $k_1,\dots,k_{m}$ balls of colours $1,\dots,m$, respectively. 
Thus, the joint distribution of the number of balls of each colour 
is given by 
\begin{equation*}\P\{A_{n,1}=k_1,\dots, A_{n,m}=k_m\}=\frac{[z^n\, x_1^{k_1}\dots x_m^{k_m}]F(z,x_1,\dots,x_m)}{[z^n]F(z,1,\dots,1)}.\end{equation*}
where $F$ has, surprisingly, the following simple algebraic closed form.
\begin{prop}
\label{PropMUrn}
The history generating function $F(z,x_1,\dots,x_{m})$, associated with the balanced triangular urn model with ball replacement matrix $M$ in~\eqref{COMPSCHEME-MV-UrnM}, % and initial conditions $A_{0,k}$, $1\le k\le m$, 
is given by 
\begin{equation}\label{Hz}
F(z,x_1,\dots,x_{m})= \prod_{k=1}^{m} X_k(z)^{A_{0,k}} 
\end{equation}
where  
\begin{equation}\label{Xm}
X_{m}(z)=x_{m} \cdot \big(1-\sigma x_{m}^{\sigma}z\big)^{-1/\sigma}
\end{equation}
and, for  $1\leq k < m$,
\begin{equation}\label{Xk}
X_{k}(z)=x_{k} \cdot \Big(1-x_{k}^{\myalpha_k}x_{m}^{-\myalpha_k}\big(1-(1-\sigma x_{m}^{\sigma}z)^{\myalpha_k/\sigma}\big)\Big)^{-1/\myalpha_k}.
\end{equation}
\end{prop}
\begin{proof}
We follow the history-counting approach introduced by Flajolet et al.~in~\cite{FlajoletGabarroPekari05,FlaDumPuy2006},
which shows that 
the evolution of the urn is captured by the matricial equation $\partial_z {\mathbf X} = {\mathbf X}^{1+M}$,
that is, one has the differential system
\begin{align*}
\dot{X}_k(z) &= X_k(z)^{1+\myalpha_k} X_{m}^{\myybeta_k}(z) \text{\, (for $1\le k< m$)} \qquad \text{and} \\
\dot{X}_{m}(z) &= X_{m}^{1+\sigma}(z),
\end{align*}
with initial conditions $X_{k}(0)=x_{k}$. 
Now, by separation of variables, we obtain Formula~\eqref{Xm}.\footnote{There is a small sign typo in~\cite[page~94]{FlaDumPuy2006} where, in the corresponding result for the $2\times 2$ model, $y_0^{-\sigma}$ should be replaced by $y_0^{\sigma}$.}
Then, by  integration, we get Formula~\eqref{Xk}.
 Finally, using the basic isomorphism between differential systems and history generating functions~\cite[Theorem 6]{FlaDumPuy2006}, 
we obtain the desired closed form~\eqref{Hz} for the generating function of urn histories.
\end{proof}
As the urn is balanced, we add the same number of balls at each step. Therefore, knowing  the number of balls of colour $1$ to $m-1$
immediately gives the number of balls of colour~$m$, so, without loss of information, we can stop tracking this parameter (i.e., we set $x_m=1$).
One then recognizes that this urn model has a multivariate pure critical composition scheme 
(with $M(z)=(1-\sigma z )^{-1/\sigma}$, and  for $k<m$, $G_k(z)= (1- z)^{-1/\myalpha_k}$,
 $H_k(z)=1-(1-\sigma z)^{\myalpha_k/\sigma}$, 
and $u_k=x_{k}^{\myalpha_{k}}$);
thus Theorem~\ref{TheMV} applies: The corresponding joint distribution (i.e., the \emph{number of drawings} of balls of each colour)
converges to a %scaled 
Dirichlet-stable product distribution. \smallskip

Note that the node degrees in generalized plane-oriented recursive trees can be modelled by such urns~\cite{Jan2005,Mor2005}. 
So, as an application, let us set $\myalpha_k=1$ and $\myybeta_{k}=m-1$ for $1 \leq k \leq m$ and $A_{0,k}=1$ (for $1\le k < m$) and $A_{0,m}=0$ as initial values.
Then, the random variables $A_{n,1}$ up to $A_{n,{m-1}}$ are exchangeable and count the refined root degree of $(m-1)$-bundled plane-oriented recursive trees, i.e., the number of children of the root in each bundle.
\smallskip

Moreover, the urn model $M$ can be used to study the joint degree distribution of
the nodes labelled $1,2,\dots,m-1$ in generalized plane-oriented recursive trees, as well as to obtain second order asymptotics and central limit theorem analogues.
\end{example}
\newpage

\section{Conclusion} 
In this article, we introduced the three-parameter Mittag-Leffler distribution and showed its universality 
for capturing transition phases related to critical compositions. 
It opens a new chapter in the long history of the Mittag-Leffler function.
This special function, $E_\alpha(t):=\sum_{k\geq 0} \frac{1}{\Gamma(\alpha k+1)} t^k$,
was  introduced by Mittag-Leffler in 1903~\cite{MittagLeffler1903a,MittagLeffler1903b}.
It was generalized  to a two-parameter function
$E_{\alpha,\beta}(t):=\sum_{k\geq 0} \frac{1}{\Gamma(\alpha k+\beta)} t^k$ by Wiman in 1905~\cite{Wiman1905a} and to a three-parameter function 
$E_{\alpha,\beta}^{\gamma}(t):=\sum_{k=0}^\infty \frac{\Gamma(k+\gamma)}{\Gamma(\alpha k+\beta) \Gamma(\gamma)} \frac{t^k}{k!}$
by Prabhakar in 1971~\cite{Prabhakar1971}.
While these functions were used for decades by mathematicians and physicists working on differential equations (see Remark~\ref{electromagnetismremark} on a link with electromagnetism), 
special functions,  or
 fractional calculus, it is a nice surprise that this three-parameter Mittag-Leffler function is 
the key 
to unify different fundamental distributions in probability theory, as summarized in Table~\ref{tab:history}.

This work makes explicit the limit laws of structures associated with schemes like $G(H(z))M(z)$, 
where we additionally mark either the number of ${\mathcal H}$-components, or the number of ${\mathcal H}$-components of a given size. 
We focused on the technically more delicate and mathematically richer case where $G$, $H$, and $M$ are simultaneously singular.
We proved that when these functions have algebraic dominant singularities with exponents between 0 and 1,
the limit laws of the schemes are moment-tilted stable distributions and distributions involving Mittag-Leffler laws. 
In Table~\ref{TableO1} we give an overview of the covered schemes and their limit laws, 
where we have convergence in distribution \emph{and} convergence of all moments.

In the extended scheme, in which one follows the total number of ${\mathcal H}$-components, 
we proved the appearance of three régimes for the corresponding limit law (continuous, Boltzmann, or a linear combination of these two),
depending on the relation of the singular exponents of $G$, $H$, and~$M$; see Table~\ref{tab:extended} on page~\pageref{sec2}.
In the size-refined scheme, in which one follows the number of ${\mathcal H}$-components of a given size, 
we proved the appearance of mixed Poisson distributions (see Definition~\ref{COMPSCHEMEdef1}) and a double phase transition for the limit law from continuous to discrete to degenerate,
with explicit threshold sizes depending on the exponents; see Table~\ref{tab:size-refinedphases} on page~\pageref{tab:size-refinedphases}.  

We also presented several extensions (logarithmic singularities, multivariate cases) and a variety of applications to important probabilistic and combinatorial objects.
This allowed us to obtain new results for
the core size of supertrees,
the number of subtrees in different increasing trees,
the returns to zero and sign changes in walks and bridges, 
the table sizes in the Chinese restaurant process, 
and the number of balls in some urn models. 
\begin{table}[hb]
\setlength{\tabcolsep}{4.5pt}
\begin{tabular}{ccc}
\toprule
Distribution & \begin{tabular}{c} Moment generating function \end{tabular} & History \\
\midrule
\begin{tabular}{c}
Mittag-Leffler $\ML(\alpha)$\\[1.5mm]
(see Definition~\ref{ex:mittagleffler})
\end{tabular} & 
\begin{tabular}{c}
$\displaystyle{\E\big(e^{tX}\big)=E_\alpha(t)}$\\
\end{tabular} &
\begin{minipage}{5.2cm}
%Special function $E_\alpha$ introduced by Mittag-Leffler in 1903~\cite{MittagLeffler1903a,MittagLeffler1903b}.
%As MGF, it 
MGF appeared as Laplace transform of stable distributions and for the local time of Markov processes~\cite{Feller1949,DarlingKac1957}.
\end{minipage}
\\
\midrule
\begin{tabular}{c}
two-parameter\\
 Mittag-Leffler $\GML(\alpha,\beta)$\\[1.5mm]
(Definition~\ref{ex:genmittagleffler})
\end{tabular}
& \begin{tabular}{c}
$\displaystyle{\E\big(e^{tX}\big)=\Gamma(\beta') E_{\alpha,\beta'}^{\gamma'}(t)}$\\[1.5mm]
$(\beta', \gamma') =(\beta,\frac{\beta}{\alpha})$
\end{tabular}
& 
\begin{minipage}{5.2cm} 
MGF introduced by Pitman in 2002~\cite{Pitman2006} for the Chinese restaurant process,
and  used for a line-breaking construction of stable trees  by Goldschmidt and Haas~\cite{GoldschmidtHaas2015}; it 
also occurs for triangular P\'olya urns (see Flajolet, Dumas, Puyhaubert~\cite{FlaDumPuy2006} and Janson~\cite{Jan2006,Jan2010}).
\end{minipage}
\\
\midrule 
\begin{tabular}{c}
three-parameter \\
Mittag-Leffler 
$\BML(\alpha,\beta,\gamma)$\\[1.5mm]
(Definition~\ref{def:BetaMittagLeffler})
\end{tabular} 
& \begin{tabular}{c}
$\displaystyle{\E\big(e^{tX}\big)= \Gamma(\beta') E_{\alpha,\beta'}^{\gamma'}(t)}$\\[1.5mm]
$(\beta',\gamma')=(\beta+\gamma,\frac{\beta}{\alpha})$
\end{tabular} & 
\begin{minipage}{5.2cm} 
MGF introduced in 2021 (in this article) for critical composition schemes.
\end{minipage}\\
\bottomrule
\end{tabular}
\caption{The successive generalizations of the Mittag-Leffler distribution, its moment generating function (MGF),
its connection to the corresponding generalized Mittag-Leffler function, and some occurrences in probability theory.}\label{tab:history}
\end{table}
\clearpage 

\begingroup
\renewcommand\arraystretch{1.81}
\begin{table}[b]
\setcellgapes{.55mm} \makegapedcells 
\begin{tabular}{@{}cccc@{}}
\toprule
\makecell{\textbf{Composition}\\\textbf{scheme}} & \textbf{Symbolic form} &\textbf{Limit law} & \textbf{Thm.}\\
\midrule
\addlinespace[-0.5mm]
Ordinary & $F(z,u)=G\big(u H(z)\big)$ & generalized Mittag-Leffler & \makecell{\ref{CoExtended1}}\\
Extended & $F(z,u)=M(z) G\big(u H(z)\big)$ & 
\makecell{
three-parameter Mittag-Leffler\\
and Boltzmann distribution}
& \makecell{\ref{TheExtended}\\ \ref{TheExtendedDegenerate} \\ \ref{TheExtendedConfluent}}
 \\
Cyclic & $F(z,u)=-\log\Big(1-u H(z)\Big)$ & Mittag-Leffler & \ref{COMPSCHEMECycleThe1}\\ 
\makecell{Multivariate \\ extended } & $F(z,\mathbf{u})=M(z)\prod\limits_{\ell=1}^{m}G_{\ell}\big(u_{\ell}H_{\ell}(z)\big)$ & \makecell{multivariate \\ product distribution} &\ref{TheMV}\\
Size-refined & $F(z,v)=M(z) G\big(H(z)-z^j h_j(1-v)\big)$ &\makecell{mixed Poisson type\\ phase transition} & \ref{TheRefined}\\ 
\makecell{Size-refined\\cyclic} & $F(z,v)=-\log\Big(1 - \big(H(z) - (1-v) h_j z^j / j!\big)\Big)$ &\makecell{mixed Poisson type\\ phase transition} &\ref{COMPSCHEMECycleThe2}\\ 
\makecell{Multivariate \\ size-refined} & $F(z,\mathbf{v})\!=\!M(z)\!\prod\limits_{\ell=1}^{m}G_{\ell}\big(H_{\ell}(z)-z^{j_\ell} h_{\ell,j_\ell}(1\!-\!v_\ell)\big)$ & \makecell{mv.\ mixed Poisson type\\ phase transition} & \ref{TheMVRefined}\\
\bottomrule 
\end{tabular}
\caption{Overview of our results on critical composition schemes, where $\mathbf{u} = (u_1,\dots, u_m)$ and $\mathbf{v} = (v_1,\dots, v_m)$.}
\label{TableO1}
\end{table}
\endgroup

The results of Table~\ref{TableO1} hold for functions $H(z)$ having a dominant singularity of algebraic type. 
Our methods can also deal with schemes having other types of singularities. 
For example, one could allow algebraic-logarithmic behaviours of the form
\begin{equation*}
H(z)\sim \left(1-\frac{z}{\rho_H}\right)^{\lambda_H} \bigg(\frac1z \log\left(1-\frac{z}{\rho_H}\right)\bigg)^{\psi_H}.
\end{equation*}

Such algebraic-logarithmic schemes cover some instances of 3-colour balanced triangular urn models~\cite{Puy2005,FlaDumPuy2006},
whose complete analysis, however, remains a challenge as other types of singularities appear, thus leading to new families of limit laws whose nature is unclear. 
Compare also with the related open problem by Janson~\cite{Jan2010}, asking for a more detailed description of the limit law of unbalanced 2-colour triangular urn models.

To keep this article readable, we eluded the question of the speed of convergence of~$X_n$, properly normalized, to its limit law $X$. 
Recently, a few articles addressed this 
for some preferential attachment rules~\cite{PekoezRoellinRoss2016} 
and for the 
Chinese restaurant process~\cite{DoleraFavaro2020} (see Remark~\ref{remark:almostsure}).
In a future work we plan to give uniform bounds on the moments, leading to Berry--Esseen-like inequalities for the speed of convergence of all the limit laws associated with critical composition schemes; see also~\cite{Hwang2003,AdellCal1993}.

Building on the current work, 
we consider in~\cite{BanderierKubaWagnerWallner2024}
an extension of our results on critical schemes by imposing a Gibbs measure on the parameter of the combinatorial structure:
that is, instead of the uniform distribution model imposed by Formula~\eqref{PXnk}, we consider the model 
$\P\{X_n=k\}= \frac{f_{n,k}q^k}{f_n(q)}.$
This leads to universal phase transitions where the three-parameter Mittag-Leffler distribution strikes again. 
(For $q=1$, this gives back the uniform distribution model.)

Last but not least, in our companion article~\cite{BanderierKubaWallner2021b}, we further extend our analysis to schemes with algebraic singularities for $\lambda_H>1$: 
We analyse a generalization of the composition scheme from~\cite{BFSS2001}, leading to stable laws (e.g., the map-Airy distribution), Gaussian laws, as well as bimodal distributions. 
For extended composition schemes we observe a new behaviour, obtaining for example beta limit laws. 
In size-refined schemes we anticipate further continuous-to-discrete phase transitions.
This will achieve the exhaustive exploration of the landscape of critical composition schemes with algebraic singularities  and their phase transitions.

\begin{acks}[Acknowledgements] We thank the two referees for their careful reading and their feedback.
We also thank Katarzyna G\'orska for pointing out that the moment generating function 
of our three-parameter Mittag-Leffler distribution is called the Prabhakar function in physics.\footnote{A recording 
of this surprising feedback is even available in our \href{https://www.youtube.com/watch?v=7kqrIys-NxM}{IH\'ES talk},
at the conference \textit{Combinatorics and Arithmetic for Physics: Special Days 2022}.
We thank Gérard Duchamp, Maxim Kontsevich, Gleb Koshevoy, Sergei Nechaev, and Karol Penson,
who had the great idea to organize this interdisciplinary event!}

Last but not least, we are pleased to dedicate this article to our colleague and mentor Alois Panholzer, 
on the occasion of his 50th birthday. 
In fact, Alois played a central r\^ole in the birth of this article in several ways:
\begin{itemize} 
	\item As a byproduct of Alois' analyses of combinatorial structures and stochastic processes (such as permutations, parking functions, lattice paths, urns, and trees),
	several phase transitions involving the mixed Poisson distribution were uncovered~\cite{PanholzerSeitz2011,KuPa2014}. This guided the second author to study a general framework for these phase transitions, and to join forces with the other authors to extend the results of~\cite{BFSS2001,Wallner2020} involving
	Rayleigh, half-normal, and stable distributions. 
	This led us to the composition schemes analysed in this work.
	\item Alois also connected the first and second
	authors via a French-Austrian international project (PHC Amadeus) in 2005/06, 
	perpetuating a long tradition of exchanges in combinatorics between Austria and France. 
	A part of this project led to the article~\cite{BanderierKubaPanholzer2009}, 
	which established Gaussian limit laws for some composition schemes related to trees. 
	\item Alois was also a constant source of inspiration and guidance for the second author. 
	Without him, we would have never started this work.
\end{itemize}
\end{acks}

{\bf Funding.} \ \ \ 
Michael Wallner was funded by the FWF Erwin Schr\"odinger Fellowship J~4162 and the FWF Stand-alone Project P~34142.
The authors also thank the  French-Austrian PHC Amadeus project ``Asymptotic behaviour of combinatorial structures'' (OeAD WTZ project FR 01/2023),
which funded mutual visits.

\bgroup
\let\oldbibliography\thebibliography
\renewcommand{\thebibliography}[1]{\oldbibliography{#1}\setlength{\itemsep}{1pt}}
 \interlinepenalty=10000
{\linespread{1}
\setlength{\itemsep}{5mm}
\bibliographystyle{cyrbiburl}
\bibliography{bibliography_short} 
}
\egroup
\end{document}